\documentclass{article}

\usepackage[english]{babel}

\usepackage[letterpaper,top=4cm,bottom=4cm,left=3cm,right=3cm,marginparwidth=1.75cm]{geometry}

\usepackage{amsthm}
\usepackage{amsmath}
\usepackage{amssymb}

\usepackage{mathtools}
\usepackage{thmtools}

\usepackage{enumitem}

\usepackage{graphicx}
\usepackage[colorlinks=true, allcolors=blue]{hyperref}
\usepackage{tikz}
\usepackage{tikz-cd}

\usepackage{subcaption}

\usepackage{titlesec}

\titleformat{\paragraph}[block]{\normalfont\bfseries}{\theparagraph}{1em}{}
\titlespacing*{\paragraph}{0pt}{3ex plus 1ex minus .2ex}{1ex plus .2ex}
\newtheorem{theorem}{Theorem}[section]
\newtheorem{lemma}[theorem]{Lemma}
\newtheorem{corollary}[theorem]{Corollary}
\newtheorem{definition}[theorem]{Definition}
\newtheorem{proposition}[theorem]{Proposition}
\newtheorem{conjecture}[theorem]{Conjecture}

\newtheorem*{notation*}{Notation}

\tikzstyle{none}=[]
\pgfdeclarelayer{edgelayer}
\pgfdeclarelayer{nodelayer}
\pgfsetlayers{edgelayer,nodelayer,main}

\tikzstyle{new style 0}=[fill=white, draw=black, shape=circle]
\tikzstyle{new style 1}=[fill=red, draw=red, shape=circle, minimum size=2pt, inner sep=0pt]
\tikzstyle{new style 2}=[fill=black, draw=black, shape=circle, minimum size=2pt, inner sep=0pt]
\tikzstyle{new style 3}=[fill=blue, draw=blue, shape=circle, minimum size=2pt, inner sep=0pt]

\tikzstyle{new edge style 0}=[-, dashed]
\tikzstyle{new edge style 1}=[-, fill=none, draw=red]
\tikzstyle{new edge style 2}=[-, fill=none, draw=blue]
\tikzstyle{new edge style 3}=[-, dashed, fill=none, draw=blue]
\tikzstyle{new edge style 4}=[-, dashed, fill=none, draw=red]
\tikzstyle{new edge style 5}=[->]
\tikzstyle{new edge style 6}=[->, dashed]
\tikzstyle{new edge style 7}=[->, dashed, draw=red]
\tikzstyle{new edge style 8}=[->, draw=red]
\tikzstyle{new edge style 9}=[<-]

\title{Inscriptions of Isosceles Trapezoids in Jordan Curves}
\author{Adam Barber}

\begin{document}
\maketitle

\begin{abstract}
We construct a Lagrangian Floer homology whose chain complex is generically generated by the inscriptions of isosceles trapezoids in a smooth Jordan curve. This is an extension of Greene and Lobb's Jordan Floer homology \cite{Greene-Lobb:Floer-homology}, which we also call Jordan Floer homology. Its non-triviality re-establishes that every smooth Jordan curve inscribes every isosceles trapezoid. By consideration of the spectral invariants associated with the real filtration known as the action filtration, we establish new cases of non-smooth Jordan curves which admit inscriptions of isosceles trapezoids.
\end{abstract}

\section{Introduction}

We say that a Jordan curve, a simple closed curve $\gamma \subset \mathbb{R}^2$, \emph{inscribes} a square if there exist 4 points in $\gamma \subset \mathbb{R}^2$ that are the vertices of a square. The classical example of an inscription problem is the Square Peg Problem, which was first posed by Toeplitz \cite{Toeplitz:1911} in 1911 and asks whether every Jordan curve inscribes a square. 

\begin{conjecture}[Square Peg Problem]
    Every Jordan curve inscribes a square.
\end{conjecture}

\noindent This remains an open problem. More generally, we say that a Jordan curve $\gamma$ \emph{inscribes} a polygon $P$ if there exist 4 points in $\gamma$ that are the vertices of a polygon similar to $P$ (we say that two polygons are similar if there exists an orientation preserving similarity between them). We consider inscriptions of isosceles trapezoids (quadrilaterals whose vertices are the intersection set of two parallel lines and a circle) in Jordan curves. As a generalisation of the Square Peg Problem, it is also an open question whether every Jordan curve inscribes every isosceles trapezoid.

\begin{conjecture}[Isosceles Trapezoid Peg Problem]
    Every Jordan curve inscribes every isosceles trapezoid.
\end{conjecture}

In \cite{Greene-Lobb:Floer-homology}, Greene and Lobb developed a Lagrangian Floer homology associated with inscriptions of rectangles in smooth Jordan curves, called Jordan Floer homology. The goal of this paper is to define and study Jordan Floer homology in the setting of inscriptions of isosceles trapezoids in smooth Jordan curves. The non-triviality of this homology for smooth curves provides an alternative proof of the already established result:

\begin{theorem}[Greene-Lobb, 2020 \cite{Greene-Lobb:Cyclic-quadrilaterals}] \label{thm:smooth_isosceles_trpezoids}
    Every smooth Jordan curve inscribes every isosceles trapezoid.
\end{theorem}
Their result was in fact somewhat stronger and established that every smooth Jordan curve inscribes every cyclic quadrilateral (quadrilaterals that inscribe in a circle).

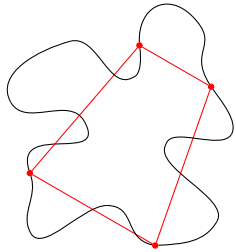
\begin{figure}[h]
    \centering

\begin{tikzpicture}[scale = 0.55, rotate=-30]
	\begin{pgfonlayer}{nodelayer}
		\node [style=new style 1] (1) at (1.75, -2) {};
		\node [style=new style 1] (2) at (-1.75, -2) {};
		\node [style=none] (3) at (-1.25, -0.25) {};
		\node [style=new style 1] (5) at (-1, 2) {};
		\node [style=new style 1] (6) at (1, 2) {};
		\node [style=none] (9) at (-1, -3.25) {};
		\node [style=none] (54) at (0.25, -1.75) {};
		\node [style=none] (55) at (2.75, -0.75) {};
		\node [style=none] (57) at (0.75, 0.25) {};
		\node [style=none] (58) at (2, 1.25) {};
		\node [style=none] (59) at (-0.25, 3.25) {};
		\node [style=none] (61) at (-0.75, 1.25) {};
		\node [style=none] (62) at (-2.5, 1.25) {};
		\node [style=none] (63) at (-3, -0.75) {};
	\end{pgfonlayer}
	\begin{pgfonlayer}{edgelayer}
		\draw [in=-30, out=135, looseness=1.75] (2) to (3.center);
		\draw [style=new edge style 1] (5) to (2);
		\draw [style=new edge style 1] (2) to (1);
		\draw [style=new edge style 1] (1) to (6);
		\draw [style=new edge style 1] (5) to (6);
		\draw [in=150, out=-45] (2) to (9.center);
		\draw [in=-150, out=-30, looseness=0.75] (9.center) to (54.center);
		\draw [in=-165, out=30, looseness=1.25] (54.center) to (1);
		\draw [in=-90, out=30] (1) to (55.center);
		\draw [in=-90, out=105, looseness=0.75] (57.center) to (58.center);
		\draw [in=330, out=90, looseness=0.75] (58.center) to (6);
		\draw [in=90, out=-90, looseness=0.75] (57.center) to (55.center);
		\draw [in=-15, out=150, looseness=1.50] (6) to (59.center);
		\draw [in=45, out=-60, looseness=0.75] (5) to (61.center);
		\draw [in=45, out=-120] (61.center) to (62.center);
		\draw [in=120, out=165, looseness=1.50] (59.center) to (5);
		\draw [in=150, out=-135] (62.center) to (63.center);
		\draw [in=150, out=-30] (63.center) to (3.center);
	\end{pgfonlayer}
\end{tikzpicture}

    \caption{An example of an inscription of an isosceles trapezoid in a Jordan curve.}
    \label{fig:inscription}
    \vspace{-0.5em}
\end{figure}

An advantage of the Jordan Floer approach over the approach used in \cite{Greene-Lobb:Cyclic-quadrilaterals} to show that every smooth Jordan curve inscribes every isosceles trapezoid is the existence of spectral invariants, which are related to the areas of inscribed isosceles trapezoids. One way to consider inscriptions of an isosceles trapezoid $T$ in a non-smooth Jordan curve $\gamma$ is to consider inscriptions of $T$ in smooth Jordan curves $\gamma_n$ limiting to $\gamma$ and use this to conclude that $\gamma$ must inscribe $T$. However, this approach suffers from the issue of \emph{shrinkout}, which we will describe precisely in Section \ref{sec:shrinkout}. 

In the case of piecewise $C^1$ Jordan curves without cusps, the issue of shrinkout can be obstructed by elementary techniques. This was observed by Matschke \cite[Section 5.3]{Matschke:2021} and when combined with Theorem \ref{thm:smooth_isosceles_trpezoids} establishes the following result.
\begin{theorem}[Matschke, 2021 \cite{Matschke:2021}]
    Every piecewise $C^1$ Jordan curve without cusps inscribes every isosceles trapezoid.
\end{theorem}
\noindent For broader classes of Jordan curves, such as rectifiable (finite length) or locally monotone (see Definition \ref{def:locally_monotone}) curves, this approach breaks down. In such cases, the spectral invariants provide an alternative approach to obstruct shrinkout, enabling us to deduce the existence of inscribed isosceles trapezoids for certain new classes of non-smooth Jordan curves.

The approach of using spectral invariants to obstruct shrinkout was introduced in \cite{Greene-Lobb:Floer-homology} in the case of rectangles. Analysis of the spectral invariants in the case of non-rectangular isosceles trapezoids is made harder due to less control over the limiting behaviour of the spectral invariants. To illuminate this issue, consider a sequence of inscribed isosceles trapezoids in a smooth Jordan curve $\gamma$ such that the angle between the diagonals of the isosceles trapezoid tends to $\pi$. Limiting to the case where all the inscribed isosceles trapezoids are rectangles, the limit is necessarily a \emph{binormal} (a line segment normal to $\gamma$ at both ends). However, without this restriction, the limit could also be a \emph{vertex} (a critical point of curvature of $\gamma$) or a \emph{quadrisecant} (a line segment that intersects $\gamma$ in at least 4 points). The potential existence of quadrisecants in this limit (see Figure \ref{fig:quadrasecant} for example) is what makes the isosceles trapezoid case harder.

An example of a class of curves for which we can control this limit is locally Lipschitz-graphical Jordan curves (see Definition \ref{def:graphically_Lipschitz}), which form a subset of rectifiable locally monotone Jordan curves and contains all piecewise $C^1$ Jordan curves without cusps. Isosceles trapezoids can be parameterised by two parameters: the angle between the diagonal $\theta$, and an aspect ratio $r$ characterising where the diagonals intersect (see Section \ref{sec:isosceles_trapezoids}). We now provide a sample result for this class of curves.
\begin{restatable}{mainthm}{theoremA} \label{thm:first_theorem}
    Let $\gamma$ be a locally Lipschitz-graphical Jordan curve bounding an area $\operatorname{Area}(\gamma)$ and of radius $\operatorname{Rad}(\gamma)$, then $\gamma$ inscribes an isosceles trapezoid of angle $\theta$ and aspect ratio $r$ for every
    \begin{equation}
        0 < \theta \ < \frac{\operatorname{Area}(\gamma)}{2(1-r) \operatorname{Rad}(\gamma)^2}. \nonumber
    \end{equation}
\end{restatable}

The final result of note is that of Asano and Ike \cite{Asano-Ike:2026}, who recently established that every rectifiable or locally monotone Jordan curve inscribes every rectangle. Their work was inspired by the techniques of Greene and Lobb \cite{Greene-Lobb:Floer-homology}, but used microlocal sheaf theory in place of Lagrangian Floer homology.
\begin{theorem}[Asano-Ike, 2026 \cite{Asano-Ike:2026}]
    Let $\gamma$ be a rectifiable or locally monotone Jordan curve, then $\gamma$ inscribes every rectangle.
\end{theorem}
\noindent However, their approach relies on the associated Lagrangians being monotone, a property that fails to hold in the setting of isosceles trapezoids. It is therefore not clear how their methods could be adapted to address the problem of inscribing isosceles trapezoids.

\subsection{Why isosceles trapezoids?}

To motivate our interest in inscriptions of isosceles trapezoids, as opposed to any other class of quadrilaterals, we consider the following result for cyclic quadrilaterals. 

\begin{theorem}[Greene-Lobb, 2020 \cite{Greene-Lobb:Cyclic-quadrilaterals}] \label{thm:cyclic-inscriptions}
    Every smooth Jordan curve inscribes every cyclic quadrilateral.
\end{theorem}

\noindent This theorem is optimal in two ways: any quadrilateral that is not cyclic does not inscribe in a circle (which is a smooth Jordan curve); and for any cyclic quadrilateral that is not an isosceles trapezoid there exists a non-smooth Jordan curve (in fact a triangle) in which it does not inscribe. This both motivates our interest in isosceles trapezoids, as the widest class of quadrilaterals for which the problem of inscription on general Jordan curves remains open, and demonstrates an obstruction to progress on the square peg problem: the issue of \emph{shrinkout}.

\subsubsection{The issue of shrinkout} \label{sec:shrinkout}

Suppose that $\gamma$ is a continuous Jordan curve in $\mathbb{C}$. The Riemann mapping theorem tells us that there exists a biholomorphic map $f: \mathbb{D} \rightarrow D$, where $\mathbb{D} = \{z \in \mathbb{C}:|z| < 1 \}$ and $D$ is the bounded region of $\mathbb{C} \setminus \gamma$. Carathéodory’s extension to the Riemann mapping theorem tells us that this map continuously extends to a homeomorphism $f:\overline{\mathbb{D}} \rightarrow \overline{D}$. The Jordan curve $\gamma$ is therefore the limit of a sequence of smooth Jordan curves $\gamma_n$, the images of concentric circles in $\mathbb{D}$ approaching the unit circle $\partial \mathbb{D}$ from the interior.

Fix a cyclic quadrilateral $Q$. From Theorem \ref{thm:cyclic-inscriptions}, we know that each smooth Jordan curve $\gamma_n$ inscribes any given cyclic quadrilateral, so we get a sequence of inscriptions $Q_n \subset \gamma_n$ where $Q_n$ is an inscription of $Q$ in $\gamma_n$. Compactness ensures that there exists a convergent subsequence $Q_{n_i}$ of inscriptions of $Q$. It is therefore tempting to conclude that $\gamma$ must inscribe $Q$. However, in the limit the inscription of $Q$ may shrink to a single point. From the discussion under Theorem \ref{thm:cyclic-inscriptions}, we know that this is necessarily the case when $Q$ is a cyclic quadrilateral that is not an isosceles trapezoid and the limiting curve $\gamma$ is a triangle that does not inscribe $Q$. We call this phenomenon \emph{shrinkout}.

To approach the issue of shrinkout, we use the same approach as \cite{Greene-Lobb:Floer-homology, Greene-Lobb:Squares-between-graphs}, that is, we use a version of Lagrangian Floer homology called Jordan Floer homology, which has an associated spectral invariant. We will show that this spectral invariant relates to an area associated with the inscribed isosceles trapezoids, so controlling the variation of this spectral invariant can be used in special cases to obstruct shrinkout.

\subsection{Inscriptions of isosceles trapezoids} \label{sec:isosceles_trapezoids}

Consider a Jordan curve $\gamma \subset \mathbb{R}^2 = \mathbb{C}$. We wish to determine whether a given isosceles trapezoid $T$ inscribes in $\gamma$. The set of isosceles trapezoids (up to similarity) can be parametrised by an aspect ratio $r \in (0,1/2]$ and an angle $\theta \in (0, \pi)$ (see Figure \ref{fig:cyclic_map}) and we denote the similarity class of isosceles trapezoids with parameters $r$ and $\theta$ by $T_{r, \theta}$. More precisely, a set $T \subset \mathbb{C}$ is a member of $T_{r,\theta}$ if $|T|=4$ and it contains the vertices of an isosceles trapezoid with aspect ratio $r$ and angle $\theta$. 

The parameters $r$ and $\theta$ characterise the way in which the diagonals of the isosceles trapezoids intersect, namely $\theta$ is the angle through which one must rotate one diagonal around the intersection point to obtain the other and $r$ represents the ratio of the lengths into which the intersection point splits each of the diagonals, as shown in Figure \ref{fig:cyclic_map}. For any similarity class $T_{r, \theta}$ and any pair of points $z,w \in \mathbb{C}$ with $z \neq w$, we define two new points $z', w' \in \mathbb{C}$ by rotating $z$ and $w$ clockwise around the point $(1-r)z + rw$ on the diagonal by an angle $\theta$. The set $T = \{z,z',w,w' \}$ therefore defines an isosceles trapezoid $T \in T_{r, \theta}$.   

\begin{notation*}
    We write $\Delta(\mathbb{C}) = \{(z,z):z \in \mathbb{C}\}$ and $\Delta(\gamma) = \{(z,z):z \in \gamma\}$ to denote the diagonals of the spaces $\mathbb{C}^2$ and $\gamma \times \gamma$ respectively. 
\end{notation*}

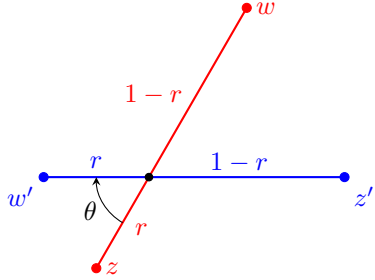
\begin{figure}[h]
    \centering

    \begin{tikzpicture}[scale=2, >=stealth]

    \coordinate (A) at (-1,0);
    \coordinate (B) at ( 1,0);
    \coordinate (O) at (-0.3,0);
  
    \coordinate (fA) at (0.35, 1.126);
    \coordinate (fB) at (-0.65,-0.606);
  
    \draw[thick, blue] (A) -- (B);
    \filldraw[blue] (A) circle(0.03) node[below left] {$w'$};
    \filldraw[blue] (B) circle(0.03) node[below right] {$z'$};

    \draw[thick, red] (fA) -- (fB); 
    \filldraw[red]  (fA) circle(0.03) node[right]  {$w$};
    \filldraw[red]  (fB) circle(0.03) node[right] {$z$};

    \filldraw[black] (O) circle(0.025);

    \draw[<-, thin] (-0.65,0) arc[start angle=180,end angle=240,radius=0.35];
    \node at (-0.69,-0.225) {$\theta$};

    \node[blue] at (0.3,0.1) {$1-r$};
    \node[blue] at (-0.65,0.1) {$r$};

    \node[red] at (-0.27,0.55) {$1-r$};
    \node[red] at (-0.35,-0.35) {$r$};

    \end{tikzpicture}

    \caption{A geometric picture of the parametrisation of isosceles trapezoids by the parameters $\theta$ and $r$ and the map $G_{r,\theta}:\mathbb{C}^2 \rightarrow \mathbb{C}^2$.}
    \label{fig:cyclic_map}
\end{figure}

The construction above defines a transformation $G_{r,\theta}$ on any pair of points $(z,w) \in \mathbb{C}^2 \setminus \Delta(\mathbb{C})$ by $G_{r,\theta}(z,w) = (z',w')$. Finding inscriptions of an isosceles trapezoid $T$ then reduces to looking for pairs of points $(z,w) \in (\gamma \times \gamma) \setminus \Delta(\gamma)$ (representing the endpoints of a diagonal) such that $G_{r,\theta}(z,w) \in \gamma \times \gamma$. There is a unique continuous extension of the map $G_{r,\theta}$ to $\mathbb{C}^2$, which can be given explicitly as follows:
\begin{proposition} \label{prop:rotation_map}
    The map $G_{r,\theta}:\mathbb{C}^2 \rightarrow \mathbb{C}^2$ is given by $G_{r,\theta} = {F_r}^{-1} \circ R_\theta \circ F_r$, where $F_r, R_\theta: \mathbb{C}^2 \rightarrow \mathbb{C}^2$ are the linear maps
    \begin{equation}
        F_r = \begin{pmatrix} 1-r  & r \\ \sqrt{r(1-r)} & -\sqrt{r(1-r)} \end{pmatrix} \quad \text{and} \quad R_\theta = \begin{pmatrix} 1  & 0 \\ 0 & e^{-i\theta} \end{pmatrix} \nonumber
    \end{equation}
    for $r \in (0, 1/2]$ and $\theta \in (0, \pi)$.
\end{proposition}
\noindent We note here that we have taken $R_\theta$ as the clockwise rotation in the second coordinate. One could just as easily use the anticlockwise rotation, but this would lead to requiring a negative Hamiltonian to generate the transformation.

Proposition \ref{prop:rotation_map} follows immediately from the following lemma.

\begin{lemma}
    Let $z,z',w,w' \in \mathbb{C}$, then if $G_{r, \theta}(z,w) = (z', w')$ and $z \neq w$, the line segments $zw$ and $z'w'$ are the diagonals of an isosceles trapezoid with aspect ratio $s$ and angle $\theta$ .
\end{lemma}

\begin{proof}
    Applying the map $F_r$ (a diffeomorphism) to both sides of the equation, we obtain $R_\theta \circ F_r(z,w) = F_r(z',w')$. Equivalence in the first coordinate ensures that the line segments $zw$ and $z'w'$ meet at the point $p = (1-r)z + rw = (1-r)z' + rw'$, ensuring that $p$ divides both the line segments $zw$ and $z'w'$ in the ratio $r:1-r$. Equivalence in the second coordinate gives $z'-w' = e^{-i\theta}(z-w)$, so $zw$ and $z'w'$ meet at an angle $\theta$ and are of equal length. This verifies the claim.
\end{proof}

Finding inscribed isosceles trapezoids in $\gamma$ is then equivalent to finding points of intersection between the tori $\gamma \times \gamma$ and $G_{r,\theta}(\gamma \times \gamma)$ away from the \emph{trivial} intersection $\Delta(\gamma)$ (agreeing with the approach used in \cite{Greene-Lobb:Floer-homology} for inscriptions of rectangles). The tori intersect cleanly along the loop of trivial intersections, which correspond to \emph{degenerate} inscriptions where all four vertices coincide. Our initial aim is to show that non-trivial intersection points exist, for which we use Lagrangian Floer homology.

\subsection{Jordan Floer homology}

The approach to construct Jordan Floer homology associated with inscriptions of isosceles trapezoids will be broadly the same as that used by Greene and Lobb for the case of rectangles in \cite{Greene-Lobb:Floer-homology}, now with an extra parameter $r$. Together, the aspect ratio $r$ and angle $\theta$ parametrise the space of isosceles trapezoids up to similarity (as described above) with the case $r=1/2$ representing rectangles. We now provide an outline of the approach, omitting technical details.

Given a Jordan curve $\gamma \subset \mathbb{C}$, we consider the torus $\gamma \times \gamma \subset \mathbb{C}^2$. This torus is Lagrangian in $\mathbb{C}^2$ with respect to the symplectic form $\omega = (1-r)  \cdot dx_1 \wedge dy_1 + r \cdot dx_2 \wedge dy_2$. We then define a Hamiltonian $H$ such that the time-1 flow of the associated Hamiltonian vector field $X_H$ is the transformation $G_{r,\theta}$. This gives a second Lagrangian torus $G_{r,\theta}(\gamma \times \gamma)$. Generically, these Lagrangian tori intersect cleanly along the loop $\Delta(\gamma)$ and transversally at points of intersection away from this loop, which we call non-trivial intersection points and represent inscribed isosceles trapezoids. To each of these intersection points $p$, there is an associated \emph{trajectory} $\tau:[0,1] \rightarrow \mathbb{C}^2$ of the Hamiltonian vector field $X_H$ with $\tau(0) \in \gamma \times \gamma$ and $\tau(1)=p$. Non-trivial intersection points are in one-to-one correspondence with non-constant trajectories.

We then construct a chain complex, the Jordan Floer chain complex $\operatorname{JFC}(\gamma, r, \theta)$, generically generated by the non-constant trajectories of $H$. The differential $\partial$ on this chain complex is defined as a count of certain strips between trajectories, called \emph{pseudo-holomorphic strips}. We only consider pseudo-holomorphic strips that avoid the diagonal $\Delta(\mathbb{C}) = \{(z,z) \mid z \in \mathbb{C} \}$. Showing that this differential satisfies $\partial^2 = 0$ is one of the key technical steps in defining Jordan Floer homology, as it requires proving that the limits of diagonal-avoiding pseudo-holomorphic strips stay away from the diagonal. Fortunately, some of the arguments for establishing Jordan Floer homology for inscriptions of rectangles made in \cite{Greene-Lobb:Floer-homology} extend directly to the isosceles trapezoid case. Taking the homology of the Jordan Floer chain complex gives the Jordan Floer homology $\operatorname{JF}(\gamma, r, \theta)$. Since generators of the chain complex correspond to inscriptions of isosceles trapezoids of aspect ratio $r$ and angle $\theta$, showing that $\operatorname{JF}(\gamma, r, \theta)$ is non-trivial is sufficient to show that $\gamma$ inscribes such an isosceles trapezoid.

The Jordan Floer chain complex has a real filtration called the action, which lifts to a filtration grading for each non-zero homology class $\alpha \in \operatorname{JF}(\gamma, r, \theta)$ called the spectral invariant of $\alpha$. There is a single non-zero homology class in $\operatorname{JF}_2(\gamma, r, \theta)$, whose associated spectral invariant we denote by $l_2(\gamma, r, \theta)$. This spectral invariant can be used to obstruct \emph{shrinkout} when considering non-smooth Jordan curve limits (see Section \ref{sec:shrinkout}).
 
\subsection{The main result}

The main goal of this paper is to construct and examine Jordan Floer homology for inscriptions of isosceles trapezoids in Jordan curves. Two classes of curves of interest are rectifiable Jordan curves, which are those of finite length, and locally monotone Jordan curves, which we now define.
\begin{definition} \label{def:locally_monotone}
    Let $\gamma$ be a Jordan curve and $\gamma:S^1 \cong \mathbb{R}/ 2\pi\mathbb{Z} \rightarrow \mathbb{C} \cong \mathbb{R}^2$ be a parametrisation $\gamma$, then $\gamma$ is called \emph{locally monotone} if for every $\theta \in \mathbb{R}$ there exists an open neighbourhood $U_\theta \subset \mathbb{R}$ of $\theta$ and a unit vector $v_\theta \in \mathbb{R}^2$ such that the map $g_{v_\theta}:\varphi \mapsto \gamma(\varphi) \cdot v_\theta$ is strictly monotone on $U_\theta$. 
\end{definition}

\noindent Locally monotone Jordan curves can be locally represented as the graph of a function. To be more precise, writing $f_{v_\theta}: \varphi \mapsto \gamma(\varphi) \cdot n_\theta$, where $n_\theta$ is a unit vector perpendicular to $v_\theta$, the map $f_{v_\theta} \circ g_{v_\theta}^{-1}: g_{v_\theta}(U_\theta) \rightarrow \mathbb{R}$ is a continuous function for every $\theta \in \mathbb{R}$.

\begin{restatable}{mainthm}{theoremB} \label{thm:intro_main_theorem}
    Let $\gamma$ be a rectifiable or locally monotone Jordan curve, $r \in (0,1/2]$ be an aspect ratio, and assume $\lim_{\theta \rightarrow \pi} l_2(\gamma,r, \theta) \neq 0$, then $\gamma$ inscribes an isosceles trapezoid of aspect ratio $r$ and angle $\theta$ for every
    \begin{equation}
        0 < \theta < \frac{\operatorname{Area}(\gamma)}{2(1-r) \cdot \operatorname{Rad} (\gamma)^2}. \nonumber
    \end{equation}
\end{restatable}
\noindent  A class of curves for which we can easily conclude that $\lim_{\theta \rightarrow \pi} l_2(\gamma,r, \theta) \neq 0$ is locally Lipschitz-graphical Jordan curves, giving Theorem \ref{thm:first_theorem} as a direct corollary .

\begin{definition} \label{def:graphically_Lipschitz}
    Let $\gamma$ be a locally monotone Jordan curve, then $\gamma$ is called \emph{locally $K$-Lipschitz-graphical} if for every $\theta \in \mathbb{R}$ there exists an open neighbourhood $U_\theta \subset \mathbb{R}$ of $\theta$ and a unit vector $v_\theta \in \mathbb{R}^2$ such that $f_{v_\theta} \circ g_{v_\theta}^{-1}$ is a $K$-Lipschitz function on $g_{v_\theta}(U_\theta)$. If $\gamma$ is locally $K$-Lipschitz graphical for some constant $K>0$ then we say that $\gamma$ is \emph{locally Lipschitz-graphical}.
\end{definition}

\section{Jordan Floer homology}

The goal of this section is to construct the Jordan Floer homology for the following data: a smooth Jordan curve $\gamma \subset \mathbb{C}$, an aspect ratio $r \in (0,1/2]$, an angle $\theta \in (0, \pi)$, a Hamiltonian perturbation $h_t: \mathbb{C}^2 \rightarrow \mathbb{C}^2$, and an almost complex structure $J_t:T\mathbb{C}^2 \rightarrow T\mathbb{C}^2$. This will closely follow the construction in \cite{Greene-Lobb:Floer-homology}, which constructs the Jordan Floer homology associated with inscriptions of rectangles in smooth Jordan curves.

The first step is to define the Jordan Floer complex $\operatorname{JFC}(\gamma, r, \theta, h_t)$ and its differential $\partial_{(\gamma, r, \theta, h_t, J_t)}$. The first main result will be verifying that this map squares to zero. This requires verifying that limits of sequences of pseudo-holomorphic strips that avoid the diagonal behave nicely, that is, they do not admit bubbling and any breaking occurs away from the diagonal. This is the main technical step in defining the Jordan Floer homology for the given data.

We then consider chain maps
\begin{equation}
    \operatorname{JFC}(\gamma, r, \theta_1, h_t^1, J_t^1) \longrightarrow \operatorname{JFC}(\gamma, r, \theta_2, h_t^2, J_t^2), \nonumber
\end{equation} known as continuation maps. They provide isomorphisms on the associated homology groups, demonstrating that the Jordan Floer homology groups $\operatorname{JF}(\gamma, s, \theta, h_t, J_t)$ are independent of the Hamiltonian perturbation $h_t$ and the almost-complex structure $J_t$. The homology groups can therefore be written simply as $\operatorname{JF}(\gamma, s, \theta)$ and we have the isomorphisms induced from the continuation maps
\begin{equation}
    \operatorname{JF}(\gamma, r, \theta_1) \longrightarrow \operatorname{JF}(\gamma, r, \theta_2). \nonumber
\end{equation}
Careful treatment of the limit as $\theta \rightarrow 0$ then allows us to determine the Jordan Floer homology for any aspect ratio $r$ and angle $\theta$ up to isomorphism as the Morse homology of a torus $\gamma \times \gamma$ relative to its diagonal curve $\Delta(\gamma)$.

\subsection{Jordan Floer setup}

To define the Jordan Floer homology associated with inscribed isosceles trapezoids, we require three key ingredients: a symplectic manifold $(\mathbb{C}^2, \omega)$, a Lagrangian $\gamma \times \gamma \subset \mathbb{C}^2$, and a Hamiltonian $H:\mathbb{C}^2 \rightarrow \mathbb{R}$. We now fix an aspect ratio $r \in (0,1/2]$ and an angle $\theta \in (0,\pi)$. To ensure that the Hamiltonian generates the transformation $G_{r, \theta}$ defined in Section 1.2, we require appropriate choices for the symplectic form $\omega$ and the Hamiltonian $H$. 

\begin{definition}[Symplectic form]
    We define the \emph{symplectic form} on $\mathbb{C}^2$ in coordinates $(z,w)=(x_1+iy_1,x_2+iy_2)$ by
    \begin{equation}
        \omega \coloneq F_s^* \omega_{\operatorname{std}} = (1-r) \cdot dx_1 \wedge dy_1 + r \cdot dx_2 \wedge dy_2, \nonumber
    \end{equation}
    where $\omega_{\operatorname{std}} = dx_1 \wedge dy_1 + dx_2 \wedge dy_2$ is the standard symplectic form on $\mathbb{C}^2$. 
\end{definition}

\begin{definition}[Hamiltonian]
    We define the \emph{Hamiltonian} $H:\mathbb{C}^2 \rightarrow \mathbb{R}$ for an aspect ratio $r \in (0,1/2]$ and an angle $\theta \in (0,\pi)$ by
    \begin{equation}
        H(z,w) = \frac{1}{2} \theta r(1-r)|z-w|^2. \nonumber
    \end{equation} 
\end{definition}

\begin{lemma}
    The Hamiltonian $H$ defined above on the symplectic manifold $(\mathbb{C}^2, \omega)$ generates the transformation $G_{r,\theta}$. That is, $\Phi_{H}^1 = G_{r,\theta}$, where $\Phi_{H}^t$ is the time-$t$ flow of $X_H$, the unique Hamiltonian vector field that satisfies $\iota_{X_H}\omega = dH$.
\end{lemma}

\begin{proof}
    Using the change of coordinates $(u,v) = F_r(z,w)$, then writing $u = X_1 + iY_1$ and $v = Re^{i\varphi}$, one has
    \begin{equation}
        \omega = dX_1 \wedge dY_1 + R \cdot dR \wedge d\varphi. \nonumber
    \end{equation}
    In these coordinates, the transformation $G_{r,\theta}$ is simply the rotation $R_\theta$, which is generated by the time-$1$ flow of the vector field $-\theta \frac{\partial}{\partial\varphi}$. Using the definition $i_X\omega \coloneq \omega(X, \cdot)$, we have
    \begin{eqnarray}
        \iota_{-\theta \frac{\partial}{\partial\varphi}} \omega &\coloneqq& \omega \left( -\theta \frac{\partial}{\partial\varphi}, \; \cdot \; \right) = \theta R \cdot dR = d \left( \frac{\theta R^2}{2} \right) \nonumber \\
        &=& d \left( \frac{\theta|v|^2}{2} \right) = d \left( \frac{1}{2}\theta r(1-r)|z-w|^2 \right) = dH. \nonumber
    \end{eqnarray}   
    This shows that the time-$1$ flow of the Hamiltonian vector field $X_H$ is the transformation $G_{r,\theta}$.
\end{proof}

\noindent With the symplectic form and Hamiltonian chosen, it is easy to verify that the tori representing pairs of points on the Jordan curve $\gamma$ and its rotation by $G_{r,\theta}$ are Lagrangian.

\begin{lemma}
    The submanifolds $\gamma \times \gamma$ and $G_{r,\theta}(\gamma \times \gamma)$ are Lagrangian in $(\mathbb{C}^2, \omega)$.
\end{lemma}

\begin{proof}
    The fact that $\gamma \times \gamma$ is Lagrangian follows from the fact that $\gamma$ is Lagrangian in $(\mathbb{C}, dx \wedge dy)$. The fact that $G_{r, \theta}(\gamma \times \gamma)$ is Lagrangian then follows immediately from the fact that $\gamma \times \gamma$ is Lagrangian along with the fact that $G_{r,\theta}$ is a symplectomorphism, since it is the flow of a Hamiltonian vector field.
\end{proof}

In order to ensure later that the Floer differential is well defined, we also define the time-dependent version of the Hamiltonian. This generates the same transformation $G_{r,\theta}$. 

\begin{definition}[Time dependent Hamiltonian]
    For a smooth function $\beta:[0,1] \rightarrow [0, \infty)$ satisfying $\int_0^1 \beta = 1$ and with support in $[0.1,0.9]$, we define the time dependent Hamiltonian $H_t:\mathbb{C}^2 \rightarrow \mathbb{R}$ by
    \begin{equation}
        H_t(z,w) = \beta(t) H(z,w) \nonumber
    \end{equation}
    for $0 \leq t \leq 1$ and $(z,w) \in \mathbb{C}^2$.
\end{definition}
\noindent The condition $\int_0^1 \beta = 1$ ensures that $H_t$ and $H$ both generate the same transformation $G_{r,\theta}$. The restricted support condition ensures $H_t=0$ for $t \in [0,0.1] \cup [0.9,1]$, which is an important technical detail later.

\subsection{Jordan Floer chain complex}

One issue we wish to avoid is that the Lagrangians may not intersect transversely away from the diagonal. To achieve transversality, we introduce a small Hamiltonian perturbation. In a small neighbourhood of the diagonal $\Delta(\gamma) = \{ (z,z) \mid z \in \gamma \}$, we can ensure that $\gamma \times \gamma$ and $\Phi^1_{H_t}(\gamma \times \gamma)$ only intersect at the diagonal, and hence we do not require a Hamiltonian perturbation in this region. This will be an important technical detail later.

\begin{definition}[Width]
    The \emph{width} of a smooth Jordan curve $\gamma$, denoted $\operatorname{Width}_{r, \theta}(\gamma)$ is the infimal diagonal length of an inscribed isosceles trapezoid with fixed aspect ratio $r \in (0,1/2]$ and angle $\theta \in (0, \pi)$.
\end{definition}

\noindent Note that we cannot define the width of a smooth Jordan curve as the infimal diagonal length of an inscribed isosceles trapezoid (with no restriction on $r$ or $\theta$), as would be analogous to the definition used in \cite{Greene-Lobb:Floer-homology} for the Jordan Floer complex associated with rectangles, since for any fixed $r \neq 1/2$ and for any smooth Jordan curve $\gamma$ it is possible to find a sequence of inscribed isosceles trapezoids that degenerate to a point as $\theta \rightarrow \pi$. With our slightly modified definition of width, we now recall some key definitions from \cite{Greene-Lobb:Floer-homology} in order to set consistent notation and terminology.

\begin{definition}[Admissible Hamiltonian perturbation]
    Fix a smooth Jordan curve $\gamma$, an aspect ratio $0 < r \leq 1/2$, and an angle $0 < \theta < \pi$. A time dependent Hamiltonian perturbation $h_t:\mathbb{C}^2 \rightarrow \mathbb{R}$ for $0 \leq t \leq 1$ is called \emph{admissible} if it satisfies: 
    \begin{enumerate}[itemsep = 0pt]
        \item $h_t = 0$ whenever $0 \leq t \leq 0.1$ or $0.9 \leq t \leq 1$,
        \item $h_t(z,w) = 0$ whenever $|z-w| \leq \operatorname{Width}_{r, \theta}(\gamma)/2$,
        \item the Lagrangians $\gamma \times \gamma$ and $\Phi_{H_t+h_t}^1(\gamma \times \gamma)$ intersect cleanly along $\Delta(\gamma)$ and transversely elsewhere.
    \end{enumerate}
    Moreover, we call $(\gamma, r, \theta, h_t)$ an \emph{admissible quadruple}.
\end{definition}

\noindent Condition (1) ensures that $H_t + h_t = 0$ whenever $t \in [0,0.1] \cup [0.9,1]$, which means that the Cauchy-Riemann-Floer equation reduces to the Cauchy-Riemann equation for such values of $t$. Condition (2) ensures that in a neighbourhood of $\Delta(\gamma)$ we have $\Phi_{H_t+h_t}^1 = \Phi_{H_t}^1 = G_{r,\theta}$, which means that the condition Lagrangians $\gamma \times \gamma$ and $\Phi_{H_t+h_t}^1(\gamma \times \gamma)$ intersect cleanly along $\Delta(\gamma)$ is automatically satisfied. Since all other intersections of $\gamma \times \gamma$ and $\Phi_{H_t}^1(\gamma \times \gamma)$ occur at a distance of at least $\operatorname{Width}_{r,\theta}(\gamma)$ from the diagonal, condition (3) reduces to the standard transversality requirement in Lagrangian Floer theory that ensures nice behaviour of the moduli spaces. 

By the above discussion, a standard transversality argument on the space of Hamiltonians satisfying conditions (1) and (2) establishes the following lemma.
\begin{lemma}
    The set of admissible Hamiltonian perturbations is a dense open subset of the set of time dependent Hamiltonians $h_t:\mathbb{C}^2 \rightarrow \mathbb{R}$ satisfying conditions (1) and (2) of admissible Hamiltonian perturbations with respect to the $C^\infty$ topology. \qed
\end{lemma}

Recall that we denote by $X_H$ the Hamiltonian vector field associated with the Hamiltonian $H$, uniquely determined by $i_{X_H}\omega \coloneqq \omega(X_H, \cdot) = dH$.

\begin{definition}[Trajectory]
    A \emph{trajectory} $\tau:[0,1] \rightarrow \mathbb{C}^2$ of a Hamiltonian $H$ is an integral curve of the vector field $X_H$, that is $\tau'(t) = X_H \circ \tau(t)$, with $\tau(0), \tau(1) \in \gamma \times \gamma$.
\end{definition}

\noindent Trajectories are flow lines of the vector field, so we have $\tau(1) = \Phi_{H}^1(\tau(0))$ and hence the intersection points of $\gamma \times \gamma$ and $\Phi_{H_t+h_t}^1(\gamma \times \gamma)$ are in one-to-one correspondence with the trajectories of $H_t + h_t$. Moreover, the constant trajectories are exactly those which lie in the diagonal and the non-constant trajectories are disjoint from the diagonal, meaning the non-trivial intersections are in one-to-one correspondence with non-constant trajectories.

An important observation is that any non-constant trajectory of $H_t + h_t$ is not only disjoint from $\Delta(\mathbb{C})$, but remains a distance at least $\operatorname{Width}_{r,\theta}(\gamma)/2$ away from $\Delta(\mathbb{C})$. This is presented as Lemma 2.7 in \cite{Greene-Lobb:Floer-homology}, but we outline the argument here. Let $N$ be an open neighbourhood of the diagonal $\Delta(\mathbb{C})$ defined by $N = \{(z,w) : |z-w| \leq \operatorname{Width}_{r,\theta}(\gamma)/2 \}$. By the definition of an admissible Hamiltonian perturbation $h_t$, within the region $N$ we have $H_t+h_t = H_t$, but any trajectory of $H_t$ (by the definition of $H_t$) stays a fixed distance from $\Delta(\mathbb{C})$. Therefore, any trajectory must remain fully inside or outside $N$. If a trajectory $\tau$ lies entirely within $N$, then it must be a trajectory of $H_t$ and therefore $\tau(0)=(z,w)$ consists of the endpoints of an inscribes isosceles trapezoid of aspect ratio $r$ and angle $\theta$, contradicting the definition of $\operatorname{Width}_{r,\theta}(\gamma)$.

In order to define a Lagrangian Floer chain complex, one typically requires the notion of cappings of trajectories. These are of particular importance when defining the differential, the symplectic action, and the Maslov index.
\begin{definition}[Capping]
    For a trajectory $\tau$, a capping of $\tau$ is a smooth map $\widehat{\tau}:[0,1] \times [0,1] \rightarrow \mathbb{C}^2$ that satisfies
    \begin{equation} \label{eqn:capping}
        \widehat{\tau}(s,0) = \tau(0), \quad \widehat{\tau}(s,1) = \tau(1), \quad \widehat{\tau}(0,t) = \tau(t), \quad \widehat{\tau}(1,t) \in \gamma \times \gamma, \nonumber
    \end{equation}
    for all $s,t \in [0,1]$. Moreover, we denote by $[\widehat{\tau}]$ the homotopy class of $\widehat{\tau}$ with respect to these conditions.
\end{definition}
\noindent Typically, the generators of the Lagrangian Floer chain complex are then simply pairs $(\tau,[\widehat{\tau}])$ where $\tau$ is a trajectory. There are two reasons why we take a slightly different approach.

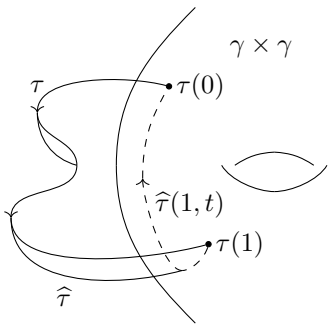
\begin{figure}[h]
    \centering
\begin{tikzpicture}[scale = 0.7]
	\begin{pgfonlayer}{nodelayer}
		\node [style=none] (0) at (1, 3) {};
		\node [style=none] (1) at (-0.5, 0) {};
		\node [style=none] (2) at (1, -3) {};
		\node [style=none] (3) at (1.5, 0) {};
		\node [style=none] (4) at (2.5, -0.5) {};
		\node [style=none] (5) at (3.5, 0) {};
		\node [style=none] (6) at (1.75, 0) {};
		\node [style=none] (7) at (3.25, 0) {};
		\node [style=none] (8) at (2.5, 0.25) {};
		\node [style=new style 2] (9) at (0.5, 1.5) {};
		\node [style=new style 2] (10) at (1.25, -1.5) {};
		\node [style=none] (11) at (-2, 1) {};
		\node [style=none] (12) at (2.25, 2.25) {$\gamma \times \gamma$};
		\node [style=none] (13) at (-2, 1.5) {$\tau$};
		\node [style=none] (14) at (1.1, 1.5) {$\tau(0)$};
		\node [style=none] (15) at (1.85, -1.5) {$\tau(1)$};
		\node [style=none] (16) at (-1.25, 0) {};
		\node [style=none] (17) at (-2.5, -1) {};
		\node [style=none] (18) at (0.75, -2) {};
		\node [style=none] (19) at (0, -0.25) {};
		\node [style=none] (20) at (0.9, -0.75) {$\widehat{\tau}(1,t)$};
		\node [style=none] (21) at (-1.5, -2.5) {$\widehat{\tau}$};
	\end{pgfonlayer}
	\begin{pgfonlayer}{edgelayer}
		\draw [in=180, out=-60, looseness=0.75] (3.center) to (4.center);
		\draw [in=240, out=0, looseness=0.75] (4.center) to (5.center);
		\draw [in=90, out=-135] (0.center) to (1.center);
		\draw [in=135, out=-90] (1.center) to (2.center);
		\draw [bend left=15] (6.center) to (8.center);
		\draw [bend left=15] (8.center) to (7.center);
		\draw [style=new edge style 5, in=90, out=165, looseness=0.75] (9) to (11.center);
		\draw [in=90, out=-90] (11.center) to (16.center);
		\draw [style=new edge style 5, in=90, out=-90] (16.center) to (17.center);
		\draw [in=-165, out=-90, looseness=0.75] (17.center) to (10);
		\draw [in=165, out=-90] (11.center) to (16.center);
		\draw [style=new edge style 0, in=0, out=-105, looseness=0.75] (10) to (18.center);
		\draw [in=-165, out=-90] (17.center) to (18.center);
		\draw [style=new edge style 6, in=-90, out=-180, looseness=0.50] (18.center) to (19.center);
		\draw [style=new edge style 0, in=-120, out=90] (19.center) to (9);
	\end{pgfonlayer}
\end{tikzpicture}
    \caption{An example of a capping $\widehat{\tau}$ of a trajectory $\tau$.}
    \label{fig:action}
\end{figure}

The first is that, as discussed above, the non-constant trajectories of $H_t+h_t$ are in one-to-one correspondence with the non-trivial intersections of $\gamma \times \gamma$ and $\Phi_{H_t+h_t}^1(\gamma \times \gamma)$. In the unperturbed case, that is $h_t = 0$, these in turn correspond to inscribes isosceles trapezoids. Constant trajectories on the other hand correspond to degenerate inscriptions of isosceles trapezoids (an inscription of an isosceles trapezoid that is a single point on the Jordan curve), which we do not wish to include in our homology theory. We therefore only consider generators $(\tau, [\widehat{\tau}])$ where $\tau$ is a non-constant trajectory.

The second is that in the Jordan Floer setting, there is a canonical choice of capping up to homotopy for non-constant trajectories that we call the \emph{preferred capping}.

\begin{definition}[Preferred Capping] \label{def:preferred_capping}
    For any non-constant trajectory $\tau$, a capping $\widehat{\tau}$ is called a \emph{preferred capping} of $\tau$ if it is disjoint from the diagonal, that is, $\operatorname{im}(\widehat{\tau}) \cap \Delta(\mathbb{C}) = \emptyset$.
\end{definition}

\noindent It is important that all preferred cappings belong to the same homotopy class of capping so that they define the same symplectic action. We therefore require the following lemma. 

\begin{lemma}
    Let $\tau$ be a non-constant trajectory, then all preferred cappings $\widehat{\tau}$ belong to the same homotopy class of capping $[\widehat{\tau}]$.
\end{lemma}
\begin{proof}   
    Recall that non-constant trajectories are disjoint from the diagonal. We note that the homotopy class of a preferred capping $\widehat{\tau}$ is completely determined by the homotopy class of $\widehat{\tau}(1,\cdot):[0,1] \rightarrow (\gamma \times \gamma) \setminus \Delta(\mathbb{\gamma})$ in $(\gamma \times \gamma) \setminus \Delta(\mathbb{\gamma})$ relative to its endpoints. We further note that the homotopy generating loop in $(\gamma \times \gamma) \setminus \Delta(\gamma)$ has winding number $\pm 1$ around $\Delta(\mathbb{C})$, so winding numbers can be changed by $\pm 1$ by concatenation with such curves. The result is then immediate, as the boundary of a preferred capping must have winding number $0$ around $\Delta(\mathbb{C})$.
\end{proof}

\begin{definition}[Jordan Floer complex]
    Let $(\gamma, r, \theta, h_t)$ be an admissible quadruple, then we denote by $\mathcal{G}(\gamma, r, \theta, h_t)$ the set of non-constant trajectories of $H_t + h_t$:
    \begin{equation}
        \tau:[0,1] \rightarrow \mathbb{C}^2, \qquad \tau(0), \tau(1) \in \gamma \times \gamma, \qquad \tau'(t) = X_{H_t + h_t} \circ \tau(t). \nonumber
    \end{equation}
    The \emph{Jordan Floer complex} $\operatorname{JFC}(\gamma, r, \theta, h_t)$ is then defined as the $\mathbb{F}_2$-vector space freely generated by $\mathcal{G}(\gamma, r, \theta, h_t)$.
\end{definition}

\subsection{Pseudo-holomorphic strips}

To define the Floer differential, we must first define \emph{pseudo-holomorphic strips}, which satisfy a Cauchy-Riemann-Floer equation, and then define \emph{moduli spaces} of these pseudo-holomorphic strips. Given an admissible quadruple $(\gamma, r, \theta, h_t)$, we must therefore choose a smooth almost-complex structure $J$ on $\mathbb{C}^2$ compatible with $\omega$. More precisely, we wish to choose a smoothly varying time-dependent almost-complex structure $J_t$ that ensures that the moduli spaces of pseudo-holomorphic strips are \emph{nice} in the sense that they consist of manifolds of the expected dimension.

To introduce these concepts formally, we first introduce the notion of a \emph{strip}, which is simply a smooth map $u: \Sigma \rightarrow \mathbb{C}^2$, where $\Sigma = \mathbb{R} \times [0,1]$ is parametrised by $(s,t)$. We then endow the symplectic manifold $(\mathbb{C}^2, \omega)$ with a smooth \emph{compatible} almost-complex structure $J$, the space of which we denote by $\mathcal{J}(\mathbb{C}^2, \omega)$. We say that $J$ is compatible with $\omega$ if $\omega(\,\cdot\,,J\,\cdot):T_p\mathbb{C}^2 \times T_p\mathbb{C}^2 \rightarrow \mathbb{R}$ defines a Riemannian metric on $\mathbb{C}^2$. This in turn induces a norm $|\cdot|_{J}^2 = \omega(\, \cdot \, , J \, \cdot \,)$ on $T\mathbb{C}^2$, which allows us to define the energy of a strip.

\begin{definition}[Energy of a strip]
    Let $\Sigma = \mathbb{R} \times [0,1]$ be parametrised by $(s,t)$ and $u: \Sigma \rightarrow \mathbb{C}^2$ be a strip in $(\mathbb{C}^2, \omega)$ with the compatible almost-complex structure $J$, then we define its energy by
    \begin{equation}
        E(u) \coloneq \int_\Sigma |\partial_s u|_{J}^2 \; dt \, ds. \nonumber
    \end{equation}
\end{definition}

\noindent The final notion that we require to define pseudo-holomorphic strips is that of a \emph{time-dependent} almost complex structure $J_t$, which are smooth paths $J_t:[0,1] \rightarrow \mathcal{J}(\mathbb{C}^2, \omega)$ of compatible almost-complex structures.

\begin{definition}[Pseudo-holomorphic strip]
We call a smooth map $u: \Sigma \rightarrow \mathbb{C}^2$ a \emph{pseudo-holomorphic strip} with respect to the time-dependent almost-complex structure $J_t \in C^\infty([0,1], \mathcal{J}(\mathbb{C}^2, \omega))$ if it satisfies the following conditions:
\begin{enumerate}[leftmargin=50pt, itemsep=0.5em]
    \item[{\makebox[35pt]{$\operatorname{(BC)}$}} - ] the Lagrangian boundary condition $u(\partial \Sigma) \subset \gamma \times \gamma$,
    \item[{\makebox[35pt]{$\operatorname{(CRF)}$}} - ] the Cauchy-Riemann-Floer equation $\partial_s u +  J_t (\partial_t u - X_{H_t+h_t} \circ  u)  = 0$,
    \item[{\makebox[35pt]{$\operatorname{(BE)}$}} - ] the bounded energy condition $E(u) < \infty$,
    \item[{\makebox[35pt]{$\operatorname{(LIM)}$}} - ] the limiting behaviour $\lim_{s \rightarrow - \infty} u(s,t) = \tau(t)$ and $\lim_{s \rightarrow \infty} u(s,t) = \tau'(t)$, where $\tau$ and $\tau'$ are trajectories of $H_t + h_t$.
\end{enumerate}
\end{definition}

\noindent For a strip $u: \Sigma \rightarrow \mathbb{C}^2$ subject to conditions (BC) and (CRF), the conditions of bounded energy (BE) and limiting behaviour (LIM) are equivalent. 

\begin{definition}[Moduli spaces of pseudo-holomorphic strips]
    We define the \emph{moduli space of pseudo-holomorphic strips}
    \begin{equation}
        \mathcal{M}(\gamma, r, \theta, h_t, J_t) \coloneqq \{u \in C^\infty(\Sigma, \mathbb{C}^2): \operatorname{(BC), (CRF), (BE)} \}. \nonumber
    \end{equation}
    Moreover, we also define the \emph{restricted moduli space of pseudo-holomorphic strips} whose limits are non-constant trajectories
    \begin{equation}
        \mathcal{M}^\circ(\gamma, r, \theta, h_t, J_t) \coloneqq \{u \in \mathcal{M}(\gamma, r, \theta, h_t, J_t): \lim_{s \rightarrow \pm \infty} u(s,t) \in \mathcal{G}(\gamma, r, \theta, h_t) \}, \nonumber
    \end{equation}
    and the \emph{moduli space of diagonal avoiding pseudo-holomorphic strips} whose closures are disjoint from the diagonal
    \begin{equation}
        \mathcal{M}^\Delta(\gamma, r, \theta, h_t, J_t) \coloneqq \{u \in \mathcal{M}(\gamma, r, \theta, h_t, J_t):  u(\overline{\Sigma}) \cap \Delta(\mathbb{C}) = \emptyset \}. \nonumber
    \end{equation}
    Finally, for trajectories $\tau, \tau' \in \mathcal{G}(\gamma, s, \theta, h_t)$, we define the \emph{moduli space of pseudo-holomorphic strips associated to $\tau$ and $\tau'$} by 
    \begin{equation}
        \mathcal{M}^\circ(\tau, \tau';J_t) \coloneq \{u \in \mathcal{M}(\gamma, r, \theta, h_t, J_t) : \lim_{s \rightarrow - \infty} u(s,t) = \tau(t), \; \lim_{s \rightarrow \infty} u(s,t) = \tau'(t) \}, \nonumber
    \end{equation}
    where the dependence on $\gamma$, $r$, $\theta$ and $h_t$ is omitted but implied by the dependence on $\tau$ and $\tau'$. Similarly, we define the \emph{moduli space of diagonal avoiding pseudo-holomorphic strips associated to $\tau$ and $\tau'$}, denoted $\mathcal{M}^\Delta(\tau, \tau';J_t)$, analogously.
\end{definition}

We note that the restricted and diagonal-avoiding moduli spaces, $\mathcal{M}^\circ$ and $\mathcal{M}^\Delta$, admit a free $\mathbb{R}$-action given by the reparametrisation
\begin{equation}
    r \cdot u(s,t) = u(r+s,t) \nonumber
\end{equation}
for $r \in \mathbb{R}$. We will define the differential in terms of counts of pseudo-holomorphic strips up to reparametrisation, motivating the following definition.

\begin{definition}[Unparametrised moduli spaces]
    For a moduli space of pseudo-holomorphic strips $\mathcal{M}$, we define the \emph{unparametrised moduli space of pseudo-holomorphic strips} by $\widehat{\mathcal{M}} = \mathcal{M} / \sim$, where $u_1 \sim u_2$ if and only if $u_1(s,t) = u_2(r+s,t)$ for some $r \in \mathbb{R}$ and all $(s,t) \in \mathbb{R} \times [0,1]$.
\end{definition}

It is also important to note that since constant trajectories necessarily lie on the diagonal we have $\mathcal{M}^\Delta(\gamma, r, \theta, h_t, J_t) \subseteq \mathcal{M}^\circ(\gamma, r, \theta, h_t, J_t)$. To prove that the differential is well defined, we need to show that paths of pseudo-holomorphic strips in $\mathcal{M}^\Delta(\gamma, r, \theta, h_t, J_t)$ limit to pseudo-holomorphic strips or concatenations of pseudo-holomorphic strips also in $\mathcal{M}^\Delta(\gamma, r, \theta, h_t, J_t)$.

\begin{proposition}[Admissible almost-complex structure existence] \label{prop:almost_complex_structures}
    Let $(\gamma, r, \theta, h_t)$ be an admissible quadruple, then there exists a non-empty subset 
    \begin{equation}
        \mathcal{J}_{\operatorname{reg}}(\gamma, r, \theta, h_t) \subset C^\infty ([0,1], \mathcal{J}(\mathbb{C}^2, \omega)) \nonumber
    \end{equation}
    of time-dependent almost-complex structures $J_t \in \mathcal{J}_{\operatorname{reg}}(\gamma, r, \theta, h_t)$ with the following properties:
    \begin{enumerate}[itemsep=0.5em]
        \item for $\tau, \tau' \in \mathcal{G}(\gamma, s, \theta, h_t)$ the restricted moduli space $\mathcal{M}^\circ(\tau, \tau';J_t)$  is a smooth manifold of dimension $\mu(u) \in \mathbb{Z}$, where $\mu(u)$ is the Maslov index of any pseudo-holomorphic strip $u$ in $\mathcal{M}^\circ(\tau, \tau';J_t)$,
        \item we have $J_t(z,w)=J_{\operatorname{std}}$ for $t \in [0,0.1] \cup [0.9,1]$ or $|z-w| \leq \operatorname{Width}_{r, \theta}(\gamma)/2$.
    \end{enumerate}
    Moreover, $\mathcal{J}_{\operatorname{reg}}(\gamma, r, \theta, h_t)$ is of the second category in the subset of the set of time-dependent complex structures $J_t \in C^\infty ([0,1], \mathcal{J}(\mathbb{C}^2, \omega))$ that satisfy (2) with respect to the $C^\infty$ topology. 
\end{proposition}

\begin{proof}
    With only condition (1), this is a standard result outlined in Auroux's survey \cite[Section 1.4]{Auroux:2013} and dealt with in detail in Schmäschke's thesis \cite[Section 7.3]{Schmäschke:2016}. That condition (2) can be added follows precisely the same argument as in the proof of \cite[Proposition 2.9]{Greene-Lobb:Floer-homology}.
\end{proof}

\subsection{Jordan Floer differential}

We now define the differential on our chain complex, which is defined in terms of counts of connected components of the 1-dimensional moduli spaces. We introduce the notation $\mathcal{M}_i$ to denote the union of all the $i$-dimensional components of the moduli space $\mathcal{M}$.

\begin{definition}[Jordan Floer differential]
    Let $\tau \in \mathcal{G}(\gamma, r, \theta, h_t)$ be a generator of the Jordan Floer complex $\operatorname{JFC}(\gamma, r, \theta, h_t)$ and $J_t \in \mathcal{J}_{\operatorname{reg}}(\gamma, r, \theta, h_t)$, then we define the \emph{Jordan Floer differential} as the linear map $\partial_{(\gamma, r, \theta, h_t, J_t)}: \operatorname{JFC}(\gamma, r, \theta, h_t) \rightarrow \operatorname{JFC}(\gamma, r, \theta, h_t)$, often denoted simply $\partial$, defined by
    \begin{equation}
        \partial(\tau) = \sum_{\tau' \in \mathcal{G}(\gamma, r, \theta, h_t)} \# \widehat{\mathcal{M}_1^\Delta}(\tau, \tau') \cdot \tau'. \nonumber
    \end{equation}
\end{definition}

We now wish to verify that this is a well defined differential for the chain complex, that is, we have $\partial^2 = 0$. The energy of pseudo-holomorphic strips is constant on any given component of the moduli space $\mathcal{M}^\circ(\gamma, r, \theta, h_t, J_t)$. By the Gromov compactness theorem, the moduli spaces $\mathcal{M}^\circ(\gamma, r, \theta, h_t, J_t)$ may therefore be compactified by including limit strips, which may be \emph{broken strips} or include \emph{bubbles}. Since the strips in $\mathcal{M}^\circ(\tau, \tau')$ do not necessarily stay away from the diagonal $\Delta(\mathbb{C})$, a limit strip could break along a constant trajectory. This justifies our definition of the differential in terms of counts of connected components of the moduli spaces $\mathcal{M}_1^\Delta(\tau, \tau')$.

Since $\mathcal{M}^\circ(\tau, \tau')$ can be compactified to $\overline{\mathcal{M}^\circ}(\tau, \tau')$ by the above argument, $\mathcal{M}^\Delta(\tau, \tau') \subseteq \mathcal{M}^\circ(\tau, \tau')$ can be compactified by including limit strips in $\overline{\mathcal{M}^\circ}(\tau, \tau')$. Let $\{u_r\}_{0 \leq r < 1} \subset \mathcal{M}_2^\Delta(\tau, \tau')$ be a continuous path of diagonal avoiding strips, then any limit strip must exist in $\overline{\mathcal{M}_2^\circ}(\tau, \tau')$, which can be broken up into the following five cases, the first four of which we will exclude.

\begin{enumerate}

    \item \textbf{Touching the diagonal}, that is, the limit $u_1 \in \mathcal{M}_2^\circ(\tau, \tau') \setminus \mathcal{M}_2^\Delta(\tau, \tau')$.    
    
    \begin{proof}[\unskip\nopunct]
        This case can be excluded, as the proof of \cite[Proposition 2.1]{Greene-Lobb:Squares-between-graphs} and the first part of the proof of \cite[Lemma 2.16]{Greene-Lobb:Floer-homology} carry over exactly to this setting. The first ensures that $u_1$ does not intersect $\Delta(\mathbb{C})$ on the boundary of its image and the first part of the proof of the second rules out the possibility that $u_1$ intersects $\Delta(\mathbb{C})$ on the interior of its image (without relying on the assumption $\gamma$ is real analytic) under the assumption that there is no boundary intersection with the diagonal.

        We wish to sketch the ideas of these proofs. We first suppose that the set of points mapped to the diagonal $S=\{p \in \Sigma \mid u_1(p) \in \Delta(\mathbb{C}) \}$ had an accumulation point $p \in \Sigma^\circ$. Then by analytic continuation of $v(s,t) = (\Phi_{H_t+h_t}^t)^{-1} \circ u(s,t)$, which are holomorphic maps in a neighbourhood of $\Delta(\gamma)$ (see \cite[Section 2.5]{Greene-Lobb:Floer-homology} for the details of this reparametrisation), we have $u_1(\Sigma) \subset \Delta(\mathbb{C})$, giving a contradiction. We may therefore assume that every point $p \in S \cap \Sigma^\circ$ is isolated.

        Let $p \in S \cap \Sigma^\circ$, then since points of $S \cap \Sigma^\circ$ are isolated we can find a loop $L \subset \Sigma \setminus S$ around $p$. The key idea we require from \cite[Lemma 2.16]{Greene-Lobb:Floer-homology} is that the loop $u_1(L)$ has zero winding number around $\Delta(\mathbb{C})$ (which is well defined as $L \subset \Sigma \setminus S$), since each $u_r(L)$ does for $0 \leq r < 1$. However, this contradicts the positivity of such isolated intersection points established in \cite{Ganatra-Pomerleano:2021}.

        We may now assume that $S \subset \partial\Sigma$. Without loss of generality, assume $p = (0,0) \in S$. The idea behind \cite[Proposition 2.1]{Greene-Lobb:Squares-between-graphs} is that in a neighbourhood $N = (-1,1) \times [0,0.1)\subset \Sigma$ (for which $J_t = J_{\operatorname{std}}$ and $X_H = 0$) the map $\pi_d \circ u_1:N_p \rightarrow \mathbb{C}$ (where $\pi_d(z,w) = z-w$ is a projection) is a non-constant map, holomorphic on its interior and $(\pi_d \circ u_1)^{-1}(0) \subseteq (-1,1) \times \{0\}$. If $(\pi_d \circ u_1)^{-1}(0)$ contained an open interval, $\pi_d \circ u_1$ would vanish identically on $N$ by \cite{Chernoff:1967}, giving a contradiction. We may therefore choose an $\varepsilon > 0$ such that $\pi_d \circ u_1(\pm\varepsilon) \neq 0$ and (without loss of generality)
        \begin{equation}
            \arg(T_{u_1(p)} \gamma) - \pi/8 \leq \arg\bigl(\pi_d \circ u_1\bigl((-1,1) \times \{0\}\bigr)\bigr) \leq \arg(T_{u_1(p)} \gamma) + \pi/8. \nonumber
        \end{equation}
        Assume further, again without loss of generality, that $\arg(T_{u_1(p)} \gamma) = \pi/4$. We may then choose a simple smooth closed curve $C \subset N$ such that $(-\varepsilon,\varepsilon) \times \{0\} \subset C$ and the region $R$ bounded by $C$ satisfies $\pi_d \circ u_1(R) \subset Q$, where $Q$ is the first quadrant of the complex plane. Writing $\pi_d \circ u = x+iy$, we note that $x,y > 0$ on $R$ so by the Hopf Lemma the outward normal derivative of $xy$ at $p$ is strictly negative. On the other hand, by the product rule, this derivative is $0$ since $x(p)=y(p)=0$, giving a contradiction.
    \end{proof}
    
    \item \textbf{Diagonal breaking}, that is, the limit $u_1$ is the concatenation $u_1 = u_1^- \# u_1^+$ of two pseudo-holomorphic strips $u_1^- \in \mathcal{M}_1(\tau, \tau'')$, $u_1^+ \in \mathcal{M}_1(\tau'', \tau')$ for some trajectory $\tau'' \notin \mathcal{G}(\gamma, r, \theta, h_t)$.
    
    \begin{proof}[\unskip\nopunct]
        To exclude diagonal breaking we first note that \cite[Lemma 2.17]{Greene-Lobb:Floer-homology} extends immediately to the smooth setting, as the analytic assumption on the Jordan curve is not used. The only possible obstruction to the proof in \cite[Lemma 2.17]{Greene-Lobb:Floer-homology} directly carrying over to the isosceles trapezoid setting is that the Hamiltonian flow $\Phi_{H_t+h_t}^t$ is different. However, under the projection $\pi_d$ the Hamiltonian flow is independent of the aspect ratio $r$ and therefore $\pi_d \circ (\Phi_{H_t+h_t}^1)^{-1}$ is still simply a rotation by the angle $\theta$. This case can then be excluded.

        We now sketch the main ideas this proof. First note that by the arguments for excluding touching the diagonal, we have that $u_1^-(\Sigma)$ and $u_1^+(\Sigma)$ must be disjoint from the diagonal. In particular, we only need to consider the case $\lim_{s\rightarrow\infty} u_1^-(s,t) = \lim_{s\rightarrow-\infty} u_1^+(s,t) = (p,p)$ for some $p \in \gamma$. We can then choose an $R > 0$ sufficiently large so that $u_1^-(D^+)$ and $u_1^+(D^-)$, where $D^+ = [R,\infty) \times [0,1]$ and $D^- = (-\infty,-R] \times [0,1]$, are close enough to the diagonal that $J_t = J_{\operatorname{std}}$ on these regions.

        As we wish to work with holomorphic maps, we introduce the reparametrisations $v_1^+$ and $v_1^-$ of $u_1^+$ and $u_1^-$ respectively, given by
        \begin{equation}
            v_1^\pm(s,t) = (\Phi_{H_t+h_t}^t)^{-1} \circ u_1^\pm(s,t). \nonumber
        \end{equation}
        Similarly, we can reparametrise $u_r$ to obtain $v_r$. See \cite[Section 2.5]{Greene-Lobb:Floer-homology} for details of this reparametrisation, the key results from which are: the reparametrised strips satisfy the Cauchy-Riemann equation
        \begin{equation}
            \partial_s v_1^\pm + J_{\operatorname{std}}(\partial_t v_1^\pm) = 0, \nonumber
        \end{equation}
        they satisfy the boundary conditions
        \begin{eqnarray}
            v_1^-([R,\infty) \times \{0\}) \subset \gamma \times \gamma,  \qquad\quad\;\;  v_1^-([R,\infty) \times \{1\}) \subset (\Phi_{H_t+h_t}^1)^{-1}(\gamma \times \gamma), \nonumber \\
            v_1^+((-\infty,-R] \times \{0\}) \subset \gamma \times \gamma,  \qquad  v_1^+((-\infty,-R] \times \{1\}) \subset (\Phi_{H_t+h_t}^1)^{-1}(\gamma \times \gamma), \nonumber
        \end{eqnarray}
        and if $u_1^\pm(\overline{\Sigma})$ is disjoint from the diagonal then so is $v_1^\pm(\overline{D^\mp})$. The same results hold for appropriately chosen domains of the reparametrisations $v_r$ of $u_r$.
        
        We then study the holomorphic maps $\pi_d \circ v_r^\pm: D^\mp \rightarrow \mathbb{C}$. By considering winding numbers of boundary curves, there exists an $a \in \pi_d \circ v_1^-(\operatorname{int}(D^+))$ and $D_r = [R_r,R_r'] \times [0,1]$ with $a \in \pi_d \circ v_r(D_r)$ such that for all $r$ sufficiently close to 1 we have that $a, 0 \in \mathbb{C}$ are in the same connected component of $\mathbb{C} \setminus \pi_d \circ v_r(\partial D_r)$. It then follows that $v_r(\Sigma)$ intersects the diagonal, giving a contradiction.
    \end{proof}
    
    \item \textbf{Disk bubbling}, that is, the limit $u_1=u_1' \# b$ contains a disk bubble $b:D^2 \rightarrow \mathbb{C}^2$.
    
    \begin{proof}[\unskip\nopunct]
        It is natural to be concerned that the arguments used to exclude disk bubbling in the case of rectangles may not extend to isosceles trapezoids since monotonicity of the Lagrangians is often used to exclude disk bubbling and the Lagrangian tori in the isosceles trapezoid setting are not monotone. However, the arguments of \cite[Lemma 2.19]{Greene-Lobb:Floer-homology} and \cite[Proposition 2.3]{Greene-Lobb:Squares-between-graphs} do not directly rely on the classical monotonicity argument and carry over directly to this setting.
        
        As noted at the beginning of the proof of \cite[Lemma 2.19]{Greene-Lobb:Floer-homology}, it is not possible for a disc bubble $B = b(D^2)$ to be disjoint from the diagonal, since its boundary loop $\partial B \subset (\gamma \times \gamma) \setminus \Delta(\gamma)$ must either bound a disc of zero symplectic area (and hence cannot bound a disc bubble), or have non-zero winding number around the diagonal (and hence $B$ must intersect $\Delta(\mathbb{C})$). Moreover, by the arguments used to exclude touching the diagonal, the resulting strip and disc bubble must be disjoint from the diagonal except (potentially) at the point the disc $B$ bubbles off.

        The arguments of \cite[Proposition 2.3]{Greene-Lobb:Squares-between-graphs} extend immediately to the isosceles trapezoid setting and exclude the possibility of bubbling at the diagonal. The arguments used in the proof of this proposition are similar to those used for excluding touching the diagonal. Namely, suppose that the bubbling occurs at $(p,p) \in \Delta(\gamma)$ and consider the projections $\pi_d \circ u'_1: R_1 \rightarrow \mathbb{C}$ and $\pi_d \circ b:R_2 \rightarrow \mathbb{C}$ for some neighbourhoods $R_1 \subset \Sigma$ of $(u_1')^{-1}(p,p)$ and $R_2 \subset D^2$ of $b^{-1}(p,p)$, with smooth boundary such that both maps are holomorphic on their interior. A contradiction is then obtained by a similar approach to the one outlined for \cite[Proposition 2.1]{Greene-Lobb:Squares-between-graphs} above.
    \end{proof}
    
    \item \textbf{Sphere bubbling}, that is, the limit $u_1=u_1' \# b$ contains a sphere bubble $b:S^2 \rightarrow \mathbb{C}^2$.
    
    \begin{proof}[\unskip\nopunct]
        This case can be excluded due to a well known argument, succinctly summarised in \cite[Lemma 2.18]{Greene-Lobb:Floer-homology}. The idea of the argument is that a sphere bubble implies the existence of a non-constant $J$-holomorphic sphere $b:S^2 \rightarrow \mathbb{C}^2$, which must have positive energy due to the non-constant property, which, in turn, is equal to the symplectic area $\omega([b])$ by the $J$-holomorphic property. However, since $H_2(\mathbb{C}^2) = 0$, we have $\omega([b])=0$ for any sphere $b:S^2 \rightarrow \mathbb{C}^2$.
    \end{proof}
    
    \item \textbf{Good breaking}, that is, the limit $u_1 \in \mathcal{M}_2^\Delta(\tau, \tau')$ or the limit $u_1 = u_1^- \# u_1^+$ is the concatenation of two pseudo-holomorphic strips $u_1^- \in \mathcal{M}(\tau, \tau'')$, $u_1^+ \in \mathcal{M}(\tau'', \tau')$ for some non-constant trajectory $\tau'' \in \mathcal{G}(\gamma, r, \theta, h_t)$.
    
    \begin{proof}[\unskip\nopunct]
        We necessarily land in this case, ensuring the boundary of the two dimensional moduli spaces consists of subsets of products of the one dimensional spaces, namely
        \begin{equation} \label{eqn:moduli_space_inclusion}
            \partial\widehat{\mathcal{M}}_2^\Delta(\tau, \tau'; J_t) \subseteq \bigsqcup_{\tau'' \in \mathcal{G}(\gamma, r, \theta, h_t)} \widehat{\mathcal{M}}_1^\Delta(\tau, \tau''; J_t) \times \widehat{\mathcal{M}}_1^\Delta(\tau'', \tau; J_t). \nonumber
        \end{equation}
        This is the key technical result in showing that the differential squares to zero.
        \renewcommand{\qedsymbol}{}
    \end{proof}
    
\end{enumerate}

\noindent By the gluing theorem for pseudo-holomorphic strips, any broken strip is locally the limit of a unique family of index 2 strips. This provides the reverse inclusion to (\ref{eqn:moduli_space_inclusion}) on the reduced moduli spaces, giving the equality
\begin{equation} \label{eqn:moduli_space_equivalence}
    \partial \widehat{\mathcal{M}}_2^\Delta(\tau, \tau'; J_t) = \bigsqcup_{\tau'' \in \mathcal{G}(\gamma, r, \theta, h_t)} \widehat{\mathcal{M}}_1^\Delta(\tau, \tau''; J_t) \times \widehat{\mathcal{M}}_1^\Delta(\tau'', \tau'; J_t). \nonumber
\end{equation}

\begin{theorem}[$\partial^2 = 0$]
    Suppose that $(\gamma, r, \theta, h_t, J_t)$ is an admissible tuple, then the boundary operator $\partial = \partial_{(\gamma, r, \theta, h_t, J_t)}$ on $\operatorname{JFC}(\gamma, r, \theta, h_t)$ satisfies $\partial^2 = 0$.
\end{theorem}

\begin{proof}
    From the equality given in (\ref{eqn:moduli_space_equivalence}) we have
    \begin{eqnarray}
        \partial^2(\tau) &=& \sum_{\tau', \tau'' \in \mathcal{G}(\gamma, r, \theta, h_t)} \left( \# \widehat{\mathcal{M}}_1^\Delta(\tau, \tau'';J_t) \right) \left( \# \widehat{\mathcal{M}}_1^\Delta(\tau'', \tau';J_t) \right) \cdot \tau'  \nonumber \\
        &=& \sum_{\tau' \in \mathcal{G}(\gamma, r, \theta, h_t)} \# \left( \partial \widehat{\mathcal{M}}_2^\Delta(\tau, \tau';J_t) \right) \cdot \tau' = 0. \nonumber
    \end{eqnarray}
    Here we use the fact that $\widehat{\mathcal{M}_2^\Delta}(\tau, \tau'')$ is a 1-dimensional manifold and therefore has an even number of points in its boundary and that $\operatorname{JFC}(\gamma, r, \theta, h_t)$ was defined over $\mathbb{F}_2$.
\end{proof}

Having verified that the map $\partial:\operatorname{JFC}(\gamma, r, \theta, h_t) \rightarrow \operatorname{JFC}(\gamma, r, \theta, h_t)$ is a well defined differential, we can now define the Jordan Floer homology of an admissible tuple $(\gamma, r, \theta, h_t, J_t)$.

\begin{definition}[Jordan Floer homology]
    For an admissible tuple $(\gamma, r, \theta, h_t, J_t)$, we define the \textit{Jordan Floer homology}
    \begin{equation}
        \operatorname{JF}(\gamma, r, \theta, h_t, J_t) \coloneq \ker \partial_{(\gamma, r, \theta, h_t, J_t)} / \operatorname{im} \partial_{(\gamma, r, \theta, h_t, J_t)}. \nonumber
    \end{equation}
\end{definition}

We recall that our aim was to show that the chain complex was non-empty, which would be implied by showing that this homology is non-trivial. We achieve this goal through continuation maps.

\subsection{The continuation package}

Currently we require a choice of Hamiltonian perturbation and almost-complex structure in order to define the Jordan Floer homology. In Lagrangian Floer homology, there is a standard procedure for showing that the homology is independent of these choices, which requires chain maps known as \emph{continuation maps} that induce isomorphisms on homology. Similarly to the differential, continuation maps are defined by a count of pseudo-holomorphic strips. In the Jordan Floer setting, we can improve on this. For a given aspect ratio $r$, we can show that the Jordan Floer homology is independent of not only the choice of Hamiltonian perturbation $h_t$ and almost-complex structure $J_t$, but also of the angle $\theta$. A detailed treatment of continuation maps in the more general setting is given in \cite[Chapter 11]{Audin-Damian:2014}.

\subsubsection{Varying the angle}

To show that Jordan Floer homology is independent of the angle $\theta$ we follow \cite[Section 2.7]{Greene-Lobb:Floer-homology}. That is, we define chain homotopy equivalences $\operatorname{JFC}(\gamma, r, \theta^1, h_t^1, J_t^1) \rightarrow \operatorname{JFC}(\gamma, r, \theta^2, h_t^2, J_t^2)$ known as continuation maps. These are defined as counts of $s$-dependent pseudo-holomorphic strips, which are defined analogously to the standard pseudo-holomorphic strips but for the following data:
\begin{enumerate}
    \item We define a strip-dependent Hamiltonian $H_{st}^{12}:\mathbb{C}^2 \rightarrow \mathbb{R}$, which for a smooth monotonic function of the angle $\theta_s: \mathbb{R} \rightarrow [0,1]$, satisfying $\theta^{12}_s = \theta^1$ for $s \leq 0$ and $\theta^{12}_s = \theta^2$ for $s \geq 1$, is given by
    \begin{equation}
        H_{st}^{12}(z,w) = \frac{1}{2}r(1-r)\theta_s^{12} |z-w|^2. \nonumber
    \end{equation}
    \item We introduce a strip-dependent Hamiltonian perturbation $h_{st}^{12}: \mathbb{C}^2 \rightarrow \mathbb{R}$ to ensure transversality, which we require to satisfy $h_{st}^{12} = h_t^1$ for $s \leq -S_1$, $h_{st}^{12} = h_t^2$ for $s \geq S_1$, $h_{st}^{12} = 0$ for $t \notin (0.1,0.9)$, and all partial derivatives of $h_{st}^{12}$ up to infinite order vanish at $\Delta(\mathbb{C})$.
    \item We choose a smooth path of almost-complex structures $J_{st}^{12}$, which is admissible for the Hamiltonian $H_{st}^{12}+h_{st}^{12}$, that satisfies $J_{st}^{12} = J_t^1$ for $s \leq -S_2$, $J_{st}^{12} = J_t^2$ for $s \geq S_2$, and $J_{st}^{12} = J_{\operatorname{std}}$ for $t \notin (0.1,0.9)$ or $(z,w) \in \mathbb{C}^2$ sufficiently close to the diagonal.
\end{enumerate}
The continuation maps are then defined in a similar way to the differential. More specifically, with respect to the bases $\mathcal{G}(\gamma, r, \theta^1, h_t^1)$ and $\mathcal{G}(\gamma, r, \theta^2, h_t^2)$, as counts of diagonal avoiding pseudo-holomorphic strips between trajectories in these bases. 

To be more precise, we define the moduli space of $s$-dependent pseudo-holomorphic strips limiting to trajectories $\tau_1 \in \mathcal{G}(\gamma, r, \theta^1, h_t^1)$ and $\tau_2 \in \mathcal{G}(\gamma, r, \theta^2, h_t^2)$ by
\begin{equation}
    \mathcal{M}^\circ(\tau_1, \tau_2;h_{st},J_{st}) \coloneq \Big\{u \in \mathcal{M}(\gamma, r, \theta_s, h_{st}, J_{st}) : \lim_{s \rightarrow - \infty} u(s,t) = \tau(t), \; \lim_{s \rightarrow \infty} u(s,t) = \tau'(t) \Big\}. \nonumber
\end{equation}
For \emph{admissible} choices of Hamiltonian perturbation $h_{st}$ and almost-complex structure $J_{st}$, each  of these moduli spaces has the structure of a manifold of dimension equal to the Maslov index of any strip in that connected component. 

That generic choices of $h_{st}$ satisfying (2) are admissible in the above sense for any fixed $J_{st}$ satisfying (3), without the restriction that all partial derivatives of $h_{st}$ up to infinite order vanish at the diagonal, is a standard result and can be found in \cite[Theorem 11.1.6]{Audin-Damian:2014}. The fact that this restriction can be added follows an argument similar to that used to prove the moduli spaces associated to the differential are manifolds (recall Proposition \ref{prop:almost_complex_structures}). More precisely, the argument for only requiring $J_t$ to vary in an open set that intersects the image of every non-constant pseudo-holomorphic strip used in the proof of \cite[Proposition 2.9]{Greene-Lobb:Floer-homology} (see the footnote of \cite[Proposition 3.4.1]{McDuff-Salamon:1994} for more details) can be applied in this setting to justify that we can fix $h_{st}$ in a region whose complement is an open set which intersects the image of every non-constant $s$-dependent pseudo-holomorphic strip in the $s$-dependent region. To see that this is satisfied by the region $\Delta(\mathbb{C})$, simply note that if the $s$-dependent region was fully contained in $\Delta(\mathbb{C})$, by analytic continuation, the whole strip must be contained in $\Delta(\mathbb{C})$ and therefore cannot limit to non-constant trajectories. 

Since we only wish to consider diagonal avoiding strips, we in turn define the moduli space of diagonal avoiding strips as the subspace
\begin{equation}
    \mathcal{M}^\Delta(\tau_1, \tau_2; h_{st}, J_{st}) \coloneq \Big\{ u \in \mathcal{M}^\circ(\tau_1, \tau_2;h_{st},J_{st}) : u(\overline{\Sigma}) \cap \Delta(\mathbb{C}) = \emptyset \Big\}. \nonumber
\end{equation}
The continuation maps are then defined analogously to the differential, with the counts of 1-dimensional connected components of $\mathcal{M}^\Delta(\tau_1, \tau_2; h_t, J_t)$ replaced by counts of 0-dimensional connected components of $\mathcal{M}^\Delta(\tau_1, \tau_2; h_{st}, J_{st})$.

\begin{definition}[Continuation maps]
    Let $\tau_1 \in \mathcal{G}(\gamma, r, \theta^1, h^1_t)$, then we define the \emph{continuation map} for the data $(\theta_s, h_{st}, J_{st})$ as the linear map $\Phi_{(\gamma, r, \theta_s, h_{st}, J_{st})}: \operatorname{JFC}(\gamma, r, \theta^1, h^1_t) \rightarrow \operatorname{JFC}(\gamma, r, \theta^2, h^2_t)$, often denoted simply $\Phi$ by
    \begin{equation}
        \Phi(\tau_1) = \sum_{\tau_2 \in \mathcal{G}(\gamma, r, \theta^2, h^2_t)} \# \mathcal{M}_0^\Delta(\tau_1, \tau_2;h_{st}, J_{st}) \cdot \tau_2. \nonumber
    \end{equation}
\end{definition}

Verifying that such a map defines a chain map requires analysing the limits of sequences of pseudo-holomorphic strips in the 1-dimensional components of the moduli spaces $\mathcal{M}_1^\Delta(\tau_1, \tau_2;h_{st}, J_{st})$. This is similar to verifying the differential for Jordan Floer homology is well defined and squares to zero. More precisely, we must verify that the limits of sequences of pseudo-holomorphic strips $u_r \in \mathcal{M}_1^\Delta(\tau_1, \tau_2;h_{st}, J_{st})$ avoid the diagonal and do not include bubbles. The previous arguments for no diagonal touching, disc bubbling, or sphere bubbling carry over to this setting, but to extend the arguments for no breaking at the diagonal we must take more care.

One obstruction to the immediate extension of the argument is that a pseudo-holomorphic strip $u \in \mathcal{M}_1^\Delta(\tau_1, \tau_2;h_{st}, J_{st})$ may limit (as $s \rightarrow \pm \infty$) to a trajectory of a Jordan Floer complex with a different angle. We recall that our definition of being \emph{close} to the diagonal as $|z-w| \leq \operatorname{Width}_{r, \theta}(\gamma)$ depended on $\theta$, which in the current setting is no longer constant. We hence introduce a new notion of width:

\begin{definition}[Theta width]
    The theta width of a smooth Jordan curve $\gamma$, denoted $\operatorname{Width}_r^\theta(\gamma)$ is the infimal diagonal length of an inscribed isosceles trapezoid of aspect ratio $r$ and angle $\phi \in (0, \theta)$.
\end{definition}

\noindent With this new notion of width, we take $\theta = \max \{ \theta_1, \theta_2 \}$ and replace $\operatorname{Width}_{r, \theta}(\gamma)$ with $\operatorname{Width}_r^\theta(\gamma)$ in all corresponding definitions. 

The other obstruction is that some of the arguments require a slight modification to deal with $s$-dependent pseudo-holomorphic strips. To be more explicit, using the same arguments for excluding diagonal breaking as before, we may restrict ourselves to the case where the broken strip $u_1 = u_1^- \# u_1^+$ only intersects the diagonal at the point it breaks. The broken strip must therefore consist of either
\begin{equation}
    u_1^- \in \mathcal{M}_0(\tau_1^-,\tau_2^-; \theta_s, h_{st},J_{st}) \qquad \text{or} \qquad u_1^+ \in \mathcal{M}_0(\tau_1^+,\tau_2^+; \theta_s, h_{st},J_{st}), \nonumber
\end{equation}
for some constant trajectory $\tau_2^- \in \Delta(\mathbb{C})$ or $\tau_1^+ \in \Delta(\mathbb{C})$ respectively, which the previous arguments do not directly deal with. In the setting of Jordan Floer homology for inscriptions of rectangles in analytic Jordan curves \cite[Lemma 2.22]{Greene-Lobb:Floer-homology} extends the arguments of \cite[Lemma 2.17]{Greene-Lobb:Floer-homology} to $s$-dependent pseudo-holomorphic strips. With the new notion of width in place, \cite[Lemma 2.22]{Greene-Lobb:Floer-homology} extends to our setting by the same observations that extend \cite[Lemma 2.17]{Greene-Lobb:Floer-homology} to this setting.

By the above arguments, we are in the case of \emph{good breaking}. More precisely, for a sequence of strips $u_r \in \mathcal{M}_1^\Delta(\tau_1, \tau_2;h_{st}, J_{st})$ for $r \in [0,1)$ limiting to a strip $u_1 \notin \mathcal{M}_1^\Delta(\tau, \tau';h_{st}, J_{st})$, the limit is necessarily a concatenation $u_1 = u_1^- \# u_1^+$ of pseudo-holomorphic strips. There are two possibilities for such strips: 
\begin{enumerate}[itemsep = 0pt]
    \item $u_1^- \in \mathcal{M}_1^\Delta(\tau_1, \tau_1'; \theta^1, h^1_t, J^1_t)$ and $u_1^+ \in \mathcal{M}_0^\Delta(\tau_1', \tau_2; \theta_s, h_{st}, J_{st})$ for some $\tau_1' \in \mathcal{G}(\gamma, r, \theta^1, h_t^1)$,
    \item $u_1^- \in \mathcal{M}_0^\Delta(\tau_1, \tau_2'; \theta_s, h_{st}, J_{st})$ and $u_1^+ \in \mathcal{M}_1^\Delta(\tau_2', \tau_2; \theta^2, h^2_t, J^1_t)$ for some $\tau_2' \in \mathcal{G}(\gamma, r, \theta^2, h_t^2)$.
\end{enumerate}
Along with the gluing theorem for pseudo-holomorphic strips, this establishes the equality on the moduli spaces of pseudo-holomorphic strips
\begin{eqnarray}
    \partial \mathcal{M}_1^\Delta(\tau_1, \tau_2; \theta_s, h_{st}, J_{st}) &=& \bigsqcup_{\tau_1' \in \mathcal{G}(\gamma, r, \theta^1, h_t^1)} \widehat{\mathcal{M}}_1^\Delta(\tau_1, \tau_1'; \theta^1, h^1_t, J^1_t) \times \mathcal{M}_0^\Delta(\tau_1', \tau_2; \theta_s, h_{st}, J_{st}) \nonumber \\
    && \quad \sqcup \bigsqcup_{\tau_2' \in \mathcal{G}(\gamma, r, \theta^2, h_t^2)} \mathcal{M}_0^\Delta(\tau_1, \tau_2'; \theta_s, h_{st}, J_{st}) \times \widehat{\mathcal{M}}_1^\Delta(\tau_2', \tau_2; \theta^2, h^2_t, J^1_t). \nonumber
\end{eqnarray}
Since $\mathcal{M}_1^\Delta(\tau_1, \tau_2; \theta_s, h_{st}, J_{st})$ is a one-dimensional manifold, recalling that the Jordan Floer chain complex was defined over $\mathbb{F}_2$ coefficients, this establishes $\Phi \circ \partial^1 + \partial^2 \circ \Phi = 0$.

Verifying that the continuation maps are chain homotopy equivalences requires studying the moduli space of homotopies of $s$-dependent pseudo-holomorphic strips 
\begin{equation}
    \mathcal{M}^\Delta \left(\tau_1, \tau_2; \left\{h_{st}^\lambda, J_{st}^\lambda \right\}_\lambda \right) \coloneq \Big\{ (\lambda, u):0 \leq \lambda \leq 1, u \in \mathcal{M}^\Delta \left(\tau_1, \tau_2; h_{st}^\lambda,J_{st}^\lambda \right) \Big\}, \nonumber
\end{equation}
where $\lambda \mapsto h_{st}^\lambda$ and $\lambda \mapsto J_{st}^\lambda$ are homotopies of admissible Hamiltonians and almost-complex structures respectively. More precisely, we are interested in the boundaries of the 1-dimensional connected components of this moduli space $\mathcal{M}_1^\Delta(\tau_1, \tau_2; \{h_{st}^\lambda, J_{st}^\lambda\}_\lambda)$, which is a 1-dimensional manifold for generic homotopies $\lambda \mapsto h_{st}^\lambda$ and $\lambda \mapsto J_{st}^\lambda$ (see \cite[Section 11.3]{Audin-Damian:2014} for example). Moreover, since we have already verified that limits of strips in $\mathcal{M}^\Delta \left(\tau_1, \tau_2; h_{st}^\lambda,J_{st}^\lambda \right)$ avoid the diagonal, so do limits of strips in $\mathcal{M}_1^\Delta(\tau_1, \tau_2; \{h_{st}^\lambda, J_{st}^\lambda\}_\lambda)$ ensuring \emph{good breaking}.

The relevant chain homotopy $h:\operatorname{JFC}(\gamma, r, \theta^1,h_t^1) \rightarrow \operatorname{JFC}(\gamma, r, \theta^1,h_t^1)$ can then be defined as counts of 0-dimensional connected components of these moduli spaces
\begin{equation}
    h(\tau_1) = \sum_{\tau_2 \in \mathcal{G}(\gamma,r,\theta^1,h_t^1)} \# \mathcal{M}_0^\Delta \left(\tau_1, \tau_2; \left\{ h_{st}^\lambda, J_{st}^\lambda \right\} \right) \cdot \tau_2. \nonumber
\end{equation}
Since $\mathcal{M}_1^\Delta(\tau_1, \tau_2; \{h_{st}^\lambda, J_{st}^\lambda\}_\lambda)$ is a 1-dimensional manifold, it has an even number of boundary components. This is sufficient to conclude that for any pair of continuation maps $\Phi^1:\operatorname{JFC}(\gamma, r, \theta^1, h^1_t) \rightarrow \operatorname{JFC}(\gamma, r, \theta^2, h^2_t)$ and $\Phi^2: \operatorname{JFC}(\gamma, r, \theta^2, h^2_t) \rightarrow \operatorname{JFC}(\gamma, r, \theta^1, h^1_t)$ we have 
\begin{equation}
    \Phi^2 \circ \Phi^1 + \operatorname{id}_1 + \;\partial_1 \circ h + h \circ \partial_1 = 0 \nonumber
\end{equation}
over $\mathbb{F}_2$, where $\operatorname{id}_1$ and $\partial_1$ are the identity and the differential on the chain complex $\operatorname{JFC}(\gamma, r, \theta^1, h^1_t)$ respectively. This verifies continuation maps induce isomorphisms between $\operatorname{JF}(\gamma,r, \theta, h_t,J_t)$ for different angles $\theta$, Hamiltonians $h_t$, and almost-complex structures $J_t$.

\subsubsection{Morse Homology limit}

We consider the limit of Jordan Floer homology $\operatorname{JF}_*(\gamma, r, \theta)$ as $\theta \rightarrow 0$. This was done for the case of rectangles in \cite{Greene-Lobb:Floer-homology}, following an argument of Oh \cite{Oh:1996} that we begin by reviewing. 

Consider a symplectic manifold $(M, \omega)$, a Lagrangian $L \subset M$, and a sequence of admissible Hamiltonians $H_n: M \rightarrow \mathbb{R}$ where $H_n = \tfrac{1}{n} f$ for some $f: M \rightarrow \mathbb{R}$. Let $N$ be a Weinstein neighbourhood of $L$ in $M$, then for $H_n$ sufficiently small $L'=\Phi_{H_n}^1(L)$ is contained in $N$. Moreover, there exists a small energy $e > 0$ such that the strips counted by the differential on the Floer chain complex $CF_*(L,L';J_t)$ of energy $E(u) < e$ are in one-to-one correspondence with Morse trajectories counted by the differential on the Morse chain complex $CM_*(L, f|_{L} + h,g)$ by the assignment $u \mapsto u(s,0)$ (for some suitable choice of perturbation $h:L \rightarrow \mathbb{R}$ such that $f|_{L} + h$ is a Morse function, metric $g:L \times L \rightarrow \mathbb{R}$, and almost-complex structure $J_t$).

In the Jordan Floer setting, since we wish to consider the limit $\theta \rightarrow 0$ for $H_t = \tfrac{1}{2} \theta r(1-r) |z-w|^2$, we take $f = |z-w|^2$. The key difference in the Jordan Floer setting is that the assignment $u \mapsto u(s,0)$ puts low energy strips counted by the Jordan Floer differential (diagonal avoiding strips of Maslov index 1) in one-to-one correspondence with Morse trajectories between non-degenerate critical points that are counted by the differential on the Morse-Bott chain complex $CM_*(\gamma \times \gamma, \Delta(\gamma), f|_{\gamma \times \gamma} + h, g)$ for a suitable choice of perturbation $h$ and metric $g$ on $\gamma \times \gamma$ as above. This was shown in the case of rectangles in \cite[Theorem 2.23]{Greene-Lobb:Floer-homology}.

The argument extends without modification to the case of isosceles trapezoids. Geometrically, this extension is expected: as $\theta_n \rightarrow 0$, convergent sequences of inscribed isosceles trapezoids $T_{\theta_n}$ of a fixed aspect ratio $r$ in $\gamma$ tend to binormals of $\gamma$. This holds regardless of the choice of $r$ and therefore we expect the limiting Morse-Bott chain complex to be independent of the choice of aspect ratio. We note that the homology of this Morse-Bott chain complex is given by
\begin{equation}
    H_*\left(\operatorname{CM}_*(\gamma \times \gamma, \Delta(\gamma),  f|_{\gamma \times \gamma}+h, g)\right) \cong H_*\left(S^1 \times S^1, \Delta(S^1); \mathbb{F}_2\right) \cong (\mathbb{F}_2)_{(2)} \oplus (\mathbb{F}_2)_{(1)}, \nonumber
\end{equation}
the Morse-Bott homology of a torus relative to its diagonal curve. By the continuation package for varying $\theta$ we obtain the Jordan Floer homology associated with inscriptions of isosceles trapezoids for any values of $r$ and $\theta$.

\begin{theorem}[Isomorphism Type]
    The Jordan Floer homology associated with inscribed isosceles trapezoids of aspect ratio $r$ and angle $\theta$ is given by
    \begin{equation}
        \operatorname{JF}_*(\gamma, r, \theta) \coloneqq H_*\big(\operatorname{JFC}_*(\gamma, r, \theta, h_t, J_t)\big) \cong (\mathbb{F}_2)_{(2)} \oplus (\mathbb{F}_2)_{(1)}. \nonumber
    \end{equation}
\end{theorem}

\subsubsection{Varying the aspect ratio}

Since Jordan Floer homology is independent (up to isomorphism) of the aspect ratio $r$ and the angle $\theta$, it is natural to hope for a continuation package for varying $r$. The main difficulty when considering varying $r$ is that the symplectic form depends on $r$. It is therefore unclear how to define continuation maps as counts of pseudo-holomorphic strips, as to define pseudo-holomorphic strips we first must make a choice of symplectic form.

\section{The symplectic action and spectral invariants}

The Jordan Floer homology has an integer homological grading, supported in degrees 1 and 2, called the Maslov index. The Jordan Floer chain complex also has a real filtration called the action. This lifts to a filtration grading for each non-zero homology class $\alpha \in \operatorname{JF}(\gamma, r, \theta)$, called the spectral invariant of $\alpha$. Recalling that Jordan Floer homology has a single homology class in each of the degrees it is supported, we obtain two spectral invariants $l_1(\gamma, r, \theta)$ and $l_2(\gamma, r, \theta)$. We can use these spectral invariants to obstruct \emph{shrinkout} (see Section \ref{sec:shrinkout}).

To introduce the symplectic action, we first introduce the space on which it is defined. We define the path space $\Omega(\mathbb{C}^2, \gamma \times \gamma)$ as the set of paths $\sigma:[0,1] \rightarrow \mathbb{C}^2$ with endpoints in $\gamma \times \gamma$. Unfortunately, the notion of a preferred capping (see Definition \ref{def:preferred_capping}) cannot be extended to all such paths. To deal with this problem, it is natural to introduce the diagonal avoiding path space \begin{equation}
    \widetilde{\Omega}(\mathbb{C}^2, \gamma \times \gamma) = \Omega\bigl(\mathbb{C}^2 \setminus \Delta(\mathbb{C}), (\gamma \times \gamma) \setminus \Delta(\gamma)\bigr) \nonumber
\end{equation}
as the diagonal avoiding subset of $\Omega(\mathbb{C}^2, \gamma \times \gamma)$. The key observation is that the definition of a preferred capping can be extended to any diagonal avoiding path, and therefore we can define the symplectic action on this space.

\begin{definition}[Symplectic action]
    For a path $\sigma \in \widetilde{\Omega}(\mathbb{C}^2, \gamma \times \gamma)$ we define its action with respect to the Hamiltonian $H_t$ by
    \begin{equation}
        \mathcal{A}_{H_t}(\sigma) = \int_0^1 H_t \circ \sigma(t) \; dt - \int_{[0,1]^2} \widehat{\sigma}^* \omega, \nonumber
    \end{equation}
    where $\widehat{\sigma}$ is a \emph{preferred capping} of $\sigma$.
\end{definition}

The action functional $\mathcal{A}_{H_t}$ plays the role of a Morse function on the path space $\widetilde{\Omega}(\mathbb{C}^2, \gamma \times \gamma)$. Its critical points correspond precisely to the non-constant trajectories of the Hamiltonian vector field $X_{H_t}$. In fact, an alternative approach to Floer homology is to attempt to apply Morse theory on the path space, using the action functional as the analogue of a Morse function. This approach is outlined by Oh in \cite[Chapter 12]{Oh:2015} and requires the Hamiltonian to be \emph{regular}, similarly to how we required the Hamiltonian to be \emph{admissible} to ensure transverse intersections of the Lagrangians.

Assuming $H_t$ is regular, the Floer differential lowers the action, filtering the Floer chain complex. The spectral invariants are then defined as the filtration grading of a given homology class in the induced filtration on the Floer homology. In the case where $H_t$ is not regular, spectral invariants can still be obtained by considering limits of sequences of spectral invariants associated with regular Hamiltonians $H_t^n$ limiting to $H_t$. Moreover, such spectral invariants still represent actions of trajectories of the Hamiltonian vector field, however, these trajectories may not correspond to transverse intersections of the Lagrangians. A detailed introduction to spectral invariants is provided in \cite{Leclercq-Zapolsky:2015}.

\subsection{Actions of inscribed isosceles trapezoids} \label{sec:actions}

Recall that inscriptions of isosceles trapezoids in a Jordan curve $\gamma$ are in correspondence with non-constant trajectories of the Hamiltonian $H_t$. Given an inscribed isosceles trapezoid $T \in T_{r, \theta}$ with aspect ratio $r \in (0,1/2]$ and angle $\theta \in (0, \pi)$, we can therefore define the action of $T$ as
\begin{equation}
    \mathcal{A}_{r, \theta}(T) \coloneq \mathcal{A}_{H_t}(\tau_T), \nonumber
\end{equation}
where $\tau_T$ is the trajectory associated with $T$. The actions of inscribed isosceles trapezoids can be represented as signed areas bounded by the diagonals of the inscribed isosceles trapezoid and the Jordan curve. This motivates why spectral invariants may be used as a tool to obstruct \emph{shrinkout}.

For the remainder of Section \ref{sec:actions}, we mostly restrict our attention to the case $r \in (0,1/2)$, which corresponds to isosceles trapezoids that are not rectangles. The rectangular case can be treated by the same methods, and all definitions and computations extend without difficulty; however, their inclusion requires additional technical distinctions, and we therefore omit them for the sake of clarity.

\subsubsection{Action of an elegantly inscribed isosceles trapezoid} \label{sec:elegant_action}

In certain special cases of inscriptions of isosceles trapezoids, there is a simple geometric interpretation of the action. One such example is the case of \emph{elegant} inscriptions. 

\begin{definition}[Elegant inscription]
    Let $T \subset \mathbb{C}$ be an inscribed isosceles trapezoid in a Jordan curve $\gamma$ and $C \subset \mathbb{C}$ be the unique circle through $T$. We say that $T \subset \mathbb{C}$ is \emph{elegantly inscribed} in $\gamma$ if $\gamma$ can be continuously isotoped to $C$ through Jordan curves, whilst keeping $T$ fixed.
\end{definition}

\noindent We now wish to describe the action of an elegantly inscribed isosceles trapezoid in terms of areas bounded by the Jordan curve and the diagonals of the isosceles trapezoid.

Let $\tau:[0,1] \rightarrow \mathbb{C}^2$ be a non-constant trajectory of $H_t$ such that the coordinates of $\tau(0)=(z,w)$ and $\tau(1)=(z',w')$ are the vertices of an elegantly inscribed isosceles trapezoid $T \in T_{r, \theta}$ in $\gamma$. We wish to calculate $\mathcal{A}_{r,\theta}(T) \coloneq \mathcal{A}_{H_t}(\tau)$. We recall that the symplectic action $\mathcal{A}_{H_t}(\tau)$ splits into two integrals: the first being the integral of the Hamiltonian along the trajectory, and the second being the symplectic area of a preferred capping of $\tau$. We will consider these integrals separately.

To compute the first integral, note that the Hamiltonian $H$ (recall $H_t = \beta(t)H$ with $\int_0^1 \beta(t) \, dt = 1$) is constant along the trajectory $\tau:[0,1] \rightarrow \mathbb{C}^2$. To be explicit, this follows from
\begin{equation}
    \frac{d}{dt}\big(H \circ \tau(t)\big) = dH\big(\dot{\tau}(t)\big) = dH \big(\beta(t)X_H\big) = \beta(t)\,dH(X_H) = \beta(t)\,\omega(X_H,X_H) = 0, \nonumber
\end{equation}
where we have used the fact $X_{H_t} = \beta(t) X_H$, which follows immediately from the linearity of the interior product $\iota_{X}$. We therefore have
\begin{equation}
    \int_0^1 H_t \circ \tau (t) dt = H \circ \tau(0) \int_0^1 \beta(t) \;dt = \frac{1}{2} \theta r(1-r) |z-w|^2. \nonumber
\end{equation}
Note that $|z-w|$ is simply the length of the diagonal of $T_{r, \theta}$, therefore the areas $A_1$ and $A_2$ indicated in Figure \ref{fig:elegant_action} are simply $A_1 = \tfrac{1}{2} \theta r^2 |z-w|^2$ and $A_2 = \tfrac{1}{2} \theta (1-r)^2 |z-w|^2$ respectively. After simplification, we obtain 
\begin{equation}
    \int_0^1 H_t \circ \tau (t) dt = (1-r) \cdot A_1 + r \cdot A_2. \nonumber
\end{equation}
This gives a geometric interpretation for the first integral in terms of circle segments bounded by the diagonals of the inscribed isosceles trapezoid.

For the second integral, we let $p_1:[0,1] \rightarrow \gamma$ parametrise the arc of $\gamma \setminus \{w,w'\}$ from $z$ to $z'$ and $p_2:[0,1] \rightarrow \gamma$ parametrise the arc of $\gamma \setminus \{z,z'\}$ from $w$ to $w'$, as indicated in the left diagram of Figure \ref{fig:action}. Then $p = (p_1,p_2):[0,1] \rightarrow \gamma \times \gamma$ defines a path from $(z,w)$ to $(z',w')$ in $\gamma \times \gamma$. Importantly, $\tau \cup p$ has winding number 0 around $\Delta(\mathbb{C})$, which is easy to see by considering the projection $\pi_{d}(z,w) = z-w$. Therefore, $\tau \cup p$ is the boundary of some preferred capping $\widehat{\tau}$. Writing $\pi_i:\mathbb{C}^2 \rightarrow \mathbb{C}$ for the projection onto the $i^\text{th}$ coordinate, we then calculate
\begin{eqnarray}
    \int_{[0,1]^2} \widehat{\tau}^*\omega &=&  (1-r)\int_{[0,1]^2} \widehat{\tau}^* dx_1 \wedge dy_1 + r \int_{[0,1]^2} \widehat{\tau}^* dx_2 \wedge dy_2 \nonumber\\
    &=& (1-r)\int_{[0,1]^2} (\pi_1 \circ \widehat{\tau})^* dx \wedge dy + r \int_{[0,1]^2} (\pi_2 \circ \widehat{\tau})^* dx \wedge dy.  \nonumber
\end{eqnarray}
These resulting integrals correspond to the signed areas co-bounded by $\tau_i = \pi_i(\tau)$ and $p_i$, as indicated by $A_3$ and $A_4$ respectively in Figure \ref{fig:elegant_action}. However, since we are dealing with signed areas, we must be careful with orientations. The orientation of $\widehat{\tau}$ is such that the induced orientation on the boundary $\tau \cup p$ is given by traversing $p$ forward and $\tau$ backward.

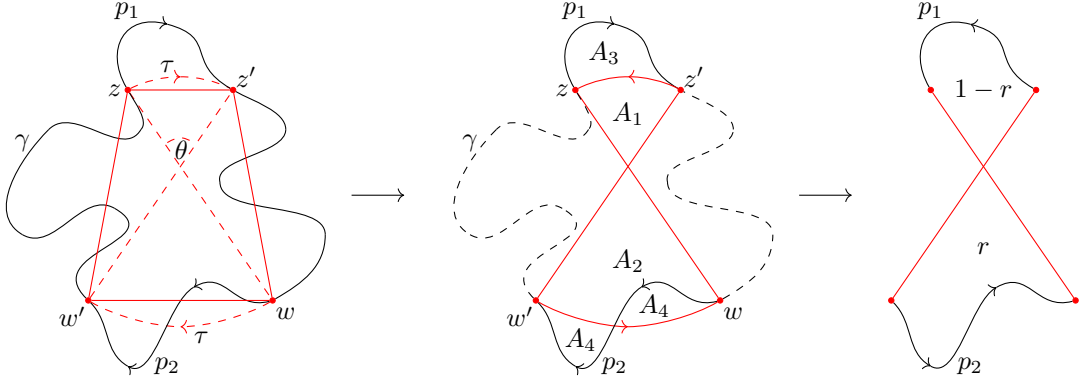
\begin{figure}[h]
    \centering

\begin{tikzpicture}[scale = 0.7]
	\begin{pgfonlayer}{nodelayer}
		\node [style=new style 1] (1) at (-2.25, -2) {};
		\node [style=new style 1] (2) at (-5.75, -2) {};
		\node [style=none] (3) at (-5.25, -0.25) {};
		\node [style=new style 1] (5) at (-5, 2) {};
		\node [style=new style 1] (6) at (-3, 2) {};
		\node [style=none] (9) at (-5, -3.25) {};
		\node [style=none] (10) at (-2.75, 2.25) {$z'$};
		\node [style=none] (11) at (-2, -2.25) {$w$};
		\node [style=none] (12) at (-6.05, -2.3) {$w'$};
		\node [style=none] (13) at (-5.25, 2) {$z$};
		\node [style=none] (14) at (-5, 3.5) {$p_1$};
		\node [style=none] (15) at (-4.25, -3.25) {$p_2$};
		\node [style=none] (16) at (-4.25, 1) {};
		\node [style=none] (17) at (-3.75, 1) {};
		\node [style=none] (18) at (-3.98, 0.87) {$\theta$};
		\node [style=none] (19) at (-4, 2.25) {};
		\node [style=none] (20) at (-4, -2.5) {};
		\node [style=none] (48) at (0.25, 0) {};
		\node [style=none] (49) at (-0.75, 0) {};
		\node [style=none] (50) at (-7, 1) {$\gamma$};
		\node [style=none] (52) at (-4.25, 2.45) {$\tau$};
		\node [style=none] (53) at (-3.6, -2.7) {$\tau$};
		\node [style=none] (54) at (-3.75, -1.75) {};
		\node [style=none] (55) at (-1.25, -0.75) {};
		\node [style=none] (57) at (-3.25, 0.25) {};
		\node [style=none] (58) at (-2, 1.25) {};
		\node [style=none] (59) at (-4.25, 3.25) {};
		\node [style=none] (61) at (-4.75, 1.25) {};
		\node [style=none] (62) at (-6.5, 1.25) {};
		\node [style=none] (63) at (-7, -0.75) {};
		\node [style=new style 1] (64) at (6.25, -2) {};
		\node [style=new style 1] (65) at (2.75, -2) {};
		\node [style=none] (66) at (3.25, -0.25) {};
		\node [style=new style 1] (67) at (3.5, 2) {};
		\node [style=new style 1] (68) at (5.5, 2) {};
		\node [style=none] (69) at (3.5, -3.25) {};
		\node [style=none] (70) at (5.75, 2.25) {$z'$};
		\node [style=none] (71) at (6.5, -2.25) {$w$};
		\node [style=none] (72) at (2.45, -2.3) {$w'$};
		\node [style=none] (73) at (3.25, 2) {$z$};
		\node [style=none] (74) at (3.5, 3.5) {$p_1$};
		\node [style=none] (75) at (4.25, -3.25) {$p_2$};
		\node [style=none] (76) at (4.25, 1) {};
		\node [style=none] (77) at (4.75, 1) {};
		\node [style=none] (79) at (4.5, 2.25) {};
		\node [style=none] (80) at (4.5, -2.5) {};
		\node [style=none] (81) at (1.5, 1) {$\gamma$};
		\node [style=none] (84) at (4.75, -1.75) {};
		\node [style=none] (85) at (7.25, -0.75) {};
		\node [style=none] (86) at (5.25, 0.25) {};
		\node [style=none] (87) at (6.5, 1.25) {};
		\node [style=none] (88) at (4.25, 3.25) {};
		\node [style=none] (89) at (3.75, 1.25) {};
		\node [style=none] (90) at (2, 1.25) {};
		\node [style=none] (91) at (1.5, -0.75) {};
		\node [style=none] (92) at (4.5, -1.25) {$A_2$};
		\node [style=none] (93) at (4.5, 1.55) {$A_1$};
		\node [style=none] (94) at (4.05, 2.7) {$A_3$};
		\node [style=none] (95) at (3.6, -2.75) {$A_4$};
		\node [style=none] (96) at (5, -2.1) {$A_4$};
		\node [style=none] (97) at (8.75, 0) {};
		\node [style=none] (98) at (7.75, 0) {};
		\node [style=new style 1] (99) at (13, -2) {};
		\node [style=new style 1] (100) at (9.5, -2) {};
		\node [style=new style 1] (102) at (10.25, 2) {};
		\node [style=new style 1] (103) at (12.25, 2) {};
		\node [style=none] (104) at (10.25, -3.25) {};
		\node [style=none] (111) at (11, 1) {};
		\node [style=none] (112) at (11.5, 1) {};
		\node [style=none] (118) at (11.5, -1.75) {};
		\node [style=none] (121) at (11, 3.25) {};
		\node [style=none] (122) at (11.25, 2) {$1-r$};
		\node [style=none] (123) at (11.25, -1) {$r$};
		\node [style=none] (124) at (10.25, 3.5) {$p_1$};
		\node [style=none] (125) at (11, -3.25) {$p_2$};
	\end{pgfonlayer}
	\begin{pgfonlayer}{edgelayer}
		\draw [in=-30, out=135, looseness=1.75] (2) to (3.center);
		\draw [style=new edge style 4] (2) to (6);
		\draw [style=new edge style 4] (5) to (1);
		\draw [style=new edge style 4, bend left, looseness=1.25] (16.center) to (17.center);
		\draw [style=new edge style 7, bend left=15, looseness=0.75] (5) to (19.center);
		\draw [style=new edge style 4, bend left=15, looseness=0.75] (19.center) to (6);
		\draw [style=new edge style 4, bend right=15, looseness=0.75] (2) to (20.center);
		\draw [style=new edge style 7, bend left=15, looseness=0.75] (1) to (20.center);
		\draw [style=new edge style 5, in=180, out=0] (49.center) to (48.center);
		\draw [in=150, out=-45] (2) to (9.center);
		\draw [style=new edge style 9, in=-150, out=-30, looseness=0.75] (9.center) to (54.center);
		\draw [style=new edge style 9, in=-165, out=30, looseness=1.25] (54.center) to (1);
		\draw [in=-90, out=30] (1) to (55.center);
		\draw [in=-90, out=105, looseness=0.75] (57.center) to (58.center);
		\draw [in=330, out=90, looseness=0.75] (58.center) to (6);
		\draw [in=90, out=-90, looseness=0.75] (57.center) to (55.center);
		\draw [in=-15, out=150, looseness=1.50] (6) to (59.center);
		\draw [in=45, out=-60, looseness=0.75] (5) to (61.center);
		\draw [in=45, out=-120] (61.center) to (62.center);
		\draw [style=new edge style 9, in=120, out=165, looseness=1.50] (59.center) to (5);
		\draw [in=150, out=-135] (62.center) to (63.center);
		\draw [in=150, out=-30] (63.center) to (3.center);
		\draw [style=new edge style 0, in=-30, out=135, looseness=1.75] (65) to (66.center);
		\draw [style=new edge style 1] (65) to (68);
		\draw [style=new edge style 1] (67) to (64);
		\draw [style=new edge style 1, bend left=15, looseness=0.75] (67) to (79.center);
		\draw [style=new edge style 8, bend right=15, looseness=0.75] (65) to (80.center);
		\draw [style=new edge style 1, bend left=15, looseness=0.75] (64) to (80.center);
		\draw [in=150, out=-45] (65) to (69.center);
		\draw [style=new edge style 9, in=-150, out=-30, looseness=0.75] (69.center) to (84.center);
		\draw [style=new edge style 9, in=-165, out=30, looseness=1.25] (84.center) to (64);
		\draw [style=new edge style 0, in=-90, out=30] (64) to (85.center);
		\draw [style=new edge style 0, in=-90, out=105, looseness=0.75] (86.center) to (87.center);
		\draw [style=new edge style 0, in=330, out=90, looseness=0.75] (87.center) to (68);
		\draw [style=new edge style 0, in=90, out=-90, looseness=0.75] (86.center) to (85.center);
		\draw [in=-15, out=150, looseness=1.50] (68) to (88.center);
		\draw [style=new edge style 0, in=45, out=-60, looseness=0.75] (67) to (89.center);
		\draw [style=new edge style 0, in=45, out=-120] (89.center) to (90.center);
		\draw [style=new edge style 9, in=120, out=165, looseness=1.50] (88.center) to (67);
		\draw [style=new edge style 0, in=150, out=-135] (90.center) to (91.center);
		\draw [style=new edge style 0, in=150, out=-30] (91.center) to (66.center);
		\draw [style=new edge style 5, in=180, out=0] (98.center) to (97.center);
		\draw [style=new edge style 1] (100) to (103);
		\draw [style=new edge style 1] (102) to (99);
		\draw [style=new edge style 5, in=150, out=-45] (100) to (104.center);
		\draw [style=new edge style 5, in=-150, out=-30, looseness=0.75] (104.center) to (118.center);
		\draw [in=-165, out=30, looseness=1.25] (118.center) to (99);
		\draw [style=new edge style 5, in=-15, out=150, looseness=1.50] (103) to (121.center);
		\draw [in=120, out=165, looseness=1.50] (121.center) to (102);
		\draw [style=new edge style 8, bend right=15, looseness=0.75] (68) to (79.center);
        \draw [style=new edge style 1] (5) to (2);
		\draw [style=new edge style 1] (2) to (1);
		\draw [style=new edge style 1] (1) to (6);
		\draw [style=new edge style 1] (5) to (6);
	\end{pgfonlayer}
\end{tikzpicture}

    \caption{The left diagram is of an elegantly inscribed isosceles trapezoid, with the associated vertices and paths (with orientations) in the above discussion labelled. The middle diagram indicates the areas $A_1$, $A_2$ of the circle segments relating to the integral $\int_0^1 H \circ \tau(t) dt$ and the signed areas $A_3$, $A_4$ co-bounded by $\tau_i \cup p_i$ relating to $\int \widehat{\tau}^* \omega$. The right diagram shows the final geometric interpretation of the action after all signs and orientations are considered.}
    \label{fig:elegant_action}
\end{figure}

Combining the expressions for the two integrals in the action, we find that the action of an elegantly inscribed isosceles trapezoid $T \in T_{r,\theta}$ is given by 
\begin{equation}
    \mathcal{A}_{r, \theta}(T) = (1-r)(A_1 - A_3) + r (A_2 - A_4). \nonumber
\end{equation}
This gives a simple geometric description of the action of an elegantly inscribed rectangle as a \emph{weighted double ice-cream area}, illustrated in the right diagram of Figure \ref{fig:elegant_action}. The \textit{ice-cream} with the smaller \textit{cone} is weighted by $1-r$, while the \textit{ice-cream} with the larger \textit{cone} is weighted by $r$.

\paragraph{Special Cases}

Before we continue to consider actions of other types of inscribed isosceles trapezoids, we note some important special cases of the action in the elegant inscription case.
\begin{enumerate}
    \item When $r=1/2$ we recover (up to a factor due to the scaled symplectic form) that the action of an elegantly inscribed rectangle is the \textit{double ice-cream area} discussed in \cite{Greene-Lobb:Floer-homology}.
    \item For fixed $r \neq 1/2$, consider a sequence of elegantly inscribed isosceles trapezoids $T_{\theta_n} \in T_{r,\theta_n}$ associated to trajectories $\tau_{\theta_n}$ that shrink to a point as $\theta_n \rightarrow \pi$. We then have $\mathcal{A}(\tau_n) \rightarrow r \cdot \operatorname{Area}(\gamma)$, since $A_1, A_2, A_3 \rightarrow 0$ and $A_4 \rightarrow -\operatorname{Area}(\gamma)$ as $\theta_n \rightarrow \pi$.
\end{enumerate}
The requirement for the elegantly inscribed isosceles trapezoids to shrink to a point as $\theta \rightarrow \pi$ may seem overly restrictive to be useful, however, we note that this is a regular phenomenon. In particular, for convex Jordan curves, this is necessarily the case. We will return to considering limits of isosceles trapezoids as $\theta \rightarrow \pi$ in more general Jordan curves shortly.

\subsubsection{Action of an almost-elegantly inscribed isosceles trapezoid} \label{sec:almost_elegant_action}

Another important class of inscriptions of isosceles trapezoids is \emph{almost-elegant} inscriptions. To define this class of inscriptions, we introduce the notion of a \emph{sweep-around move} for Jordan curves. Consider a Jordan curve $\gamma$ in $\mathbb{C}$. By compactifying $\mathbb{C}$ by adding a point at infinity, we can view $\gamma$ as a Jordan curve in $S^2 \cong \mathbb{C} \cup \{\infty\}$ and then we can push an arc of the Jordan curve through the point $\infty$. We will call the process of pushing an arc through $\infty$ the sweep around move.

\begin{definition}[Almost-Elegant Inscription]
    Let $T \subset \mathbb{C}$ be an inscribed isosceles trapezoid in a Jordan curve $\gamma$ and $C \subset \mathbb{C}$ be the unique circle through $T$. Moreover, let $\gamma'$ be the arc of $\gamma \setminus T$ whose closure has endpoints $w, w' \in T$, the endpoints of the longer parallel edge of the trapezoid. We say that $T \subset \mathbb{C}$ is \emph{almost-elegantly inscribed} in $\gamma$ if $\gamma$ can be continuously isotoped to $C$ through Jordan curves, whilst keeping $T$ fixed and using exactly one sweep around move for the arc $\gamma'$.
\end{definition}

\begin{figure}[h]
    \centering

\begin{tabular}{cc}
\begin{tikzpicture}[scale=0.7]
	\begin{pgfonlayer}{nodelayer}
		\node [style=none] (0) at (-3, -2) {};
		\node [style=new style 1] (1) at (-3.5, -2.75) {};
		\node [style=new style 1] (2) at (-7, -2.75) {};
		\node [style=none] (3) at (-7.75, -1) {};
		\node [style=none] (4) at (-7.75, 0.25) {};
		\node [style=new style 1] (5) at (-6.25, 1.25) {};
		\node [style=new style 1] (6) at (-4.25, 1.25) {};
		\node [style=none] (7) at (-3, -0.5) {};
		\node [style=none] (8) at (-5.25, 2) {};
		\node [style=none] (9) at (-2.5, -3.25) {};
		\node [style=none] (10) at (-4, 1.5) {$z'$};
		\node [style=none] (11) at (-3.25, -3) {$w$};
		\node [style=none] (12) at (-7.25, -3) {$w'$};
		\node [style=none] (13) at (-6.4, 1.4) {$z$};
		\node [style=none] (16) at (-5.5, 0.25) {};
		\node [style=none] (17) at (-5, 0.25) {};
		\node [style=none] (49) at (-2.75, 0.25) {$\gamma$};
		\node [style=none] (53) at (-1.75, -1) {};
		\node [style=none] (54) at (-2.5, 1.25) {};
		\node [style=none] (55) at (-4.5, 3) {};
		\node [style=none] (56) at (-6.75, 1.75) {};
		\node [style=none] (57) at (-8.75, 0.75) {};
		\node [style=none] (58) at (-8.5, -1.5) {};
		\node [style=none] (59) at (-8, -3.5) {};
	\end{pgfonlayer}
	\begin{pgfonlayer}{edgelayer}
		\draw [in=75, out=-60, looseness=1.75] (6) to (7.center);
		\draw [in=45, out=-105, looseness=1.75] (7.center) to (0.center);
		\draw [in=75, out=-135] (0.center) to (1);
		\draw [in=-15, out=120, looseness=1.50] (2) to (3.center);
		\draw [in=-150, out=165] (3.center) to (4.center);
		\draw [in=-120, out=30, looseness=1.25] (4.center) to (5);
		\draw [in=-105, out=-135, looseness=1.75] (9.center) to (1);
		\draw [in=-120, out=45, looseness=1.50] (5) to (8.center);
		\draw [in=60, out=120, looseness=1.50] (6) to (8.center);
		\draw [style=new edge style 1] (5) to (2);
		\draw [style=new edge style 1] (2) to (1);
		\draw [style=new edge style 1] (1) to (6);
		\draw [style=new edge style 1] (5) to (6);
		\draw [style=new edge style 4] (2) to (6);
		\draw [style=new edge style 4] (5) to (1);
		\draw [in=-45, out=60, looseness=1.25] (9.center) to (53.center);
		\draw [in=0, out=135, looseness=0.75] (53.center) to (54.center);
		\draw [in=-15, out=165] (54.center) to (55.center);
		\draw [in=60, out=165] (55.center) to (56.center);
		\draw [in=60, out=-135] (56.center) to (57.center);
		\draw [in=135, out=-120] (57.center) to (58.center);
		\draw [in=120, out=-45] (58.center) to (59.center);
		\draw [in=-60, out=-60, looseness=1.50] (59.center) to (2);
	\end{pgfonlayer}
\end{tikzpicture}

&

\begin{tikzpicture}[scale=0.6]
	\begin{pgfonlayer}{nodelayer}
		\node [style=none] (60) at (8, -3.75) {};
		\node [style=new style 1] (61) at (7.25, -2.75) {};
		\node [style=new style 1] (62) at (3.75, -2.75) {};
		\node [style=none] (63) at (3, -1) {};
		\node [style=none] (64) at (3, 0.25) {};
		\node [style=new style 1] (65) at (4.5, 1.25) {};
		\node [style=new style 1] (66) at (6.5, 1.25) {};
		\node [style=none] (67) at (7.5, 1.5) {};
		\node [style=none] (68) at (5.5, 2) {};
		\node [style=none] (69) at (5, -3.75) {};
		\node [style=none] (70) at (6.8, 1.5) {$z'$};
		\node [style=none] (71) at (7.5, -3) {$w$};
		\node [style=none] (72) at (3.4, -3) {$w'$};
		\node [style=none] (73) at (4.3, 1.42) {$z$};
		\node [style=none] (76) at (5.25, 0.25) {};
		\node [style=none] (77) at (5.75, 0.25) {};
		\node [style=none] (84) at (6.75, 3) {};
		\node [style=none] (85) at (4.25, 2.5) {};
		\node [style=none] (86) at (2.75, 1.25) {};
		\node [style=none] (87) at (1.75, -1.25) {};
		\node [style=none] (88) at (3.5, -4) {};
		\node [style=none] (90) at (5.5, -5.25) {};
    \end{pgfonlayer}
	\begin{pgfonlayer}{edgelayer}
		\draw [in=-90, out=-60, looseness=1.50] (66) to (67.center);
		\draw [in=45, out=60, looseness=1.25] (60.center) to (61);
		\draw [in=-15, out=135, looseness=1.75] (62) to (63.center);
		\draw [in=-150, out=165] (63.center) to (64.center);
		\draw [in=-120, out=30, looseness=1.25] (64.center) to (65);
		\draw [in=120, out=-45, looseness=1.25] (62) to (69.center);
		\draw [in=-135, out=-60, looseness=1.25] (69.center) to (61);
		\draw [in=-120, out=45, looseness=1.50] (65) to (68.center);
		\draw [in=60, out=120, looseness=1.50] (66) to (68.center);
		\draw [style=new edge style 1] (65) to (62);
		\draw [style=new edge style 1] (62) to (61);
		\draw [style=new edge style 1] (61) to (66);
		\draw [style=new edge style 1] (65) to (66);
		\draw [style=new edge style 4] (62) to (66);
		\draw [style=new edge style 4] (65) to (61);
		\draw [in=-45, out=90, looseness=1.25] (67.center) to (84.center);
		\draw [in=45, out=135] (84.center) to (85.center);
		\draw [in=15, out=-135] (85.center) to (86.center);
		\draw [in=135, out=-165, looseness=1.25] (86.center) to (87.center);
		\draw [in=165, out=-45, looseness=1.25] (87.center) to (88.center);
		\draw [in=-120, out=0, looseness=0.75] (90.center) to (60.center);
		\draw [in=180, out=-15] (88.center) to (90.center);
	\end{pgfonlayer}
\end{tikzpicture}
\end{tabular}

    \caption{The left diagram is of an almost-elegantly inscribed isosceles trapezoid where the arc between $w$ and $w'$ has to be swept around to isotope $\gamma$ to the trapezoid. The isosceles trapezoid in the right diagram is not almost-elegantly inscribed since the arc between $w$ and $z'$ has to be swept around to isotope $\gamma$ to the trapezoid.}
    \label{fig:almost_elegant_inscription}
\end{figure}
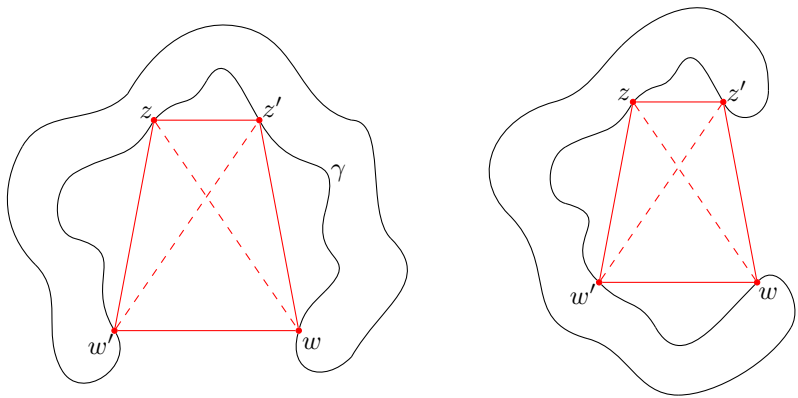

To calculate the action of an almost-elegantly inscribed isosceles trapezoid, we follow a very similar process to that of calculating the action of an elegantly inscribed isosceles trapezoid. The main difference in this setting is that we must take more care when defining the path $p$ to ensure that the winding number of $\tau \cup p$ around $\Delta(\gamma)$ is zero. 


We let $p_1:[0,1] \rightarrow \gamma$ parametrise the arc of $\gamma$ from $z$ to $z'$ through $w'$ then $w$ and let $p_2:[0,1] \rightarrow \gamma$ parametrise the arc of $\gamma$ from $w$ to $w'$ through $z'$ and $z$. These curves are represented on the right side of Figure \ref{fig:almost_elegant_action}. Note that we must be careful to parametrise $p_1$ and $p_2$ so that $p = (p_1,p_2):[0,1] \rightarrow \gamma \times \gamma$ stays away from the diagonal; however, it is easy to see that this is always possible. Given such a parametrisation, it is also easy to verify (by considering the projection $\pi_d$) that $\tau \cup p$ has winding number $0$ around the diagonal $\Delta(\gamma)$ and therefore the preferred capping class $[\widehat{\tau}]$ of $\tau$ contains a representative $\widehat{\tau}$ co-bounded by $\tau$ and $p$.

We may then repeat the computations for elegantly inscribed isosceles trapezoids in the almost-elegantly inscribed isosceles trapezoid setting. This obtains a geometric description of the action of an almost-elegantly inscribed isosceles trapezoid as a weighted sum of the signed areas of the following regions: the small cone (weight $1-r$), the large cone (weight $r$), the region bounded by $\tau_1 \cup p_1$ (weight $1-r$), and the region bounded by $\tau_2 \cup p_2$ (weight $r$). Gluing the regions with the same weight together, while keeping track of the orientations, yields a representation of the action as the weighted sum of the signed areas of the regions indicated in Figure \ref{fig:almost_elegant_action} with the indicated orientations.

\begin{figure}[h]
    \centering

\begin{tikzpicture}[scale=0.6]
	\begin{pgfonlayer}{nodelayer}
		\node [style=none] (0) at (-2.75, -1.75) {};
		\node [style=new style 1] (1) at (-3.25, -2.5) {};
		\node [style=new style 1] (2) at (-6.75, -2.5) {};
		\node [style=none] (3) at (-7.5, -0.75) {};
		\node [style=none] (4) at (-7.5, 0.5) {};
		\node [style=new style 1] (5) at (-6, 1.5) {};
		\node [style=new style 1] (6) at (-4, 1.5) {};
		\node [style=none] (7) at (-2.75, -0.25) {};
		\node [style=none] (8) at (-5, 2.25) {};
		\node [style=none] (9) at (-2.25, -3) {};
		\node [style=none] (10) at (-3.7, 1.75) {$z'$};
		\node [style=none] (11) at (-3, -2.75) {$w$};
		\node [style=none] (12) at (-7.15, -2.75) {$w'$};
		\node [style=none] (13) at (-6.2, 1.7) {$z$};
		\node [style=none] (14) at (-5.25, 0.5) {};
		\node [style=none] (15) at (-4.75, 0.5) {};
		\node [style=none] (16) at (-5, 0.25) {$\theta$};
		\node [style=none] (17) at (-5, 1.75) {};
		\node [style=none] (18) at (-5, -3) {};
		\node [style=none] (19) at (-2.5, 0.5) {$\gamma$};
		\node [style=none] (20) at (-4.75, 2) {$\tau$};
		\node [style=none] (21) at (-4.6, -3.25) {$\tau$};
		\node [style=none] (22) at (-1.5, -0.75) {};
		\node [style=none] (23) at (-2.25, 1.5) {};
		\node [style=none] (24) at (-4.25, 3.25) {};
		\node [style=none] (25) at (-6.5, 2) {};
		\node [style=none] (26) at (-8.5, 1) {};
		\node [style=none] (27) at (-8.25, -1.25) {};
		\node [style=none] (28) at (-7.75, -3.25) {};
		\node [style=none] (29) at (-1, 0) {};
		\node [style=none] (30) at (0.5, 0) {};
		\node [style=none] (36) at (5.75, 1.5) {};
		\node [style=none] (37) at (7.75, 1.5) {};
		\node [style=none] (47) at (6.75, 0) {};
		\node [style=none] (50) at (16, -2.5) {};
		\node [style=none] (51) at (12.5, -2.5) {};
		\node [style=none] (62) at (14.25, 0) {};
		\node [style=none] (65) at (16.5, -1.75) {};
		\node [style=none] (66) at (16, -2.5) {};
		\node [style=none] (67) at (12.5, -2.5) {};
		\node [style=none] (68) at (11.75, -0.75) {};
		\node [style=none] (69) at (11.75, 0.5) {};
		\node [style=none] (70) at (13.25, 1.5) {};
		\node [style=none] (71) at (15.25, 1.5) {};
		\node [style=none] (72) at (16.5, -0.25) {};
		\node [style=none] (73) at (14.25, 2.25) {};
		\node [style=none] (82) at (14.25, 1.75) {};
		\node [style=none] (84) at (9, -1.75) {};
		\node [style=none] (85) at (8.5, -2.5) {};
		\node [style=none] (86) at (5, -2.5) {};
		\node [style=none] (87) at (4.25, -0.75) {};
		\node [style=none] (88) at (4.25, 0.5) {};
		\node [style=none] (89) at (5.75, 1.5) {};
		\node [style=none] (90) at (7.75, 1.5) {};
		\node [style=none] (91) at (9, -0.25) {};
		\node [style=none] (93) at (9.5, -3) {};
		\node [style=none] (106) at (10.25, -0.75) {};
		\node [style=none] (107) at (9.5, 1.5) {};
		\node [style=none] (108) at (7.5, 3.25) {};
		\node [style=none] (109) at (5.25, 2) {};
		\node [style=none] (110) at (3.25, 1) {};
		\node [style=none] (111) at (3.5, -1.25) {};
		\node [style=none] (112) at (4, -3.25) {};
		\node [style=none] (113) at (8.9, 2.5) {$p_1$};
		\node [style=none] (114) at (15.35, 2.25) {$p_2$};
		\node [style=none] (115) at (1.75, 0) {$(1-r)$};
		\node [style=none] (118) at (10.75, 0) {$+\;r$};
		\node [style=none] (119) at (-0.25, 0.5) {$\mathcal{A}$};
	\end{pgfonlayer}
	\begin{pgfonlayer}{edgelayer}
		\draw [in=75, out=-60, looseness=1.75] (6) to (7.center);
		\draw [in=45, out=-105, looseness=1.75] (7.center) to (0.center);
		\draw [in=75, out=-135] (0.center) to (1);
		\draw [in=-15, out=120, looseness=1.50] (2) to (3.center);
		\draw [in=-150, out=165] (3.center) to (4.center);
		\draw [in=-120, out=30, looseness=1.25] (4.center) to (5);
		\draw [in=-105, out=-135, looseness=1.75] (9.center) to (1);
		\draw [in=-120, out=45, looseness=1.50] (5) to (8.center);
		\draw [in=60, out=120, looseness=1.50] (6) to (8.center);
		\draw [style=new edge style 1] (5) to (2);
		\draw [style=new edge style 1] (2) to (1);
		\draw [style=new edge style 1] (1) to (6);
		\draw [style=new edge style 1] (5) to (6);
		\draw [style=new edge style 4] (2) to (6);
		\draw [style=new edge style 4] (5) to (1);
		\draw [style=new edge style 4, bend left, looseness=1.25] (14.center) to (15.center);
		\draw [style=new edge style 7, bend left=15, looseness=0.75] (5) to (17.center);
		\draw [style=new edge style 4, bend left=15, looseness=0.75] (17.center) to (6);
		\draw [style=new edge style 4, bend right=15, looseness=0.75] (2) to (18.center);
		\draw [style=new edge style 7, bend left=15, looseness=0.75] (1) to (18.center);
		\draw [in=-45, out=60, looseness=1.25] (9.center) to (22.center);
		\draw [in=0, out=135, looseness=0.75] (22.center) to (23.center);
		\draw [in=-15, out=165] (23.center) to (24.center);
		\draw [in=60, out=165] (24.center) to (25.center);
		\draw [in=60, out=-135] (25.center) to (26.center);
		\draw [in=135, out=-120] (26.center) to (27.center);
		\draw [in=120, out=-45] (27.center) to (28.center);
		\draw [in=-60, out=-60, looseness=1.50] (28.center) to (2);
		\draw [style=new edge style 5] (29.center) to (30.center);
		\draw [style=new edge style 4] (36.center) to (47.center);
		\draw [style=new edge style 4] (47.center) to (37.center);
		\draw [style=new edge style 4] (51.center) to (62.center);
		\draw [style=new edge style 4] (62.center) to (50.center);
		\draw [style=new edge style 5, in=75, out=-60, looseness=1.75] (71.center) to (72.center);
		\draw [in=45, out=-105, looseness=1.75] (72.center) to (65.center);
		\draw [in=75, out=-135] (65.center) to (66.center);
		\draw [in=-15, out=120, looseness=1.50] (67.center) to (68.center);
		\draw [style=new edge style 5, in=-150, out=165] (68.center) to (69.center);
		\draw [in=-120, out=30, looseness=1.25] (69.center) to (70.center);
		\draw [style=new edge style 5, in=-120, out=45, looseness=1.50] (70.center) to (73.center);
		\draw [in=60, out=120, looseness=1.50] (71.center) to (73.center);
		\draw [in=75, out=-60, looseness=1.75] (90.center) to (91.center);
		\draw [style=new edge style 5, in=45, out=-105, looseness=1.75] (91.center) to (84.center);
		\draw [in=75, out=-135] (84.center) to (85.center);
		\draw [style=new edge style 5, in=-15, out=120, looseness=1.50] (86.center) to (87.center);
		\draw [in=-150, out=165] (87.center) to (88.center);
		\draw [in=-120, out=30, looseness=1.25] (88.center) to (89.center);
		\draw [in=-105, out=-135, looseness=1.75] (93.center) to (85.center);
		\draw [in=-45, out=60, looseness=1.25] (93.center) to (106.center);
		\draw [style=new edge style 5, in=0, out=135, looseness=0.75] (106.center) to (107.center);
		\draw [in=-15, out=165] (107.center) to (108.center);
		\draw [in=60, out=165] (108.center) to (109.center);
		\draw [style=new edge style 5, in=60, out=-135] (109.center) to (110.center);
		\draw [in=135, out=-120] (110.center) to (111.center);
		\draw [in=120, out=-45] (111.center) to (112.center);
		\draw [in=-60, out=-60, looseness=1.50] (112.center) to (86.center);
	\end{pgfonlayer}
\end{tikzpicture}

    \caption{The left of the diagram shows an almost-elegantly inscribed isosceles trapezoid where the arc between $w$ and $w'$ has to be swept around to isotope $\gamma$ to the trapezoid. The right hand side represents the corresponding action given by a weighted sum of the areas bounded by $p$ under the projections to the first and second complex coordinate and the correspondingly weighted cones. Note that the sign of the areas is dependent on the orientation of the curve.}
    \label{fig:almost_elegant_action}
\end{figure}
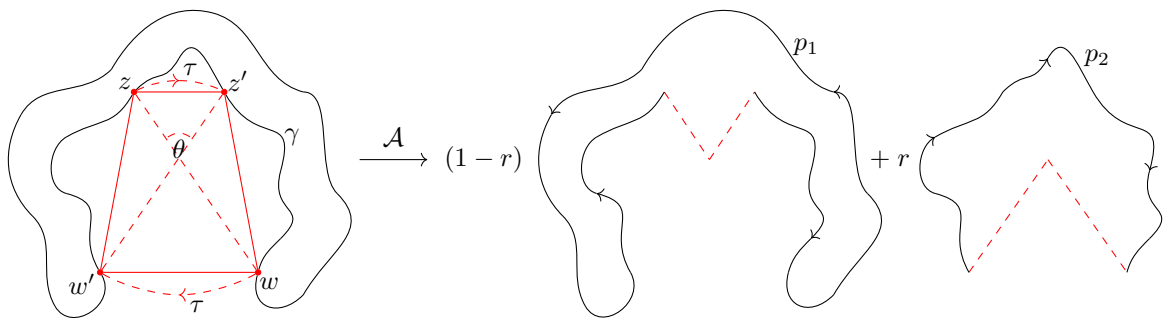

\subsubsection{Action limit for large angles}

We now return to considering the limits of sequences of inscribed isosceles trapezoids $T_{\theta_n}$ in a given Jordan curve $\gamma$ and the associated actions $\mathcal{A}_{r,\theta_n}(T_{\theta_n})$ as $\theta_n \rightarrow \pi$. For $r=1/2$, the case of rectangles, such sequences of inscribed rectangles in a smooth Jordan curve $\gamma$ always limit to binormals, whose action equals half the area of the region bounded by the Jordan curve. For $r \in (0,1/2)$, there are two distinct possibilities for the limit.

\paragraph{Shrinkout}

The first possibility is that the isosceles trapezoids shrink to a point in the limit, that is, $T_{\theta_n} \rightarrow p \in \gamma$. Unlike the case of shrinkout when varying the Jordan curve, this variety of shrinkout does not pose any problems. In fact, this is the simpler case to deal with, largely due to the following result.

\begin{lemma} \label{lemma:elegant_in_limit}
    Let $\gamma$ be a smooth Jordan curve and $T_n \in T_{r, \theta_n}$ be a sequence of inscribed isosceles trapezoids in $\gamma$ of fixed aspect ratio $r \in (0, 1/2)$ limiting to a point, that is, $T_n \rightarrow p \in \gamma$. Then $\theta_n \rightarrow \pi$ and there exists an $N>0$ such that for all $n>N$ the inscribed isosceles trapezoids $T_{\theta_n}$ are either all elegantly inscribed or all almost-elegantly inscribed. 
\end{lemma}

\begin{proof}

    Since $\gamma$ is smooth, it admits a smooth curvature function. The motivation of the argument is that if there existed an inscribed isosceles trapezoid arbitrarily close to $p$ that was not elegantly or almost-elegantly inscribed of an angle arbitrarily close to $\pi$ then the curvature function would necessarily blow up at $p$. We now formalise this argument.

    Fix a unit speed parametrisation $x+iy=\gamma \colon S^1 = \mathbb{R}/L \to \mathbb{C}$ of $\gamma$, where $L$ is the length of $\gamma$, and denote the curvature function by $\kappa\colon S^1\to\mathbb{R}$. By compactness, $|\kappa(t)|$ has a maximum, which we call $K$. Let $D(p,\varepsilon) \subset \mathbb{C}$ be the open ball of radius $\varepsilon$ around $p$. Since $T_n \rightarrow p$, for any $\varepsilon>0$ there exists an $N_\varepsilon > 0$ such that $T_n \subset \widetilde{\gamma}_\varepsilon = \gamma \cap D(p,\varepsilon)$. Moreover, for $\varepsilon$ sufficiently small $\widetilde{\gamma}_\varepsilon$ consists of a single connected component, whose length we denote by $l_\varepsilon$. Since the integral of curvature gives the angle through which the tangent turns, the tangent of $\widetilde{\gamma}_\varepsilon$ turns through a maximum angle of $l_\varepsilon \cdot K$. Note that $l_\varepsilon \cdot K \rightarrow 0$ as $\varepsilon \rightarrow 0$.
    
    By Cauchy's mean value theorem on the real and imaginary component functions of $\gamma:S^1 \rightarrow \mathbb{C}$, if $\gamma(t_1)=p_1$ and $\gamma(t_2)=p_2$ for $t_1 < t_2$, then there exists a $t\in(t_1, t_2)$ such that $\arg(\gamma'(t)) = \arg(p_2-p_1)$. Therefore, if $\gamma$ were to pass through the three points $p_1, p_2, p_3$ in that order, then the tangent vector of $\gamma$ must turn through an angle of at least $d_\varphi(p_2-p_1,p_3-p_2)$, where $d_\varphi(z_1, z_2)$ is the (absolute) angle between the complex numbers $z_1$ and $z_2$.
    
    Label the points $T_n=\{T^1_n,T^2_n,T^3_n,T^4_n\}$ in an anticlockwise order such that $T^1_n$ and $T^4_n$ are the endpoints of the longer of the parallel edges of the isosceles trapezoid. To establish that $T_n$ is elegantly or almost-elegantly inscribed, it is sufficient to show that $\widetilde{\gamma}_\varepsilon$ passes through the points in the order $T^1_n, \ldots, T^4_n$ or its reverse. Once we recall the tangent of $\widetilde{\gamma}_\varepsilon$ turns through a maximum angle of $l_\varepsilon \cdot K$, by the above discussion this is clear since passing through the points in any other order would require turning through an angle greater than $\pi/2$. Moreover, by an elementary geometric calculation
    \begin{equation}
        \varphi_n=|\arg(T^3_n-T^2_n)-\arg(T^2_n-T^1_n)| = \arctan \left( \frac{1}{1-2r} \cot\left(\frac{\theta_n}{2} \right) \right). \nonumber
    \end{equation}
    Since $\widetilde{\gamma}_\varepsilon$ passes through the points $T_n^1$, $T_n^2$, $T_n^3$ in that order, we require $\varphi_n \leq l_\varepsilon \cdot K$, or equivalently $1 \leq (1-2r)\tan(\theta_n/2)\tan(l_\varepsilon \cdot K)$. We obtain the same inequality when $\gamma$ passes through $T_n$ in the reverse of this order. Taking the limit $\varepsilon \rightarrow 0$ gives $\theta_n \rightarrow \pi$ as claimed.
\end{proof}

\noindent We recall that elegantly inscribed and almost-elegantly inscribed isosceles trapezoids are the two special cases for which we made the computation of the action explicit. Therefore, it is easy to consider the limit of the actions corresponding to the isosceles trapezoids $T_n$, giving the following immediate consequence of Lemma \ref{lemma:elegant_in_limit}.

\begin{corollary}
    Let $\gamma$ be a smooth Jordan curve and $T_n \in T_{r, \theta_n}$ be a sequence of inscribed isosceles trapezoids in $\gamma$ of fixed aspect ratio $r \in (0, 1/2)$ with limit $T_n \rightarrow p \in \gamma$, then the action associated with the inscribed isosceles trapezoids limits to either 
    \begin{equation}
        \mathcal{A}_{r, \theta}(\tau_n) \rightarrow r \cdot \operatorname{Area}(\gamma) \quad \text{or} \quad \mathcal{A}_{r, \theta}(\tau_n) \rightarrow (1-r) \cdot \operatorname{Area}(\gamma). \nonumber
    \end{equation}
\end{corollary}

\begin{proof}
    By Lemma \ref{lemma:elegant_in_limit}, for sufficiently large $n$ the inscribed isosceles trapezoids are either elegantly inscribed or almost-elegantly inscribed. In the case of elegantly inscribed isosceles trapezoids, it is clear from Figure \ref{fig:elegant_action} that the action must limit to $r \cdot \operatorname{Area}(\gamma)$. Similarly, in the case of almost-elegantly inscribed isosceles trapezoids, it is clear from Figure \ref{fig:almost_elegant_action} that the action must limit to $(1-r) \cdot \operatorname{Area}(\gamma)$.
\end{proof}

\noindent Another interesting consequence of Lemma \ref{lemma:elegant_in_limit} is that we can say more about what points $p \in \gamma$ the sequence of isosceles trapezoids $T_n$ can limit to. 

\begin{corollary}
    Let $\gamma$ be a smooth Jordan curve and $T_n \in T_{r, \theta_n}$ be a sequence of inscribed isosceles trapezoids in $\gamma$ of fixed aspect ratio $r \in (0, 1/2)$ with limit $T_n \rightarrow p \in \gamma$, then $p$ is a vertex of $\gamma$, that is, the derivative of the curvature of $\gamma$ vanishes at $p$.
\end{corollary}

\begin{proof}
    Again, by Lemma \ref{lemma:elegant_in_limit}, for sufficiently large $n$ the inscribed isosceles trapezoids are either elegantly inscribed or almost-elegantly inscribed. Assume without loss of generality that $T_n$ is elegantly inscribed, that is, $\gamma$ passes through the points of $T_n$ in the order $T_n^1,\ldots,T_n^4$ (with the same ordering as Lemma \ref{lemma:elegant_in_limit}), and let $t_n^i \in \mathbb{R}$ be such that $\gamma(t_n^i) = T_n^i$. To establish the result, it is therefore sufficient to show that $\widetilde{\gamma}:[t_1,t_4] \rightarrow \mathbb{C}$ contains a vertex. 

    We first note that there is a unique circle $C_n$ that passes through $T_n$, and we denote its radius by $R_n$. At each $T_n^i \in C_n$ there is a unique tangent circle to $\widetilde{\gamma}$ of radius $1/\kappa(t_n^i)$ called the osculating circle which we denote by $C_n^i$. For a contradiction, we assume without loss of generality that $\kappa:[t_1,t_4] \rightarrow \mathbb{R}$ is monotone increasing. The key observation is that $\widetilde{\gamma}$ must remain in the bounded region of $\mathbb{C} \setminus C_n$ for all $t > t_i$ and in the unbounded region of $\mathbb{C} \setminus C_n$ for all $t < t_i$. We then split our argument into two main cases.
    \begin{enumerate}
        \item If $\widetilde{\gamma}'(t_n^2)$ points inside $C_n$, then as $T_n^1$ lies on $\widetilde{\gamma}$, $C_n^2$ must intersect $C_n$ on $C_n^{12}$ the shorter arc of $C_n$ between $T_n^1$ and $T_n^2$ and therefore the radius of $C_n^i$ must be less than $R_i$. The curve $\widetilde{\gamma}$ must then remain within $C_n^2$ for all $t \geq t_2$, contradicting the fact that $T_n^3$ lies on $C_n \setminus C_n^{12}$.
        \item If $\gamma'(t_n^2)$ points outside $C_n$, then for some $t \in (t_2, t_3]$ we have $\widetilde{\gamma}(t) \in C_n$ and $\widetilde{\gamma}'(t)$ points inside $C_n$. The arguments of Case 1 may then be repeated to obtain a contradiction of where $T_n^4$ must lie on $C_n$.
    \end{enumerate}
    In the above cases, we have omitted the possibility that $\widetilde{\gamma}'(t_n^2)$ is tangential to $C_n$. If $1/\kappa(t_n^2) \leq R_n$ then the result is immediate as $\widetilde{\gamma}$ must remain in the bounded region of $\mathbb{C} \setminus C_n$ for all $t > t_n^2$, and if $1/\kappa(t_n^2) > R_n$ then one can repeat the argument of Case 2.
\end{proof}

\paragraph{Quadrisecants}

The second possibility is that the isosceles trapezoids limit to a \emph{quadrisecant}, that is, the vertices of the isosceles trapezoid limit to four different points that lie on a straight line. Each quadrisecant $Q \in T_{r,\pi}$ is the limit of two disjoint one parameter families of inscriptions of isosceles trapezoids $T^i_\theta$ parametrised by $\theta \in (\pi-\varepsilon, \pi)$ for $\varepsilon$ sufficiently small. An example is given in Figure \ref{fig:quadrasecant}. Therefore, each quadrisecant has two associated actions given by the limits of the actions of the inscribed isosceles trapezoids $T^i_\theta$ as $\theta \rightarrow \pi$. An interesting result is that the sum of the two actions associated with any given quadrisecant is simply the area of the Jordan curve.

Unfortunately, despite the duality result above, the action in the quadrisecant case is significantly more difficult to analyse than in the shrinkout case. In particular, an action associated with a quadrisecant can be made arbitrarily large or small simply by varying the Jordan curve, whilst keeping the area of the Jordan curve constant and keeping the points at which the quadrisecant is inscribed fixed. It is this lack of control over the action in the limit as $\theta \rightarrow \pi$ that makes the isosceles trapezoid case more difficult to approach than the rectangular case.

\begin{figure}[h]
    \centering

\begin{tikzpicture}[scale=0.8]

	\begin{pgfonlayer}{nodelayer}
		\node [style=new style 3] (8) at (-3, 1.5) {};
		\node [style=new style 3] (9) at (1.25, 0.75) {};
		\node [style=new style 3] (10) at (-1.25, 0.75) {};
		\node [style=new style 3] (11) at (3, 1.5) {};
		\node [style=new style 1] (12) at (-3, -1.5) {};
		\node [style=new style 1] (13) at (1.25, -0.75) {};
		\node [style=new style 1] (14) at (3, -1.5) {};
		\node [style=new style 1] (15) at (-1.25, -0.75) {};
		\node [style=new style 2] (16) at (-3.25, 0) {};
		\node [style=new style 2] (17) at (-1.5, 0) {};
		\node [style=new style 2] (18) at (1.5, 0) {};
		\node [style=new style 2] (19) at (3.25, 0) {};
	\end{pgfonlayer}
	\begin{pgfonlayer}{edgelayer}
		\draw [style=new edge style 3] (11) to (10);
		\draw [style=new edge style 3] (8) to (9);
		\draw (16) to (19);
		\draw [in=-90, out=105] (12) to (16);
		\draw [in=255, out=90] (16) to (8);
		\draw [in=120, out=60] (8) to (11);
		\draw [in=90, out=-75] (11) to (19);
		\draw [in=75, out=-90] (19) to (14);
		\draw [in=-120, out=-120, looseness=2.00] (14) to (13);
		\draw [in=-90, out=60] (13) to (18);
		\draw [in=-60, out=90] (18) to (9);
		\draw [in=30, out=150, looseness=1.25] (9) to (10);
		\draw [in=90, out=-135] (10) to (17);
		\draw [in=120, out=-90] (17) to (15);
		\draw [in=-60, out=-60, looseness=2.00] (15) to (12);
		\draw [style=new edge style 2] (8) to (10);
		\draw [style=new edge style 2] (10) to (9);
		\draw [style=new edge style 2] (9) to (11);
		\draw [style=new edge style 2] (8) to (11);
		\draw [style=new edge style 1] (12) to (14);
		\draw [style=new edge style 1] (14) to (13);
		\draw [style=new edge style 1] (13) to (15);
		\draw [style=new edge style 1] (15) to (12);
		\draw [style=new edge style 4] (14) to (15);
		\draw [style=new edge style 4] (12) to (13);
	\end{pgfonlayer}

\end{tikzpicture}

    \caption{An example of a quadrisecant. In this case, the quadrisecant (black) can be viewed as the limit of a sequence $T_n \in T_{r, \theta_n}$ of either elegantly inscribed (blue) or almost-elegantly inscribed (red) isosceles trapezoids as $\theta_n \rightarrow \infty$.}
    \label{fig:quadrasecant}
\end{figure}
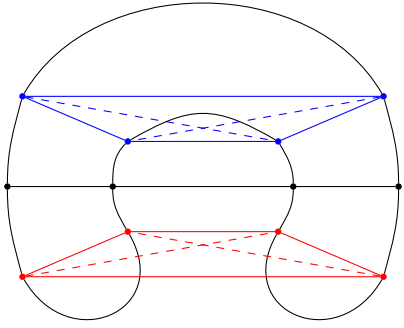

\subsection{Spectral invariants}

We recall that the action functional can be viewed as the analogue of the Morse function in Morse theory. In Morse theory, the flow lines counted in defining the differential always point in the direction of decreasing function value. Similarly, the Jordan Floer differential $\partial_{\mathrm{JF}}$ decreases the action. Consequently, Jordan Floer homology $\mathrm{JF}(\gamma,r,\theta)$ is filtered by the action. We can then define the spectral invariant $l_{\alpha}(\gamma,r,\theta)$ of a homology class $\alpha$ as the lowest filtration grading in which a representative of $\alpha$ is supported.

To be more precise, the chain complex $\operatorname{JFC}(\gamma,r,\theta,h_t)$ is filtered by the action $\mathcal{A}_{H_t+h_t}:\mathcal{G}(\gamma, r, \theta, h_t) \rightarrow \mathbb{R}$, that is, we have subgroups $\operatorname{JFC}^a(\gamma,r,\theta,h_t) \subseteq \operatorname{JFC}(\gamma,r,\theta,h_t)$ generated by the elements $\tau \in \mathcal{G}(\gamma, r, \theta, h_t)$ such that $\mathcal{A}_{H_t+h_t}(\tau) \leq a$. Moreover, for any admissible complex structure $J_t$, the differential $\partial_{(\gamma, r, \theta, h_t, J_t)}$ respects the filtration in the sense that if a trajectory $\tau'$ of $H_t+h_t$ appears with a non-zero coefficient in $\partial_{(\gamma, r, \theta, h_t, J_t)}(\tau)$ for some $\tau \in \mathcal{G}(\gamma, r, \theta, h_t)$, then $\mathcal{A}_{H_t+h_t}(\tau') < \mathcal{A}_{H_t+h_t}(\tau)$. To see this formally, we include the following standard result.

\begin{lemma}[Action-energy identity] \label{lemma:energy_of_strip}
    With the same notation used in the definition of continuation maps, let $u \in \mathcal{M}(\tau_1, \tau_2; h_{st}, J_{st})$ be a pseudo-holomorphic strip between the trajectories $\tau_1 \in \mathcal{G}(\gamma,r, \theta_1, h_t^1)$ and $\tau_2 \in \mathcal{G}(\gamma,r, \theta_2, h_t^2)$, then its energy is given by
    \begin{equation}
        E(u) = \mathcal{A}_{H^1+h_t^1}(\tau_1)-\mathcal{A}_{H^2+h_t^2}(\tau_2) + \int_\Sigma \bigl(\partial_s (H_{st}+h_{st})\bigr) \circ u(s,t) \; dt \, ds. \nonumber
    \end{equation}
\end{lemma}

\noindent It is clear from the definition of the energy of a strip that $E(u) \geq 0$, with equality if and only if the strip is constant in the $s$ direction. The result $\mathcal{A}_{H_t+h_t}(\tau') < \mathcal{A}_{H_t+h_t}(\tau)$ then follows immediately from the fact that, when defining the differential, $H_{st}$ and $h_{st}$ are constant in $s$.

Since the differential respects the filtration on $\operatorname{JFC}(\gamma,r,\theta,h_t)$, the subgroups $\operatorname{JFC}^a(\gamma,r,\theta,h_t)$ form chain complexes with the differential $\partial_{(\gamma, r, \theta, h_t, J_t)}$. Therefore, for each homology class $\alpha \in \operatorname{JF}(\gamma,r,\theta,h_t,J_t)$ we can define the associated spectral invariant 
\begin{equation}
    l_{\alpha}(\gamma,r,\theta,h_t,J_t) = \inf \{ a : \alpha \text{ is represented by a cycle in } \operatorname{JFC}^a(\gamma,r,\theta,h_t) \text{ with } \partial_{(\gamma, r, \theta, h_t, J_t)}\}. \nonumber
\end{equation}
We note that the spectral invariant must be equal to the action of some non-constant trajectory, more precisely $l_{\alpha}(\gamma,r,\theta,h_t,J_t) = \mathcal{A}_{H_t+h_t}(\tau)$ for some $\tau \in \mathcal{G}(\gamma, r, \theta, h_t)$.

However, we do not want the spectral invariant to depend on the data $h_t$ and $J_t$. Recall that continuation maps provide isomorphisms between the homology groups $\operatorname{JF}(\gamma,r,\theta, h_t, J_t)$ for different Hamiltonian perturbations $h_t$ and almost-complex structures $J_t$. Since we fix $\theta$, $H_{st}$ is constant in $s$. Therefore, by Lemma \ref{lemma:energy_of_strip}, these chain maps are filtered of filtration degree 
\begin{equation}
    \varepsilon(h_{st}) = \int_\Sigma \bigl(\partial_s h_{st}\bigr) \circ u(s,t) \; dt \, ds, \nonumber
\end{equation} 
which depends only on the Hamiltonian perturbations $h_{st}$. We can therefore immediately remove the dependence of the spectral invariant on $J_t$. Moreover, $\varepsilon(h_{st})$ can be made arbitrarily small by sending $h_{st} \rightarrow 0$ in the $C^\infty$ topology. This also forces $h^1_t, h^2_t \rightarrow 0$. Therefore, we can remove the dependence on $h_t$ by considering the limit $h_t \rightarrow 0$. We then arrive at a spectral invariant $l_\alpha(\gamma, r, \theta)$ for homology classes $\alpha \in \operatorname{JF}(\gamma, r, \theta)$.

\begin{definition}[Spectral invariants]
    Let $\gamma$ be a smooth Jordan curve, $r \in (0,1/2]$ be an aspect ratio, and $\theta \in (0,\pi)$ be an angle. Then for any homology class $\alpha_i \in \operatorname{JF}_i(\gamma,r,\theta)$  we define the \emph{spectral invariants} $l_{\alpha_i}(\gamma, r, \theta) \in \mathbb{R}$ by
    \begin{equation}
        l_{\alpha_i}(\gamma, r, \theta) = \lim_{n \rightarrow \infty} l_{\alpha_i}(\gamma,r, \theta, h_t^n, J_t^n), \nonumber
    \end{equation}
    where $h_t^n$ is a sequence of admissible Hamiltonian perturbations with $h_t^n \rightarrow 0$ and $J_t$ is a sequence of admissible almost-complex structures.
\end{definition} 

\noindent We note the abuse of notation, using $\alpha_i$ to denote both the homology class $\alpha_i \in \operatorname{JF}_i(\gamma, r, \theta)$ and the homology class $\alpha_i \in \operatorname{JF}_i(\gamma, r, \theta, h_t, J_t)$. The relationship between them is provided by the isomorphisms induced from the continuation maps.

Recalling that $\operatorname{JF}_*(\gamma, r, \theta) = (\mathbb{F}_2)_{(2)} \oplus (\mathbb{F}_2)_{(1)}$, there is only one homology class in each degree, so we can simply use the notation $l_i$ to denote $l_{\alpha_i}$. Moreover, we call $l_i$ the homological degree $i$ spectral invariant. Importantly, we note that $l_i(\gamma,r,\theta)$ is always the action of some inscribed isosceles trapezoid in $\gamma$ of aspect ratio $r$ and angle $\theta$.

We will commonly view the spectral invariants as functions of the angle $\theta$, that is, $l_i(\gamma, r, \cdot):(0,\pi) \rightarrow \mathbb{R}$. By the result in the following subsection bounding the variation of the spectral invariant in $\theta$, we can extend the spectral invariant to a continuous, monotone function
\begin{equation}
    l_i(\gamma, r, \cdot):[0, \pi] \rightarrow \mathbb{R}_{\geq 0} \nonumber
\end{equation}
with $l_i(\gamma, r, 0) = 0$. Similarly, we can generalise the spectral invariant to non-smooth Jordan curves $\gamma$ by considering a sequence of smooth Jordan curves $\gamma_n$ limiting to $\gamma$. We can then define the spectral invariant of $\gamma$ by
\begin{equation}
    l_i(\gamma, r, \theta) = \lim_{n \rightarrow \infty} l_i(\gamma_n, r, \theta). \nonumber
\end{equation}

For a detailed introduction to spectral invariants in the case of monotone Lagrangians, see \cite{Leclercq-Zapolsky:2015}.

\subsubsection{Variation of the spectral invariant under change of the angle}

We now wish to bound the variation of the spectral invariant in $\theta$. For this result, we follow \cite{Greene-Lobb:Floer-homology} closely, but we include the proof here for the convenience of the reader.

\begin{proposition} \label{prop:bounded_spectral_variation}
    Let $\gamma$ be a smooth Jordan curve of radius $\operatorname{Rad}(\gamma)$, then for any $r \in (0,1/2]$ and any $0 < \theta_1 < \theta_2 < \pi$ the spectral invariant satisfies 
    \begin{equation}
        0 < l_i(\gamma, r, \theta_2) - l_i(\gamma,r,\theta_1) \leq 2r(1-r)\operatorname{Rad}(\gamma)^2 \cdot (\theta_2 - \theta_1). \nonumber
    \end{equation}
\end{proposition}

\begin{proof}
By standard arguments of Cerf theory, we can choose a smooth path of Hamiltonian perturbations $h_t^\theta$ for $\theta \in [\theta_1, \theta_2]$ such that $h_t^{\theta_0}$ is an admissible Hamiltonian perturbation for the triple $(\gamma, r, \theta_0)$ for all but finitely many critical values $\theta_c \in [\theta_1, \theta_2]$. Importantly, such a path of Hamiltonian perturbations can be taken to be arbitrarily small with respect to the $C^k$ norm for any fixed $k \geq 0$. See \cite{Greene-Lobb:Floer-homology, Floer:1988} for the details of this argument.

Take $\theta_0$ to be a non-critical angle of $h_t^\theta$ and $\tau_\theta$ to be a continuous path of trajectories of the Hamiltonians $H_t^\theta + h_t^\theta$ (where $H_t^\theta=H_t$, but with the superscript to emphasise the dependence on theta) such that $\tau_{\theta_0}$ is a trajectory whose action is represented by the spectral invariant. Such a path of trajectories exists locally, since $h_t^\theta$ remains admissible locally. We can then calculate
\begin{eqnarray}
    \frac{d}{d\theta}l_i(\gamma, r, \theta, h^\theta_t, J^\theta_t)\Big|_{\theta=\theta_0} &=& \frac{d}{d\theta} \Bigl( \mathcal{A}_{H^\theta_t+h^\theta_t}(\tau_\theta) \Bigr) \bigg|_{\theta=\theta_0} \nonumber \\
    &=& \frac{d}{d\theta} \Bigl( \mathcal{A}_{H^\theta_t+h^\theta_t}(\tau_{\theta_0}) \Bigr) \bigg|_{\theta=\theta_0} + \frac{d}{d\theta} \Bigl( \mathcal{A}_{H^{\theta_0}_t+h^{\theta_0}_t}(\tau_\theta) \Bigr) \bigg|_{\theta=\theta_0} \nonumber
\end{eqnarray}
We note that trajectories are critical points of the action, and therefore the second term vanishes. It therefore remains to calculate the first term
\begin{eqnarray}
    \frac{d}{d\theta} \Bigl( \mathcal{A}_{H^\theta_t+h^\theta_t}(\tau_{\theta_0}) \Bigr) \bigg|_{\theta=\theta_0} = \int_0^1 \frac{d}{d\theta} \big(H^\theta_t+h^\theta_t\big) \Big|_{\theta_0} \!\!\! \circ \tau_{\theta_0}(t) \;dt - \frac{d}{d\theta} \left( \int_{[0,1]^2} \widehat{\tau}^*_{\theta_0} \omega \right) \Bigg|_{\theta=\theta_0}. \nonumber
\end{eqnarray}
Again we can easily see that the second term is zero, this time as it is independent of $\theta$. Recalling that the Hamiltonian depends linearly on the angle $\theta$, we can easily consider the derivative of $H_t^\theta$. The first term can then be calculated as
\begin{eqnarray}
    \int_0^1 \frac{d}{d\theta} \big(H^\theta_t+h^\theta_t\big) \Big|_{\theta_0} \!\!\! \circ \tau_{\theta_0}(t) \;dt = \frac{1}{\theta_0} \int_0^1 H^{\theta_0}_t \circ \tau_{\theta_0}(t) \;dt + \int_0^1 \frac{d}{d\theta} \big(h^\theta_t\big) \Big|_{\theta_0} \!\!\! \circ \tau_{\theta_0}(t)  \;dt. \nonumber
\end{eqnarray}
We recall that $\tau_{\theta_0}$ is a trajectory of the Hamiltonian $H^{\theta_0}_t + h^{\theta_0}_t$ and therefore remains a distance of at least $\operatorname{Width_{r, \theta_0}}/2$ away from the diagonal. Moreover, as $h_t^{\theta_0} \rightarrow 0$ the trajectory $\tau_{\theta_0}$ limits pointwise to a trajectory $\widetilde{\tau}$ of $H_t^{\theta_0}$. Defining $\varepsilon(h_t^{\theta_0}) = \sup \{ | \tau_{\theta_0}(t) - \widetilde{\tau}(t)| : t \in[0,1] \}$ we have
\begin{equation}
    \frac{1}{8}r(1-r)\theta_0 \cdot \operatorname{Width_{r, \theta_0}}(\gamma)^2 \leq \int_0^1 H^{\theta_0}_t \circ \tau_{\theta_0}(t) \;dt \leq \frac{1}{2}r(1-r)\theta_0 \cdot \big(2 \operatorname{Rad}(\gamma) + \varepsilon(h_t) \big)^2. \nonumber
\end{equation}
Note that $\varepsilon(h_t^{\theta_0}) \rightarrow 0$ as $h_t^{\theta_0} \rightarrow 0$. We also note that the second integral can be made arbitrarily small by taking the path of Hamiltonian perturbations $h_t^\theta$ arbitrarily small. Therefore, considering the limit as $h_t^\theta \rightarrow 0$, we obtain
\begin{equation}
    0 < \frac{d}{d\theta}l(\gamma, r, \theta) \leq 2r(1-r) \operatorname{Rad}(\gamma)^2. \nonumber
\end{equation}
The fundamental theorem of calculus then gives the desired result.
\end{proof}

\subsubsection{Triangle inequality for the spectral invariants}

Another important property of spectral invariants is that they satisfy a triangle inequality, which can be found in \cite{Leclercq-Zapolsky:2015} and was previously used in \cite{Greene-Lobb:Positive-measure} in the rectangle setting. In our setting, the triangle equality of relevance is purely for the spectral invariant in homological degree 2 and it plays an important role in the obstruction of shrinkout.

\begin{lemma} \label{lemma:spectral_triangle_inequality}
    Let $\gamma$ be a smooth Jordan curve, $r \in (0,1/2]$ an aspect ratio, and $\theta_1, \theta_2 \in [0, \pi]$ angles such that $\theta_1 + \theta_2 \leq \pi$, then we have the triangle inequality for the homological degree 2 spectral invariant
    \begin{equation}
        l_2(\gamma,r, \theta_1 + \theta_2) \leq l_2(\gamma,r,\theta_1) + l_2(\gamma,r,\theta_2). \nonumber
    \end{equation}
\end{lemma}

\begin{proof}
    This follows immediately from the triangle inequality in \cite[Theorem 27]{Leclercq-Zapolsky:2015} once we establish that the product on Floer homology $\star:\operatorname{HF}_k(L; H_1, J_1) \otimes \operatorname{HF}_l(L;H_2,J_2) \rightarrow \operatorname{HF}_{k+l-n}(L;H,J)$ defined in \cite[Section 3.9.1]{Zapolsky:2015} satisfies $\alpha \star \alpha = \alpha$ in the Jordan Floer homology setting where $\alpha \in \operatorname{JF}_2(\gamma,r,\theta)$ is the generator of the second homology group. This product respects continuation maps. Therefore, it is sufficient to establish the result in the Morse-Bott homology limit as $\theta \rightarrow 0$. We recall that this limit is simply the relative homology of a torus $\gamma \times \gamma$ to its diagonal curve $\Delta(\gamma)$, for which $\alpha$ is the relative fundamental class and therefore acts as a unit, giving the required result.
\end{proof}

\section{Inscriptions in non-smooth Jordan curves}

We now wish to consider inscriptions of isosceles trapezoids in non-smooth Jordan curves. As first introduced in the discussion of the issue of shrinkout (Section \ref{sec:shrinkout}), the approach will be to approximate a non-smooth Jordan curve $\gamma$ by a sequence of smooth Jordan curves $\gamma_n$. Convergent sequences of inscribed isosceles trapezoids $T_n$ in the smooth Jordan curves $\gamma_n$ limit to some set of points $T \subset \gamma$. However, the limiting set $T$ could be a single point, that is, the inscriptions $T_n$ \emph{shrinkout}. The aim is therefore to obstruct shrinkout.

We will restrict ourselves to considering \emph{rectifiable} and \emph{locally monotone} Jordan curves. Recall that rectifiable Jordan curves are those of finite length, and locally monotone Jordan curves can be represented locally as the graph of a continuous function. Such Jordan curves are easier to deal with since they can be approximated by smooth Jordan curves with nice properties, making it easier to obstruct shrinkout.

\subsection{Approximating Curves} \label{sec:approximating_curves}

We introduced the standard approach to approximating non-smooth Jordan curves in Section \ref{sec:shrinkout}. This is the approach we will use when approximating rectifiable Jordan curves. Let $\gamma \subset \mathbb{C}$ be a Jordan curve and $D$ be the bounded region of $\mathbb{C} \setminus \gamma$. By the Riemann mapping theorem, there exists a biholomorphic map $f:\mathbb{D} \rightarrow D$, where $\mathbb{D} \subset \mathbb{C}$ is the open unit disc. By Carathéodory’s extension to the Riemann mapping theorem, this extends continuously to a homeomorphism $f: \overline{\mathbb{D}} \rightarrow \overline{D}$.

We can then define smooth Jordan curves $\gamma_t = f(C_t)$, where $C_t = \{ z \in \mathbb{C} \mid |z| = t \}$ is the circle centred at the origin of radius $t$. The set of curves $\{ \gamma_t \mid t \in (0,1) \}$ defines an isotopy of smooth Jordan curves whose limit is the Jordan curve $\gamma = \gamma_1$. These Jordan curves have a natural parametrisation given by $\gamma_t(\theta) = f(te^{i\theta})$. An important property of such approximating Jordan curves is given by the following classical theorem.

\begin{lemma}[Riesz-Privalov Theorem, \cite{Pommerenke:1992}] \label{lemma:rectifiable_approx}
    Let $\gamma$ be a rectifiable Jordan curve of length $L$ and let $\gamma_t$ be defined as above, then $\gamma_t$ are of uniformly bounded length.
\end{lemma}

When approximating locally monotone Jordan curves (recall Definition \ref{def:locally_monotone}), we wish to preserve the local monotonicity constant, which we now define. Recall that for a locally monotone Jordan curve $\gamma$ parameterised by $\gamma: \mathbb{R}/2\pi\mathbb{Z} \rightarrow \mathbb{R}^2$, for each point $\theta \in \mathbb{R}$ there exists a unit vector $v_\theta \in \mathbb{R}^2$ and a neighbourhood $U_\theta = (\theta_1, \theta_2)$ of $\theta$ such that $g_{v_\theta}:s \mapsto \gamma(s) \cdot v_p$ is strictly monotone. There are many such choices for the pair $(v_\theta, U_\theta)$, but we let \smash{$(\widetilde{v}_\theta, \widetilde{U}_\theta)$} with \smash{$\widetilde{U}_\theta = (\widetilde{\theta}_1, \widetilde{\theta}_2)$} denote a pair that maximises $\mu(\theta, v_\theta, U_\theta) = \min\{|g_{v_\theta}(\theta)-g_{v_\theta}(\theta_1)|, |g_{v_\theta}(\theta)-g_{v_\theta}(\theta_2)| \}$. 
\begin{definition}
    Let $\gamma$ be a locally monotone Jordan curve, $\gamma:S^1 \cong \mathbb{R}/2\pi \mathbb{Z} \rightarrow \mathbb{C} \cong \mathbb{R}^2$ be a parametrisation of $\gamma$ and \smash{$(\widetilde{v}_\theta, \widetilde{U}_\theta)$} be defined as above. Denoting \smash{$\mu(\theta) \coloneq \mu(\theta, \widetilde{v}_\theta, \widetilde{U}_\theta)$}, we define the \emph{local monotonicity constant} of $\gamma$ as $\mu \coloneq \min\{\mu(\theta) \mid \theta \in \mathbb{R}\}$. 
\end{definition}
\noindent For the approximating Jordan curves defined by the previous construction, it is not clear that this should be the case. We therefore take a different approach when approximating locally monotone Jordan curves, the construction of which was used by Asano and Ike in \cite[Chapter 5.3]{Asano-Ike:2026} and we now describe.

Let $\gamma$ be a locally monotone Jordan curve and let $\gamma(\cdot):\mathbb{R}/\mathbb{Z} \rightarrow \mathbb{\mathbb{C}}$ be a parametrisation, then for every $\theta \in \mathbb{R}$ there exists a direction $v_p \in \mathbb{R}^2$ and a neighbourhood $U_\theta = (\theta_1, \theta_2) \subset \mathbb{R}$ such that $s \mapsto \gamma(s) \cdot v_\theta$ is monotone on $U_\theta$. We therefore define a parameterisation of $\gamma_\varepsilon$ by 
\begin{equation}
    \gamma_\varepsilon(s) = \int_\mathbb{R} \frac{1}{\varepsilon} \varphi\left(\frac{u}{\varepsilon}\right) \gamma(s-u) \; du, \nonumber
\end{equation}
where $\varphi:\mathbb{R} \rightarrow \mathbb{R}_{\geq 0}$ is a smooth function supported in $[-1,1]$ and $\int_\mathbb{R} \varphi = 1$. The curves $\gamma_\varepsilon$ are, therefore, smooth closed curves that converge pointwise to $\gamma$ as $\varepsilon \rightarrow 0$. 

To see that $\gamma_\varepsilon$ is a smooth Jordan curve, it remains to check that it is simple, that is, $\gamma(\cdot):S^1 \rightarrow \mathbb{C}$ is injective. We first introduce $U_\theta^\varepsilon \subset U_\theta$ defined by $U_\theta^\varepsilon = (\theta_1+\varepsilon, \theta_2 - \varepsilon)$ and then note that $s \mapsto \gamma_\varepsilon(s) \cdot v_\theta$ is locally monotone for all $\theta \in \mathbb{R}$ (see Lemma \ref{lemma:locally_monotone_approx} below). Therefore, $\gamma(U_\theta^\varepsilon)$ is a simple curve for every $\theta \in \mathbb{R}$ and if $\gamma_\varepsilon(s_1) = \gamma_\varepsilon(s_2)$ there must be no $\theta$ such that $s_1, s_2 \in U_\theta^\varepsilon$. By compactness, the lengths of $U_\theta$ are bounded below by some $l > 0$, so we must have $|s_1-s_2| > l-\varepsilon$. On the other hand, again by compactness, we have $d = \min \{|\gamma(s_1)-\gamma(s_2)| : |s_1-s_2| \geq l/2 \} > 0$. Therefore, as $\gamma_\varepsilon \rightarrow \gamma$ pointwise, for sufficiently small $\varepsilon$ we see that $\gamma_\varepsilon$ is simple and hence a smooth Jordan curve.

\begin{lemma} \label{lemma:locally_monotone_approx}
    Let $\gamma$ be a locally monotone Jordan curve of local monotonicity constant $\mu$ and let $\gamma_\varepsilon$ be defined as above, then letting $\mu_t$ be the local monotonicity constant of $\gamma_t(\theta)$, we have $\liminf_{\varepsilon \rightarrow 0} \mu_t \geq \mu$.
\end{lemma}

\begin{proof}
    Let $v_\theta$ be a unit vector and $U_\theta$ be a neighbourhood of $\theta$ such that $g_{v_\theta}(s) = \gamma(s) \cdot v_\theta$ is strictly monotone. It is sufficient to show that the map $g_{v_\theta}^\varepsilon(s) = \gamma_\varepsilon(s) \cdot v_\theta$ is strictly monotone on $U_\theta^\varepsilon$, which is immediate since $g_{v_\theta}^\varepsilon$ is simply a mollification of $g$.

\end{proof}

We recall that for locally $K$-Lipschitz-graphical curves (see Definition \ref{def:graphically_Lipschitz}) for every $\theta \in \mathbb{R}$ there exists an open neighbourhood $U_\theta \subset \mathbb{R}$ of $\theta$ and a unit vector $v_\theta \in \mathbb{R}^2$ such that $f_{v_\theta} \circ g_{v_\theta}^{-1}$ is $K$-Lipschitz on $g_{v_\theta}(U_\theta)$, where $f_{v_\theta}:s \mapsto \gamma(s) \cdot n_\theta$ where $n_\theta \in \mathbb{R}^2$ is a unit normal to $v_\theta$. However, we note that we do not necessarily have $(g(\theta)-\mu, g(\theta)+\mu) \subseteq g_{v_\theta}(U_\theta)$ as the definition of $\mu$ only requires $f_{v_\theta} \circ g_{v_\theta}^{-1}$ to be a continuous function and not Lipschitz. Therefore, we must introduce the analogue of the local monotonicity constant for locally Lipschitz-graphical curves. 

As before, we note that there are many choices of the pair $(v_\theta, U_\theta)$ such that $f_{v_\theta} \circ g_{v_\theta}^{-1}$ is $K$-Lipschitz on $g_{v_\theta}(U_\theta)$. We let \smash{$(\widetilde{v}_\theta^K, \widetilde{U}_\theta^K)$} denote a pair that maximises $\mu(\theta, v_\theta, U_\theta)$. We can then define the local $K$-Lipschitz-graphical constant as follows.

\begin{definition}
    Let $\gamma$ be a locally Lipschitz-graphical Jordan curve, $\gamma:S^1 \cong \mathbb{R}/2\pi \mathbb{Z} \rightarrow \mathbb{C} \cong \mathbb{R}^2$ be a parametrisation of $\gamma$ and \smash{$(\widetilde{v}_\theta^K, \widetilde{U}_\theta^K)$} be defined as above. Denoting \smash{$\mu_K(\theta) \coloneq \mu(\theta, \widetilde{v}_\theta^K, \widetilde{U}_\theta^K)$}, we define the \emph{local $K$-Lipschitz-graphical constant} of $\gamma$ as $\mu_K \coloneq \min\{\mu_K(\theta) \mid \theta \in \mathbb{R}\}$. 
\end{definition}

We now wish to show that the local $K$-Lipschitz graphical constants of the approximating Jordan curve satisfy the same property as the local monotonicity constants.

\begin{lemma} \label{lemma:lipschitz_approx}
    Let $\gamma$ be a locally $K$-Lipschitz-graphical Jordan curve of local $K$-Lipschitz-graphical constant $\mu_K$, then for sufficiently small $\varepsilon$ we have $\gamma_\varepsilon$ is a locally $K$-Lipschitz-graphical Jordan curve. Moreover, letting $\mu_{K,t}$ be the local $K$-Lipschitz-graphical constant of $\gamma_t$, we have $\liminf_{\varepsilon \rightarrow 0} \mu_{K,t} \geq \mu_K$. 
\end{lemma}

\begin{proof}
    Let $v_\theta$ be a unit vector and $U_\theta$ be a neighbourhood of $\theta$ such that $f_{v_\theta} \circ g_{v_\theta}^{-1}$ is $K$-Lipschitz on $g_{v_\theta}(U_\theta)$. Then for any $\theta_1, \theta_2 \in U_\theta$ we have $| \gamma(\theta_1)\cdot n_\theta-\gamma(\theta_2) \cdot n_\theta | \leq L |\gamma(\theta_1) \cdot v_\theta - \gamma(\theta_2) \cdot v_\theta|$. By the proof of Lemma \ref{lemma:locally_monotone_approx}, the function $s \mapsto \gamma_\varepsilon(s) \cdot v_\theta$ is strictly monotone on $U_\theta^\varepsilon$, so it is sufficient to show that the same inequality holds for $\gamma_\varepsilon$ for any $\theta_1, \theta_2 \in U_\theta^\varepsilon$. We have
    \begin{eqnarray}
        | \gamma_\varepsilon(\theta_1) \cdot n_\theta - \gamma_\varepsilon(\theta_2) \cdot n_\theta| &\leq& \int_\mathbb{R} \frac{1}{\varepsilon} \varphi\left(\frac{u}{\varepsilon}\right) \cdot \Big|\big( \gamma(\theta_1-u) - \gamma(\theta_2-u) \big) \cdot n_\theta \Big| \; du \nonumber \\
        &\leq& K \int_\mathbb{R} \frac{1}{\varepsilon} \varphi\left(\frac{u}{\varepsilon}\right) \cdot \Big|\bigl( \gamma(\theta_1-u) - \gamma(\theta_2-u) \bigr) \cdot v_\theta \Big| \; du \nonumber \\
        &=& K \cdot \big|\gamma_\varepsilon(\theta_1) \cdot v_\theta - \gamma_\varepsilon(\theta_2) \cdot v_\theta \big| \nonumber,
    \end{eqnarray}
    where for the first inequality, we used the definition of $\gamma_\varepsilon$, for the second inequality we used the Lipschitz assumption above, and for the final equality we used the fact that $g_{v_\theta}$ is monotone on $U_\theta$. 
\end{proof}

\subsection{Shrinkout obstruction criterion}

For rectifiable and locally monotone Jordan curves, the limiting action of any sequence of inscribed isosceles trapezoids $T_n \subset \gamma_n$ that shrinks to a point $p \in \gamma$ can only take certain values. To prove this claim, we follow the same approach as \cite[Lemma 5.1]{Greene-Lobb:Floer-homology}, generalised to include the possibility of locally monotone Jordan curves, and the resulting lemma provides a way to obstruct shrinkout.

\begin{lemma} \label{lemma:shrinkout}
    Let $\gamma$ be a rectifiable or locally monotone Jordan curve and let $\gamma_n$ be defined as $\gamma_t$ for $t = 1-\tfrac{1}{n}$ from Section \ref{sec:approximating_curves}. Let $T_n \subset \gamma_n$ be a sequence of non-degenerate isosceles trapezoids of fixed aspect ratio $r$ and angle $\theta$ limiting to a point $p \in \gamma$, then working in $\mathbb{R} / \operatorname{Area}(\gamma) \mathbb{Z}$ we have
    \begin{equation}
        \lim_{n \rightarrow \infty} \mathcal{A}_{r, \theta}(T_n)  \in \{0, r \cdot \operatorname{Area}(\gamma), (1-r) \cdot \operatorname{Area}(\gamma) \}. \nonumber
    \end{equation}
\end{lemma}

\begin{proof}
    Each $\gamma_n \setminus T_n$ consists of four components $\gamma_n^1$, $\gamma_n^2$, $\gamma_n^3$ and $\gamma_n^4$. Since $\gamma$ is simple, we may assume (without loss of generality) that $\gamma_n^2, \gamma_n^3, \gamma_n^4 \rightarrow p$ and $\gamma_n^1 \rightarrow \gamma$ in the Hausdorff metric. We may then choose discs $B_n = B(p, r_n)$ centred at $p$ of radius $r_n$ such that $\gamma_n \setminus \gamma_n^1 \subset B_n$ and $r_n \rightarrow 0$. We then consider the action
    \begin{equation}
        \mathcal{A}_{r, \theta}(T_n) \coloneqq \mathcal{A}_{H_t}(\tau_n) = \int_0^1 H_t \circ \tau_n(t) \; dt - \int_{[0,1]^2} \widehat{\tau}_n^* \, \omega, \nonumber
    \end{equation}
    where $\tau_n$ is the trajectory of $H_t$ corresponding to the inscribed isosceles trapezoid $T_n$.
    
    To compute the first term of this action, we recall that $H_t = \beta(t)H$ with $\int \beta(t) \,dt = 1$, and $H$ is constant along trajectories of $H_t$. The first integral then satisfies
    \begin{equation}
        \int_0^1 H_t \circ \tau_n(t) \; dt =  \frac{1}{2} \theta r(1-r) \cdot  \operatorname{Diag(T_n)} \rightarrow 0 \nonumber
    \end{equation}
    as $n \rightarrow \infty$ and can therefore be ignored.

    To compute the second term of this action, we would typically require a preferred capping $\widehat{\tau}_n$ of $\tau_n$, or, by Stokes' theorem, simply its boundary $\tau_n \cup p_n$ for some path $p_n \subset (\gamma_n \times \gamma_n) \setminus \Delta(\gamma_n)$ between $\tau_n(0)$ and $\tau_n(1)$ such that $\tau_n \cup p_n$ has zero winding number around $\Delta(\mathbb{C})$. However, since we only require the action in $\mathbb{R} / {\operatorname{Area}(\gamma)\mathbb{Z}}$, we can drop the zero winding number requirement. To see that this is sufficient, note that if we considered a path $\tau_n \cup p_n'$ with winding number $k \neq 0$ around $\Delta(\mathbb{C})$, one could obtain a path $p$ with zero winding number around $\Delta(\mathbb{C})$ by concatenating $p_n'$ with a loop $w_n \subset (\gamma_n \times \gamma_n) \setminus \Delta(\gamma_n)$ of winding number $-k$ around $\Delta(\mathbb{C})$. Note that $w_n$ bounds a disc of area $-k \cdot \operatorname{Area}(\gamma_n)$. Letting $p_n = p_n' \# w_n$ be the result of this concatenation, then in $\mathbb{R} / \operatorname{Area}(\gamma)\mathbb{Z}$, we have
    \begin{equation}
        \int_{[0,1]^2} \widehat{\tau}_n^* \, \omega = \int_{\tau_n \cup p} \eta = \int_{\tau_n \cup p'} \eta, \nonumber
    \end{equation}
    where $\eta = (1-r) \cdot x_1 dy_1 + r \cdot x_2 dy_2$ so that $d \eta = \omega$. To calculate the action in $\mathbb{R} / \operatorname{Area}(\gamma)\mathbb{Z}$, it is therefore sufficient to choose any path $p_n' \subset (\gamma_n \times \gamma_n) \setminus \Delta(\gamma_n)$ between $\tau_n(0)$ and $\tau_n(1)$.
    
    Let $\tau_n(0)=(z,w)$ and $\tau_n(1)=(z',w')$, then we define injective paths $p_n^1, p_n^2:[0,1] \rightarrow \gamma_n$ with $p_n^1(0) = z$, $p_n^1(1) = z'$ and $p_n^2(0) = w$, $p_n^2(1) = w'$. We further require that $p_n^1$ and $p_n^2$ both run clockwise around $\gamma_n$, then by careful parametrisation we can ensure $p_n^1(t) \neq p_n^2(t)$ for all $t \in [0,1]$. The path $p_n' = (p_n^1,p_n^2)$ is then a diagonal avoiding path as required. We can then calculate
    \begin{equation}
        \int_{\tau_n \cup p_n'} \eta = (1-r) \int_{\pi_1(\tau_n) \cup p_n^1} x \, dy + r \int_{\pi_2(\tau_n) \cup p_n^2} x \, dy, \nonumber
    \end{equation}
    which is the weighted sum of the signed areas bounded by $\pi_1(\tau_n) \cup p_n^1$ and $\pi_2(\tau_n) \cup p_n^2$ respectively. Denote these unweighted areas by $A_n^1$ and $A_n^2$ respectively.
    
    Note that the curves $p_n^j$ consist of unions of the closures of the curves $\gamma_n^i$ and $\pi_j(\tau_n) \subset B_n$. If $\gamma_n^1 \cap p_n^j = \emptyset$, then $\pi_j(\tau_n) \cup p_n^j \subset B_n$. On the other hand, if $\gamma_n^1 \subset p_n^j$, we have $\pi_j(\tau_n) \cup (\gamma \setminus p_n^j) \subset B_n$ and bounds an area of $A_n^j + \operatorname{Area}(\gamma_n)$. If $\gamma$ is rectifiable, these are closed curves of uniformly bounded length by Lemma \ref{lemma:rectifiable_approx}, and if $\gamma$ is locally monotone, these are the union of the graph of a function and an arc of a circle by Lemma \ref{lemma:locally_monotone_approx}. In either case, since they are contained within shrinking disks $B_n$, the area they bound becomes arbitrarily small as $n \rightarrow \infty$, which means that we have $A_n^j \rightarrow 0$ or $A_n^j + \operatorname{Area}(\gamma_n) \rightarrow 0$. We therefore have
    \begin{equation}
        \int_{\tau_n \cup p_n'} \eta = (1-r) \cdot A_n^1 + r \cdot A_n^2 \longrightarrow - \mathbf{1}_1 \cdot (1-r)\operatorname{Area}(\gamma) - \mathbf{1}_2 \cdot r\operatorname{Area}(\gamma) \nonumber
    \end{equation}
    as $n \rightarrow \infty$, where $\mathbf{1}_j \in \{0,1\}$ is the indicator function taking the value 1 if $\gamma_n^1 \subset p_n^j$ for all sufficiently large $n$ and 0 otherwise. This gives the desired result.
\end{proof}

To obstruct shrinkout, it therefore suffices to show that the limit of the action $A_{r, \theta}(T_n)$ for a sequence of inscribed isosceles trapezoids $T_n \subset \gamma_n$ does not take one of the values $0$, $r \cdot \operatorname{Area(\gamma)}$, or $(1-r) \cdot \operatorname{Area}(\gamma)$ in $\mathbb{R}/\operatorname{Area}(\gamma)\mathbb{Z}$. What makes this more challenging in the isosceles trapezoid case is that for a fixed smooth Jordan curve $\gamma_n$, the spectral invariant $l_i(\gamma_n, r, \theta)$ could potentially remain arbitrarily small as $\theta \rightarrow \pi$ as it could represent the action of a quadrisecant of arbitrarily small action.

\subsection{Proof of the main results}

We now have all the machinery in place to prove the main results. These results are derived from the following theorem, which uses the same approach as the proof of \cite[Theorem A]{Greene-Lobb:Floer-homology}. However, some modifications are required due to the possibility of quadrisecants with small actions representing the spectral invariant in the $\theta \rightarrow \pi$ limit. We begin by restating and proving Theorem \ref{thm:intro_main_theorem}.

\theoremB* 

\begin{proof}
    Fix an aspect ratio $r$ and approximate $\gamma$ by a sequence of smooth Jordan curves $\gamma_n$, where $\gamma_n$ is $\gamma_\varepsilon$ for $\varepsilon=1/n$ from Section \ref{sec:approximating_curves} if $\gamma$ is locally monotone and $\gamma_t$ for $t=1-1/n$ from Section \ref{sec:approximating_curves} otherwise. If $\gamma$ is locally monotone of constant $\mu$, by Lemma \ref{lemma:locally_monotone_approx} each curve $\gamma_n(\theta)$ is also locally monotone of constant $\mu_n$ with $\liminf_{n \rightarrow \infty} \mu_n \geq \mu$. Otherwise, by Lemma \ref{lemma:rectifiable_approx} the curves $\gamma_n$ are uniformly bounded in length. We then rescale each $\gamma_n$ to ensure that Area($\gamma_n$) = Area($\gamma$). The rescaled curves still have a uniformly bounded length $L$ or are locally $\mu_n$-monotone, and converge in $C^0$ to $\gamma$.

    We then consider the spectral invariant of each of the smooth approximating curves 
    \begin{equation}
        l_2(\gamma_n, r, \cdot):[0,\pi] \longrightarrow \mathbb{R}_{\geq 0}. \nonumber
    \end{equation}
    By Lemma \ref{lemma:spectral_triangle_inequality}, we have $l_2(\gamma_n,r,\pi/2^k) \geq l_2(\gamma_n,r,\pi)/2^k$ for all $n,k \in \mathbb{N}$. Therefore, letting $\delta_k = \pi/2^k$ and $\varepsilon_{n,k} = l_2(\gamma_n,r, \pi)/2^k$, for all $\theta \geq \delta_k$ we have $l_2(\gamma_n,r,\theta) \geq \varepsilon_{n,k}$. On the other hand, by Proposition \ref{prop:bounded_spectral_variation}, we have for each spectral invariant
    \begin{equation}
        0 < l_2(\gamma_n,r,\theta_2) - l_2(\gamma_n,r,\theta_1) \leq 2r(1-r) \operatorname{Rad}(\gamma_n)^2 \cdot (\theta_2 - \theta_1) \nonumber
    \end{equation}
    for all $0 \leq \theta_1 < \theta_2 \leq \pi$. By taking $\theta_1 = 0$ and $\theta_2 = \theta$  we immediately obtain that for all
    \begin{equation}
        \theta \leq \frac{r \cdot \operatorname{Area}(\gamma) - \varepsilon_{n,k}}{2r(1-r) \cdot \operatorname{Rad}(\gamma_n)^2} \nonumber
    \end{equation}
    we have $l_2(\gamma_n,r,\theta) \leq r \cdot \operatorname{Area}(\gamma) - \varepsilon_{n,k}$. The spectral invariants $l_2(\gamma_n,r,\theta)$ are therefore bounded  by $\varepsilon_{n,k} \leq l_2(\gamma_n,r, \theta) \leq r \cdot \operatorname{Area}(\gamma) - \varepsilon_{n,k}$ for all
    \begin{equation}
        \theta \in I_{n,k} =\left( \delta_k,\frac{r \cdot \operatorname{Area}(\gamma) - \varepsilon_{n,k}}{2r(1-r) \cdot \operatorname{Rad}(\gamma_n)^2} \right) \nonumber
    \end{equation}
    for any $n,k \in \mathbb{N}$. Since $l(\gamma,r, \pi) \neq 0$, we have $\varepsilon_k = \lim_{n \rightarrow \infty} \varepsilon_{n,k}> 0$ and we can invoke Lemma \ref{lemma:shrinkout} to conclude that $\gamma$ inscribes a (non-degenerate) isosceles trapezoid of aspect ratio $r$ and angle $\theta$ for every $\theta \in I_k = \lim_{n \rightarrow \infty} I_{n,k}$. Taking $k \rightarrow \infty$ then gives the desired result.
\end{proof}

To establish subclasses of rectifiable or locally monotone Jordan curves that inscribe isosceles trapezoids of sufficiently small angle, it therefore suffices to show that $l(\gamma,r,\pi) \neq 0$ for all curves in that class. We now restate Theorem \ref{thm:first_theorem}, which gives a class of curves for which we can show $l(\gamma,r,\pi) \neq 0$ and therefore follows as a corollary of Theorem \ref{thm:intro_main_theorem}.

\theoremA*

\begin{proof}
    By Theorem \ref{thm:intro_main_theorem}, it is sufficient to prove that $l(\gamma, r ,\pi) \neq 0$. For $r = 1/2$ this was established in \cite{Greene-Lobb:Floer-homology}. We therefore assume $r \neq 1/2$ and suppose for a contradiction that $l(\gamma, r ,\pi) = 0$.
    
    Since $\gamma$ is locally Lipschitz-graphical, there exist constants $K, \mu_K > 0$ such that $\gamma$ is locally $K$-Lipschitz-graphical of locally $K$-Lipschitz-graphical constant $\mu_K$. Let $\gamma_n$ be the curve $\gamma_\varepsilon$ defined in Section \ref{sec:approximating_curves} for $\varepsilon = 1/n$. We aim to show that for some $\varepsilon > 0$, all inscribed isosceles trapezoids in $\gamma_n$ of angle $\theta > \pi - \varepsilon$ are either elegantly inscribed, almost-elegantly inscribed, or have diagonal length bounded below by some constant $d$ for sufficiently large $n$.
    
    To see that this is sufficient, recall that $l(\gamma_n, r, \theta) = \mathcal{A}_{r, \theta}(T_n^\theta)$ for some inscribed $T_n^\theta$ in $\gamma_n$. Suppose that such $T_n^\theta$ are not uniformly bounded below in diagonal length, then there exists a sequence of inscribed isosceles trapezoids $T_k$ of angle $\theta_k$ in $\gamma_{n_k}$ which shrink to a point as $k \rightarrow \infty$. We may further assume that $\theta_n \rightarrow \pi$ and $n_k \rightarrow \infty$ as otherwise we could simply choose a smaller $\varepsilon$ or larger $N$. By the assumptions above, for sufficiently large $n$ all isosceles trapezoids $T_n^{\theta_n}$ must be elegantly or almost-elegantly inscribed. Following the proof of Lemma \ref{lemma:shrinkout}, allowing for a variable angle $\theta_n$ and using the cappings from Sections \ref{sec:elegant_action} and \ref{sec:almost_elegant_action}, we find that $\mathcal{A}_{r, \theta_n}(T_n^{\theta_n}) \rightarrow r \cdot \operatorname{Area}(\gamma)$ if the $T_n^{\theta_n}$ are elegantly inscribed and $\mathcal{A}_{r, \theta_n}(T_n^{\theta_n}) \rightarrow (1-r) \cdot \operatorname{Area}(\gamma)$ if the $T_n^{\theta_n}$ are almost-elegantly inscribed. This contradicts the assumption that $l(\gamma, r, \pi) = 0$.
    
    We can therefore assume that the inscribed isosceles trapezoids $T_n^\theta$ have diagonal length uniformly bounded below by some constant $d$. By the proof of Proposition \ref{prop:bounded_spectral_variation} we then have
    \begin{equation}
        l(\gamma_n,r,\pi) = \lim_{\theta \rightarrow \pi} l(\gamma_n,r,\theta) \geq \int_{\pi-\varepsilon}^\pi \frac{d}{d\theta}l(\gamma_n,r,\theta) \;d\theta \geq \frac{1}{8}r(1-r)d \cdot \varepsilon > 0, \nonumber
    \end{equation}
    giving a uniform positive lower bound on $l(\gamma_n,r,\pi)$ in $n$. This once again contradicts $l(\gamma, r, \pi) = 0$ and therefore proves the claim.

    It remains to establish that all inscribed isosceles trapezoids of sufficiently large angle are of the claimed form. We note that by Lemma \ref{lemma:lipschitz_approx} we have that $\gamma_n$ is locally $K$-Lipschitz graphical and denoting its locally $K$-Lipschitz graphical constant by $\mu_K^n$ we have $\liminf_{n \rightarrow \infty} \mu_K^n \geq \mu_K$. The Jordan curve $\gamma_n$ can therefore be locally represented as the graph of a Lipschitz function of Lipschitz constant $L$ over a region of length at least $\mu_K^n$ on each side of $\gamma_n(p)$. An elementary geometric calculation shows that if 
    \begin{equation}
        (1-2r) \tan \left( \frac{\theta}{2} \right) > L, \nonumber
    \end{equation}
    any inscribed isosceles trapezoid of angle $\theta$ and aspect ratio $r$ on such a local graph is necessarily elegantly or almost-elegantly inscribed in the wider curve $\gamma_n$. Taking $\varepsilon < 2 \arctan\bigl((1-2r)/L\bigr)$, we see that every inscribed isosceles trapezoid in $\gamma_n$ of angle $\theta > \pi - \varepsilon$ and diagonal length less than $\mu_K^n$ is elegantly or almost-elegantly inscribed, proving the claim.
\end{proof}

\bibliographystyle{amsplain}
\bibliography{bibliography}

\end{document}